%% file: svharxiv.tex
\documentclass[11pt]{article}
\usepackage[textwidth=6.0in,textheight=9.0in,centering]{geometry}
\usepackage{tikz}
\usepackage{wrapfig}
\usepackage{titlesec,titletoc}
\usepackage{etex}
\usepackage{amsxtra,amscd}
\usepackage{graphicx}
\usepackage[amsmath,hyperref,thmmarks]{ntheorem}
\usepackage{natbib}
\usepackage{verbatim}
\usepackage{mathptmx}
\usepackage[scaled=0.92]{helvet}
\usepackage{amsmath}
\usepackage{enumitem}
\usepackage{amsfonts}
\usepackage{amssymb}
\usepackage{url}
\usepackage{natbib}
\theoremnumbering{arabic}
\theoremheaderfont{\scshape}
\RequirePackage{latexsym}
\theorembodyfont{\slshape}
\theoremseparator{.}
\newtheorem{X}{X}[section]

\newtheorem{conjecture}[X]{Conjecture}
\newtheorem{corollary}[X]{Corollary}

\newtheorem{lemma}[X]{Lemma}
\newtheorem{proposition}[X]{Proposition}
\newtheorem{summary}[X]{Summary}
\newtheorem{theorem}[X]{Theorem}
\theorembodyfont{\upshape}
\newtheorem{definition}[X]{Definition}
\newtheorem{example}[X]{Example}

\newtheorem{plain}[X]{}

\newtheorem{question}[X]{Question}
\newtheorem{fundamentalquestion}[X]{Fundamental Question}
\newtheorem{remark}[X]{Remark}
\theorembodyfont{\small}
\newtheorem{aside}[X]{Aside}
\newtheorem*{nt}{Notes}
\theorembodyfont{\normalsize}
\theoremstyle{nonumberplain}
\theoremsymbol{\ensuremath{_\Box}}
\newtheorem{proof}{Proof}

\qedsymbol{\ensuremath{_\Box}}
\input{defna}

\begin{document}

\title{Shimura Varieties and Moduli}
\date{April 30, 2011, v2.00}
\author{J.S. Milne}
\maketitle

\begin{abstract}
Connected Shimura varieties are the quotients of hermitian symmetric domains
by discrete groups defined by congruence conditions. We examine their relation
with moduli varieties.

\end{abstract}
\tableofcontents

\renewcommand{\theequation}{\arabic{equation}}

\section*{Introduction}

The hermitian symmetric domains are the complex manifolds isomorphic to
bounded symmetric domains. The Griffiths period domains are the parameter
spaces for polarized rational Hodge structures. A period domain is a hermitian
symmetric domain if the universal family of Hodge structures on it is a
variation of Hodge structures, i.e., satisfies Griffiths transversality. This
rarely happens, but, as Deligne showed, every hermitian symmetric domain can
be realized as the subdomain of a period domain on which certain tensors for
the universal family are of type $(p,p)$ (i.e., are Hodge tensors).

In particular, every hermitian symmetric domain can be realized as a moduli
space for Hodge structures plus tensors. This all takes place in the analytic
realm, because hermitian symmetric domains are not algebraic varieties. To
obtain an algebraic variety, we must pass to the quotient by an arithmetic
group. In fact, in order to obtain a moduli variety, we should assume that the
arithmetic group is defined by congruence conditions. The algebraic varieties
obtained in this way are the connected Shimura varieties.

The arithmetic subgroup lives in a semisimple algebraic group over
$\mathbb{Q}{}$, and the variations of Hodge structures on the connected
Shimura variety are classified in terms of auxiliary reductive algebraic
groups${}$. In order to realize the connected Shimura variety as a moduli
variety, we must choose the additional data so that the variation of Hodge
structures is of geometric origin. The main result of the article classifies
the connected Shimura varieties for which this is known to be possible.
Briefly, in a small number of cases, the connected Shimura variety is a moduli
variety for abelian varieties with polarization, endomorphism, and level
structure (the PEL case); for a much larger class, the variety is a moduli
variety for abelian varieties with polarization, Hodge class, and level
structure (the PHL case); for all connected Shimura varieties except those of
type $E_{6}$, $E_{7}$, and certain types $D$, the variety is a moduli variety
for abelian \textit{motives} with additional structure. In the remaining
cases, the connected Shimura variety is not a moduli variety for abelian
motives, and it is not known whether it is a moduli variety at all.\medskip

We now summarize the contents of the article.

\S 1. As an introduction to the general theory, we review the case of elliptic
modular curves. In particular, we prove that the modular curve constructed
analytically coincides with the modular curve constructed algebraically using
geometric invariant theory.

\S 2. We briefly review the theory of hermitian symmetric domains. To give a
hermitian symmetric domain amounts to giving a real semisimple Lie group $H$
with trivial centre and a homomorphism $u$ from the circle group to $H$
satisfying certain conditions. This leads to a classification of hermitian
symmetric domains in terms of Dynkin diagrams and special nodes.

\S 3. The group of holomorphic automorphisms of a hermitian symmetric domain
is a real Lie group, and we are interested in quotients of the domain by
certain discrete subgroups of this Lie group. In this section we review the
fundamental theorems of Borel, Harish-Chandra, Margulis, Mostow, Selberg,
Tamagawa, and others concerning discrete subgroups of Lie groups.

\S 4. The arithmetic locally symmetric varieties (resp. connected Shimura
varieties) are the quotients of hermitian symmetric domains by arithmetic
(resp. congruence) groups. We explain the fundamental theorems of Baily and
Borel on the algebraicity of these varieties and of the maps into them.

\S 5. We review the definition of Hodge structures and of their variations,
and state the fundamental theorem of Griffiths that motivated their definition.

\S 6. We define the Mumford-Tate group of a rational Hodge structure, and we
prove the basic results concerning their behaviour in families.

\S 7. We review the theory of period domains, and explain Deligne's
interpretation of hermitian symmetric domains as period subdomains.

\S 8. We classify certain variations of Hodge structures on locally symmetric
varieties in terms of group-theoretic data.

\S 9. {\small I}n order to be able to realize all but a handful of locally
symmetric varieties as moduli varieties, we shall need to replace algebraic
varieties and algebraic classes by more general objects. In this section, we
prove Deligne's theorem that all Hodge classes on abelian varieties are
absolutely Hodge, and have algebraic meaning, and we define abelian motives.

\S 10. Following Satake and Deligne, we classify the symplectic embeddings of
an algebraic group that give rise to an embedding of the associated hermitian
symmetric domain into a Siegel upper half space.

\S 11. We use the results of the preceding sections to determine which Shimura
varieties can be realized as moduli varieties for abelian varieties (or
abelian motives) plus additional structure.

Although the expert will find little that is new in this article, there is
much that is not well explained in the literature. As far as possible,
complete proofs have been included.

\subsection{Notations}

We use $k$ to denote the base field (always of characteristic zero), and
$k^{\mathrm{al}}$ to denote an algebraic closure of $k$. \textquotedblleft
Algebraic group\textquotedblright\ means \textquotedblleft affine algebraic
group scheme\textquotedblright\ and \textquotedblleft algebraic
variety\textquotedblright\ means \textquotedblleft geometrically reduced
scheme of finite type over a field\textquotedblright.%
\index{algebraic group}%
\index{algebraic variety}
For a smooth algebraic variety $X$ over $\mathbb{C}{}$, we let $X^{\text{an}}$
denote the set $X(\mathbb{C}{})$ endowed with its natural structure of a
complex manifold. The tangent space at a point $p$ of space $X$ is denoted by
$T_{p}(X)$.

Vector spaces and representations are finite dimensional unless indicated
otherwise. The linear dual of a vector space $V$ is denoted by $V^{\vee}$. For
a $k$-vector space $V$ and commutative $k$-algebra $R$, $V_{R}=R\otimes_{k}V$.
For a topological space $S$, we let $V_{S}$ denote the constant local system
of vector spaces on $S$ defined by $V$. By a lattice in a real vector space,
we mean a full lattice, i.e., the $\mathbb{Z}{}$-module generated by a basis
for the vector space.

A \textit{vector sheaf}%
\index{vector sheaf}%
on a complex manifold (or scheme) $S$ is a locally free sheaf of
$\mathcal{O}{}_{S}$-modules of finite rank. In order for $\mathcal{W}{}$ to be
a vector subsheaf of a vector sheaf $\mathcal{V}{}$, we require that the maps
on the fibres $\mathcal{W}{}_{s}\rightarrow\mathcal{V}{}_{s}$ be injective.
With these definitions, vector sheaves correspond to vector bundles and vector
subsheaves to vector subbundles.

The quotient of a Lie group or algebraic group $G$ by its centre $Z(G)$ is
denoted by $G^{\mathrm{ad}}$.%
\index{Gad@$G^{\mathrm{ad}}$}
A Lie group or algebraic group is said to be \textit{adjoint}%
\index{algebraic group!adjoint}
if it is semisimple (in particular, connected) with trivial centre. An
algebraic group is \textit{simple} (resp. \textit{almost simple}) if it
connected noncommutative and every proper normal subgroup is trivial (resp.
finite). An \textit{isogeny} of algebraic groups is a surjective homomorphism
with finite kernel. An algebraic group $G$ is \textit{simply connected}%
\index{algebraic group!simply connected}
if it is semisimple and every isogeny $G^{\prime}\rightarrow G$ with
$G^{\prime}$ connected is an isomorphism. The inner automorphism of $G$
defined by an element $g$ is denoted by $\inn(g)$.%
\index{inn(g)@$\inn(g)$}
Let $\ad\colon G\rightarrow G^{\mathrm{ad}}$ be the quotient map.%
\index{ad@$\ad$}
There is an action of $G^{\mathrm{ad}}$ on $G$ such that $\ad(g)$ acts as
$\inn(g)$ for all $g\in G(k^{\mathrm{al}})$. For an algebraic group $G$ over
$\mathbb{R}{}$, $G(\mathbb{R}{})^{+}$ is the identity component of
$G(\mathbb{R}{})$ for the real topology. For a finite extension of fields
$L/k$ and an algebraic group $G$ over $L$, we write $(G)_{L/k}$ for algebraic
group over $k$ obtained by (Weil) restriction of scalars.%
\index{restriction of scalars}%
\index{Weil restriction}%
\index{GLk@$(G)_{L/k}$}
As usual, $\mathbb{G}_{m}=\GL_{1}$ and $\mu_{N}$ is the kernel of
$\mathbb{G}_{m}\overset{N}{\longrightarrow}\mathbb{G}_{m}$.

A \textit{prime}%
\index{prime}
of a number field $k$ is a prime ideal in $\mathcal{O}{}_{k}$ (a finite
prime), an embedding of $k$ into $\mathbb{R}{}$ (a real prime), or a conjugate
pair of embeddings of $k$ into $\mathbb{C}{}$ (a complex prime). The ring of
finite ad\`{e}les of $\mathbb{Q}{}$ is $\mathbb{A}{}_{f}=\mathbb{Q}{}%
\otimes\left(  \prod\nolimits_{p}\mathbb{Z}{}_{p}\right)  $.

We use $\iota$ or $z\mapsto\bar{z}$ to denote complex conjugation on
$\mathbb{C}{}$ or on a subfield of $\mathbb{C}{}$, and we use $X\simeq Y$ to
mean that $X$ and $Y$ isomorphic with a specific isomorphism --- which
isomorphism should always be clear from the context.

For algebraic groups we use the language of modern algebraic geometry, not the
more usual language, which is based on Weil's Foundations. For example, if $G$
and $G^{\prime}$ are algebraic groups over a field $k$, then by a homomorphism
$G\rightarrow G^{\prime}$ we mean a homomorphism defined over $k$, not over
some universal domain. Similarly, a simple algebraic group over a field $k$
need not be geometrically (i.e., absolutely) simple.

\section{Elliptic modular curves}

\begin{quote}
{\small The first Shimura varieties, and the first moduli varieties, were the
elliptic modular curves. In this section, we review the theory of elliptic
modular curves as an introduction to the general theory.}
\end{quote}

\subsection{Definition of elliptic modular curves}

Let $D$ be the complex upper half plane,%
\[
D=\{z\in\mathbb{C}{}\mid\Im(z)>0\}.
\]
The group $\SL_{2}(\mathbb{R}{})$ acts transitively on $D$ by the rule%
\[%
\begin{pmatrix}
a & b\\
c & d
\end{pmatrix}
z=\frac{az+b}{cz+d}.
\]
A subgroup $\Gamma$ of $\SL_{2}(\mathbb{Z}{})$ is a congruence subgroup%
\index{congruence subgroup}
if, for some integer $N\geq1$, $\Gamma$ contains the principal congruence
subgroup%
\index{principal congruence subgroup}
of level $N$,%
\[
\Gamma(N)\overset{\textup{{\tiny def}}}{=}\left\{  A\in\SL_{2}(\mathbb{Z}%
{})\mid A\equiv I\bmod N\right\}  \text{.}%
\]
An elliptic modular curve%
\index{elliptic modular curve}
is the quotient $\Gamma\backslash D$ of $D$ by a congruence group $\Gamma$.
Initially this is a one-dimensional complex manifold, but it can be
compactified by adding a finite number of \textquotedblleft
cusps\textquotedblright, and so it has a unique structure of an algebraic
curve compatible with its structure as a complex manifold.\footnote{We are
using that the functor $S\rightsquigarrow S^{\text{an}}$ from smooth algebraic
varieties over $\mathbb{C}{}$ to complex manifolds defines an equivalence from
the category of \textit{complete} smooth algebraic curves to that of
\textit{compact} Riemann surfaces.} This curve can be realized as a moduli
variety for elliptic curves with level structure, from which it is possible
deduce many beautiful properties of the curve, for example, that it has a
canonical model over a specific number field, and that the coordinates of the
special points on the model generate class fields.

\subsection{Elliptic modular curves as moduli varieties}

For an elliptic curve $E$ over $\mathbb{C}{}$, the exponential map defines an
exact sequence%
\begin{equation}
0\rightarrow\Lambda\rightarrow T_{0}(E^{\text{an}})\overset{\exp
}{\longrightarrow}E^{\text{an}}\rightarrow0 \label{hq49}%
\end{equation}
with%
\[
\Lambda\simeq\pi_{1}(E^{\text{an}},0)\simeq H_{1}(E^{\text{an}},\mathbb{Z}%
{}).
\]
The functor $E\rightsquigarrow(T_{0}E,\Lambda)$ is an equivalence from the
category of complex elliptic curves to the category of pairs consisting of a
one-dimensional $\mathbb{C}{}$-vector space and a lattice. Thus, to give an
elliptic curve over $\mathbb{C}{}$ amounts to giving a two-dimensional
$\mathbb{R}{}$-vector space $V$, a complex structure on $V$, and a lattice in
$V$. It is known that $D$ parametrizes elliptic curves plus additional data.
Traditionally, to a point $\tau$ of $D$ one attaches the quotient of
$\mathbb{C}{}$ by the lattice spanned by $1$ and $\tau$. In other words, one
fixes the real vector space and the complex structure, and varies the lattice.
From the point of view of period domains and Shimura varieties, it is more
natural to fix the real vector space and the lattice, and vary the complex
structure.\footnote{The choice of a trivialization of a variation of integral
Hodge structures attaches to each point of the underlying space a fixed real
vector space and lattice, but a varying Hodge structure --- see below.}

Thus, let $V$ be a two-dimensional vector space over $\mathbb{R}{}$. A complex
structure%
\index{complex structure}
on $V$ is an endomorphism $J$ of $V$ such that $J^{2}=-1$. From such a $J$, we
get a decomposition $V_{\mathbb{C}{}}=V_{J}^{+}\oplus V_{J}^{-}$ of
$V_{\mathbb{C}{}}$ into its $+i$ and $-i$ eigenspaces, and the isomorphism
$V\rightarrow V_{\mathbb{C}{}}/V_{J}^{-}$ carries the complex structure $J$ on
$V$ to the natural complex structure on $V_{\mathbb{C}{}}/V_{J}^{-}$. The map
$J\mapsto V_{\mathbb{C}{}}/V_{J}^{-}$ identifies the set of complex structures
on $V$ with the set of nonreal one-dimensional quotients of $V_{\mathbb{C}{}}%
$, i.e., with $\mathbb{P}{}(V_{\mathbb{C}{}})\smallsetminus\mathbb{P}{}(V)$.
This space has two connected components.

Now choose a basis for $V$, and identify it with $\mathbb{R}{}^{2}$. Let
$\psi\colon V\times V\rightarrow\mathbb{R}{}$ be the alternating form%
\[
\psi(\left(  \!%
\begin{smallmatrix}
a\\
b
\end{smallmatrix}
\!\right)  ,\left(  \!%
\begin{smallmatrix}
c\\
d
\end{smallmatrix}
\!\right)  )=\det\left(  \!%
\begin{smallmatrix}
a & c\\
b & d
\end{smallmatrix}
\!\right)  =ad-bc.
\]
On one of the connected components, which we denote $D$, the symmetric
bilinear form
\[
(x,y)\mapsto\psi_{J}(x,y)\overset{\textup{{\tiny def}}}{=}\psi(x,Jy)\colon
V\times V\rightarrow\mathbb{R}{}%
\]
is positive definite and on the other it is negative definite. Thus $D$ is the
set of complex structures on $V$ for which $+\psi$ (rather than $-\psi$) is a
Riemann form. Our choice of a basis for $V$ identifies $\mathbb{P}%
{}(V_{\mathbb{C}{}})\smallsetminus\mathbb{P}{}(V)$ with $\mathbb{P}{}%
^{1}(\mathbb{C}{})\smallsetminus\mathbb{P}{}^{1}(\mathbb{R}{})$ and $D$ with
the complex upper half plane.

Now let $\Lambda$ be the lattice $\mathbb{Z}{}^{2}$ in $V$. For each $J\in D$,
the quotient $(V,J)/\Lambda$ is an elliptic curve $E$ with $H_{1}%
(E^{\text{an}},\mathbb{Z}{})\simeq\Lambda$. In this way, we obtain a
one-to-one correspondence between the points of $D$ and the isomorphism
classes of pairs consisting of an elliptic curve $E$ over $\mathbb{C}{}$ and
an ordered basis for $H_{1}(E^{\text{an}},\mathbb{Z}{})$.

Let $E_{N}$ denote the kernel of multiplication by $N$ on an elliptic curve
$E$. Thus, for the curve $E=(V,J)/\Lambda$,
\[
E_{N}(\mathbb{C}{})=\tstyle\frac{1}{N}\Lambda/\Lambda\simeq\Lambda
/N\Lambda\approx(\mathbb{Z}{}/N\mathbb{Z}{})^{2}.
\]
A level-$N$ structure on $E$ is a pair of points $\eta=(t_{1},t_{2})$ in
$E(\mathbb{C}{})$ that forms an ordered basis for $E_{N}(\mathbb{C}{})$.

For an elliptic curve $E$ over any field, there is an algebraically defined
(Weil) pairing
\[
e_{N}\colon E_{N}\times E_{N}\rightarrow\mu_{N}.
\]
When the ground field is $\mathbb{C}{}$, this induces an isomorphism
$\bigwedge^{2}\left(  E_{N}(\mathbb{C}{})\right)  \simeq\mu_{N}(\mathbb{C}{}%
)$. In the following, we fix a primitive $N$th root $\zeta$ of $1$ in
$\mathbb{C}{}$, and we require that our level-$N$ structures satisfy the
condition $e_{N}(t_{1},t_{2})=\zeta$.

Identify $\Gamma(N)$ with the subgroup of $\SL(V)$ whose elements preserve
$\Lambda$ and act as the identity on $\Lambda/N\Lambda$. On passing to the
quotient by $\Gamma(N)$, we obtain a one-to-one correspondence between the
points of $\Gamma(N)\backslash D$ and the isomorphism classes of pairs
consisting of an elliptic curve $E$ over $\mathbb{C}{}$ and a level-$N$
structure $\eta$ on $E$. Let $Y_{N}$ denote the algebraic curve over
$\mathbb{C}{}$ with $Y_{N}^{\text{an}}=\Gamma(N)\backslash D$.

Let $f\colon E\rightarrow S$ be a family of elliptic curves over a scheme $S$,
i.e., a flat map of schemes together with a section whose fibres are elliptic
curves. A level-$N$ structure on $E/S$ is an ordered pair of sections to $f$
that give a level-$N$ structure on $E_{s}$ for each closed point $s$ of $S$.

\begin{proposition}
\label{h11}Let $f\colon E\rightarrow S$ be a family of elliptic curves on a
smooth algebraic curve $S$ over $\mathbb{C}{}$, and let $\eta$ be a level-$N$
structure on $E/S$. The map $\gamma\colon S(\mathbb{C}{})\rightarrow
Y_{N}(\mathbb{C}{})$ sending $s\in S(\mathbb{C}{})$ to the point of
$\Gamma(N)\backslash D$ corresponding to $(E_{s},\eta_{s})$ is regular%
\index{regular map}%
, i.e., defined by a morphism of algebraic curves.
\end{proposition}

\begin{proof}
We first show that $\gamma$ is holomorphic. For this, we use that
$\mathbb{P}{}(V_{\mathbb{C}{}})$ is the Grassmann manifold classifying the
one-dimensional quotients of $V_{\mathbb{C}{}}$. This means that, for any
complex manifold $M$ and surjective homomorphism $\alpha\colon\mathcal{O}%
{}_{M}\otimes_{\mathbb{R}{}}V\rightarrow\mathcal{W}{}$ of vector sheaves on
$M$ with $\mathcal{W}{}$ of rank $1$, the map sending $m\in M$ to the point of
$\mathbb{P}{}(V_{\mathbb{C}{}})$ corresponding to the quotient $\alpha
_{m}\colon V_{\mathbb{C}{}}\rightarrow\mathcal{W}_{m}$ of $V_{\mathbb{C}{}}$
is holomorphic.

Let $f\colon E\rightarrow S$ be a family of elliptic curves on a connected
smooth algebraic variety $S$. The exponential map defines an exact sequence of
sheaves on $S^{\text{an}}$%
\[
0\longrightarrow R_{1}f_{\ast}\mathbb{Z}\longrightarrow{}\mathcal{T}{}%
_{0}(E^{\text{an}}/S^{\text{an}})\longrightarrow E^{\text{an}}\longrightarrow
0
\]
whose fibre at a point $s\in S^{\text{an}}$ is the sequence (\ref{hq49}) for
$E_{s}$. From the first map in the sequence we get a surjective map%
\begin{equation}
\mathcal{O}{}_{S^{\text{an}}}\otimes_{\mathbb{Z}{}}R_{1}f_{\ast}%
\mathbb{Z}\twoheadrightarrow\mathcal{T}{}_{0}(E^{\text{an}}/S^{\text{an}}).
\label{hq48}%
\end{equation}
Let $(t_{1},t_{2})$ be a level-$N$ structure on $E/S$. Each point of
$S^{\text{an}}$ has an open neighbourhood $U$ such that $t_{1}|_{U}$and
$t_{2}|_{U}$ lift to sections $\tilde{t}_{1}$ and $\tilde{t}_{2}$ of
$\mathcal{T}{}_{0}(E^{\text{an}}/S^{\text{an}})$ over $U$; now $N\tilde{t}%
_{1}$ and $N\tilde{t}_{2}$ are sections of $R_{1}f_{\ast}\mathbb{Z}{}$ over
$U$, and they define an isomorphism
\[
\mathbb{Z}_{U}^{2}{}\rightarrow R_{1}f_{\ast}\mathbb{Z}{}|_{U}\text{.}%
\]
On tensoring this with $\mathcal{O}{}_{U^{\text{an}}}$,%
\[
\mathcal{O}{}_{U^{\text{an}}}\otimes_{\mathbb{Z}{}}\mathbb{Z}_{U}^{2}%
{}\rightarrow\mathcal{O}_{U^{\text{an}}}{}\otimes R_{1}f_{\ast}\mathbb{Z}%
{}|_{U}%
\]
and composing with (\ref{hq48}), we get a surjective map%
\[
\mathcal{O}_{U^{\text{an}}}{}\otimes_{\mathbb{R}{}}{}V\twoheadrightarrow
\mathcal{T}{}_{0}(E^{\text{an}}/S^{\text{an}})|U
\]
of vector sheaves on $U$, which defines a holomorphic map $U\rightarrow
\mathbb{P}{}(V_{\mathbb{C}{}})$. This maps into $D$, and its composite with
the quotient map $D\rightarrow\Gamma(N)\backslash D$ is the map $\gamma$.
Therefore $\gamma$ is holomorphic.

It remains to show that $\gamma$ is algebraic. We now assume that $S$ is a
curve. After passing to a finite covering, we may suppose that $N$ is even.
Let $\bar{Y}_{N}$ (resp. $\bar{S}$) be the completion of $Y_{N}$ (resp. $S$)
to a smooth complete algebraic curve. We have a holomorphic map%
\[
S^{\text{an}}\overset{\gamma}{\longrightarrow}Y_{N}^{\text{an}}\subset\bar
{Y}_{N}^{\text{an}}\text{;}%
\]
to show that it is regular, it suffices to show that it extends to a
holomorphic map of compact Riemann surfaces $\bar{S}^{\text{an}}%
\rightarrow\bar{Y}_{N}^{\text{an}}$. The curve $Y_{2}$ is isomorphic to
$\mathbb{P}{}^{1}\smallsetminus\{0,1,\infty\}$. The composed map%
\[
S^{\text{an}}\overset{\gamma}{\longrightarrow}Y_{N}^{\text{an}}\overset
{\text{onto}}{\longrightarrow}Y_{2}^{\text{an}}\approx\mathbb{P}{}%
^{1}(\mathbb{C}{})\smallsetminus\{0,1,\infty\}
\]
does not have an essential singularity at any of the (finitely many) points of
$\bar{S}^{\text{an}}\smallsetminus S^{\text{an}}$ because this would violate
the big Picard theorem.\footnote{Recall that this says that a holomorphic
function on the punctured disk with an essential singularity at $0$ omits at
most one value in $\mathbb{C}{}$. Therefore a function on the punctured disk
that omits two values has (at worst) a pole at $0$, and so extends to a
function from the whole disk to $\mathbb{P}{}^{1}(\mathbb{C}{})$.} Therefore,
it extends to a holomorphic map $\bar{S}^{\text{an}}\rightarrow\mathbb{P}%
{}^{1}(\mathbb{C}{})$, which implies that $\gamma$ extends to a holomorphic
map $\bar{\gamma}\colon\bar{S}^{\text{an}}\rightarrow\bar{Y}_{N}^{\text{an}}$,
as required.
\end{proof}

Let $\mathcal{F}$ be the functor sending a scheme $S$ of finite type over
$\mathbb{C}{}$ to the set of isomorphism classes of pairs consisting of a
family elliptic curves $f\colon E\rightarrow S$ over $S$ and a level-$N$
structure $\eta$ on $E$. According to Mumford 1965, Chapter
7,\nocite{mumford1965} the functor $\mathcal{F}{}$ is representable when
$N\geq3$. More precisely, when $N\geq3$ there exists a smooth algebraic curve
$S_{N}$ over $\mathbb{C}{}$ and a family of elliptic curves over $S_{N}$
endowed with a level $N$ structure that is universal in the sense that any
similar pair on a scheme $S$ is isomorphic to the pullback of the universal
pair by a unique morphism $\alpha\colon S\rightarrow S_{N}$.

\begin{theorem}
\label{h14}There is a canonical isomorphism $\gamma\colon S_{N}\rightarrow
Y_{N}$.
\end{theorem}

\begin{proof}
According to Proposition \ref{h11}, the universal family of elliptic curves
with level-$N$ structure on $S_{N}$ defines a morphism of smooth algebraic
curves $\gamma\colon S_{N}\rightarrow Y_{N}$. Both sets $S_{N}(\mathbb{C}{})$
and $Y_{N}(\mathbb{C}{})$ are in natural one-to-one correspondence with the
set of isomorphism classes of complex elliptic curves with level-$N$
structure, and $\gamma$ sends the point in $S_{N}(\mathbb{C}{})$ corresponding
to a pair $(E,\eta)$ to the point in $Y_{N}(\mathbb{C}{})$ corresponding to
the same pair. Therefore, $\gamma(\mathbb{C}{})$ is bijective, which implies
that $\gamma$ is an isomorphism.
\end{proof}

In particular, we have shown that the curve $S_{N}$, constructed by Mumford
purely in terms of algebraic geometry, is isomorphic by the obvious map to the
curve $Y_{N}$, constructed analytically. Of course, this is well known, but it
is difficult to find a proof of it in the literature. For example, Brian
Conrad has noted that it is used without reference in \cite{katzM1985}.

Theorem \ref{h14} says that there exists a single algebraic curve over
$\mathbb{C}{}$ enjoying the good properties of both $S_{N}$ and $Y_{N}$.

\section{Hermitian symmetric domains}

\begin{quote}
{\small The natural generalization of the complex upper half plane is a
hermitian symmetric domain.}
\end{quote}

\subsection{Preliminaries on Cartan involutions and polarizations}

Let $G$ be a connected algebraic group over $\mathbb{R}{}$, and let
$\sigma_{0}\colon g\mapsto\bar{g}$ denote complex conjugation on
$G_{\mathbb{C}{}}$ with respect to $G$. A \emph{Cartan involution }%
\index{Cartan involution}
of $G$ is an involution $\theta$ of $G$ (as an algebraic group over
$\mathbb{R}{}$) such that the group%
\index{Gt@$G^{(\theta)}$}%
\[
G^{(\theta)}(\mathbb{R}{})=\{g\in G(\mathbb{C}{})\mid g=\theta(\bar{g})\}
\]
is compact. Then $G^{(\theta)}$ is a compact real form of $G_{\mathbb{C}}$,
and $\theta$ acts on $G(\mathbb{C})$ as $\sigma_{0}\sigma=\sigma\sigma_{0}$
where $\sigma$ denotes complex conjugation on $G_{\mathbb{C}}$ with respect to
$G^{(\theta)}$.

Consider, for example, the algebraic group $\GL_{V}$ attached to a real vector
space $V$. The choice of a basis for $V$ determines a transpose operator
$g\mapsto g^{t}$, and $\theta\colon g\mapsto(g^{t})^{-1}$ is a Cartan
involution of $\GL_{V}$ because $\GL_{V}^{(\theta)}(\mathbb{R}{})$ is the
unitary group. The basis determines an isomorphism $\GL_{V}\simeq\GL_{n}$, and
$\sigma_{0}(A)=\bar{A}$ and $\sigma(A)=(\bar{A}^{t})^{-1}$ for $A\in
\GL_{n}(\mathbb{C}{})$.

A connected algebraic group $G$ has a Cartan involution if and only if it has
a compact real form, which is the case if and only if $G$ is reductive. Any
two Cartan involutions of $G$ are conjugate by an element of $G(\mathbb{R}{}%
)$. In particular, all Cartan involutions of $\GL_{V}$ arise, as in the last
paragraph, from the choice of a basis for $V$. An algebraic subgroup $G$ of
$\GL_{V}$ is reductive if and only if it is stable under $g\mapsto g^{t}$ for
some basis of $V$, in which case the restriction of $g\mapsto(g^{t})^{-1}$ to
$G$ is a Cartan involution. Every Cartan involution of $G$ is of this form.
See \cite{satake1980}, I, \S 4.

Let $C$ be an element of $G(\mathbb{R}{})$ whose square is central (so
$\inn(C)$ is an involution). A $C$-\emph{polarization}%
\index{C-polarization@$C$-polarization}
on a real representation $V$ of $G$ is a $G$-invariant bilinear form
$\varphi\colon V\times V\rightarrow\mathbb{R}{}$ such that the form
$\varphi_{C}\colon(x,y)\mapsto\varphi(x,Cy)$ is symmetric and positive definite.

\begin{theorem}
\label{h20b}If $\inn(C)$ is a Cartan involution of $G$, then every finite
dimensional real representation of $G$ carries a $C$-polarization; conversely,
if one \textup{faithful} finite dimensional real representation of $G$ carries
a $C$-polarization, then $\inn(C)$ is a Cartan involution.
\end{theorem}

\begin{proof}
An $\mathbb{R}{}$-bilinear form $\varphi$ on a real vector space $V$ defines a
sesquilinear form $\varphi^{\prime}\colon(u,v)\mapsto\varphi_{\mathbb{C}{}%
}(u,\bar{v})$ on $V(\mathbb{C}{})$, and $\varphi^{\prime}$ is hermitian (and
positive definite) if and only if $\varphi$ is symmetric (and positive definite).

Let $G\rightarrow\GL_{V}$ be a representation of $G$. If $\inn(C)$ is a Cartan
involution of $G$, then $G^{(\inn C)}(\mathbb{R}{})$ is compact, and so there
exists a $G^{(\inn C)}$-invariant positive definite symmetric bilinear form
$\varphi$ on $V$. Then $\varphi_{\mathbb{C}{}}$ is $G(\mathbb{C}{}%
)$-invariant, and so%
\[
\varphi^{\prime}(gu,(\sigma g)v)=\varphi^{\prime}(u,v),\quad\text{for all
}g\in G(\mathbb{C}{})\text{, }u,v\in V_{\mathbb{C}{}},
\]
where $\sigma$ is the complex conjugation on $G_{\mathbb{C}{}}$ with respect
to $G^{(\inn C)}$. Now $\sigma g=\inn(C)(\bar{g})=\inn(C^{-1})(\bar{g})$, and
so, on replacing $v$ with $C^{-1}v$ in the equality, we find that%
\[
\varphi^{\prime}(gu,(C^{-1}\bar{g}C)C^{-1}v)=\varphi^{\prime}(u,C^{-1}%
v),\quad\text{for all }g\in G(\mathbb{C}{})\text{, }u,v\in V_{\mathbb{C}{}}.
\]
In particular, $\varphi(gu,C^{-1}gv)=\varphi(u,C^{-1}v)$ when $g\in
G(\mathbb{R}{})$ and $u,v\in V$. Therefore, $\varphi_{C^{-1}}$ is
$G$-invariant. As $(\varphi_{C^{-1}})_{C}=\varphi$, we see that $\varphi$ is a
$C$-polarization.

For the converse, one shows that, if $\varphi$ is a $C$-polarization on a
faithful representation, then $\varphi_{C}$ is invariant under $G^{(\inn
C)}(\mathbb{R}{})$, which is therefore compact.
\end{proof}

\begin{plain}
\label{h20c}\textsc{Variant.} Let $G$ be an algebraic group over $\mathbb{Q}%
{}$, and let $C$ be an element of $G(\mathbb{R}{})$ whose square is central. A
$C$-\emph{polarization} on a $\mathbb{Q}{}$-representation $V$ of $G$ is a
$G$-invariant bilinear form $\varphi\colon V\times V\rightarrow\mathbb{Q}{}$
such that $\varphi_{\mathbb{R}{}}$ is a $C$-polarization on $V_{\mathbb{R}{}}%
$. In order to show that a $\mathbb{Q}{}$-representation $V$ of $G$ is
polarizable, it suffices to check that $V_{\mathbb{R}{}}$ is polarizable. We
prove this when $C^{2}$ acts as $+1$ or $-1$ on $V$, which are the only cases
we shall need. Let $P(\mathbb{Q}{})$ (resp. $P(\mathbb{R}{})$) denote the
space of $G$-invariant bilinear forms on $V$ (resp. on $V_{\mathbb{R}{}}$)
that are symmetric when $C^{2}$ acts as $+1$ or skew-symmetric when it acts as
$-1$. Then $P(\mathbb{R}{})=\mathbb{R}{}\otimes_{\mathbb{Q}{}}P(\mathbb{Q}{}%
)$. The $C$-polarizations of $V_{\mathbb{R}{}}$ form an open subset of
$P(\mathbb{R}{})$, whose intersection with $P(\mathbb{Q}{})$ consists of the
$C$-polarizations of $V$.
\end{plain}

\subsection{Definition of hermitian symmetric domains}

Let $M$ be a complex manifold, and let $J_{p}\colon T_{p}M\rightarrow T_{p}M$
denote the action of $i=\sqrt{-1}$ on the tangent space at a point $p$ of $M$.
A \emph{hermitian metric}%
\index{hermitian metric}
on $M$ is a riemannian metric $g$ on the underlying smooth manifold of $M$
such that $J_{p}$ is an isometry for all $p$.\footnote{Then $g_{p}$ is the
real part of a unique hermitian form on the complex vector space $T_{p}M$,
which explains the name.} A \emph{hermitian manifold}%
\index{hermitian manifold}
is a complex manifold equipped with a hermitian metric $g$, and a
\emph{hermitian symmetric space}%
\index{hermitian symmetric space}
is a connected hermitian manifold $M$ that admits a symmetry at each point
$p$, i.e., an involution $s_{p}$ having $p$ as an isolated fixed point. The
group $\mathrm{Hol}(M)$%
\index{HolM@$\mathrm{Hol}(M)$}
of holomorphic automorphisms of a hermitian symmetric space $M$ is a real Lie
group whose identity component $\mathrm{Hol}(M)^{+}$ acts transitively on $M$.

Every hermitian symmetric space $M$ is a product of hermitian symmetric spaces
of the following types:

\begin{itemize}
\item Noncompact type --- the curvature is negative\footnote{This means that
the sectional curvature $K(p,E)$ is $<0$ for every $p\in M$ and every
two-dimensional subspace $E$ of $T_{p}M$.} and $\mathrm{Hol}(M)^{+}$ is a
noncompact adjoint Lie group; example, the complex upper half plane.

\item Compact type --- the curvature is positive and $\mathrm{Hol}(M)^{+}$ is
a compact adjoint Lie group; example, the Riemann sphere.

\item Euclidean type --- the curvature is zero; $M$ is isomorphic to a
quotient of a space $\mathbb{C}{}^{n}$ by a discrete group of translations.
\end{itemize}

\noindent In the first two cases, the space is simply connected. A hermitian
symmetric space is \emph{indecomposable}%
\index{hermitian symmetric space!indecomposable}
if it is not a product of two hermitian symmetric spaces of lower dimension.
For an indecomposable hermitian symmetric space $M$ of compact or noncompact
type, the Lie group $\mathrm{Hol}(M)^{+}$ is simple. See \cite{helgason1978},
Chapter VIII.

A \emph{hermitian symmetric domain}%
\index{hermitian symmetric domain}
is a connected complex manifold that admits a hermitian metric for which it is
a hermitian symmetric space of noncompact type.\footnote{Usually a hermitian
symmetric domain is defined to be a complex manifold \textit{equipped} with a
hermitian metric etc.. However, a hermitian symmetric domain in our sense
satisfies conditions (A.1) and (A.2) of \cite{kobayashi1959}, and so has a
canonical Bergman metric, invariant under all holomorphic automorphisms.} The
hermitian symmetric domains are exactly the complex manifolds isomorphic to
bounded symmetric domains (via the Harish-Chandra embedding; \cite{satake1980}%
, II \S 4). Thus a connected complex manifold $M$ is a hermitian symmetric
domain if and only if

\begin{enumerate}
\item it is isomorphic to a bounded open subset of $\mathbb{C}{}^{n}$ for some
$n$, and

\item for each point $p$ of $M$, there exists a holomorphic involution of $M$
(the \emph{symmetry} at $p$)%
\index{symmetry}
having $p$ as an isolated fixed point.
\end{enumerate}

For example, the bounded domain $\{z\in\mathbb{C}{}\mid\left\vert z\right\vert
<1\}$ is a hermitian symmetric domain because it is homogeneous and admits a
symmetry at the origin ($z\mapsto-1/z$). The map $z\mapsto\frac{z-i}{z+i}$ is
an isomorphism from the complex upper half plane $D$ onto the open unit disk,
and so $D$ is also a hermitian symmetric domain. Its automorphism group is%
\[
\mathrm{Hol}(D)\simeq\SL_{2}(\mathbb{R}{})/\{\pm I\}\simeq\PGL_{2}%
(\mathbb{R}{})^{+}\text{.}%
\]

\subsection{Classification in terms of real groups}

\begin{plain}
\label{h20p}Let $U^{1}$%
\index{U1@$U^{1}$}
be the circle group, $U^{1}=\{z\in\mathbb{C}{}\mid\left\vert z\right\vert
=1\}$. For each point $o$ of a hermitian symmetric domain $D$, there is a
unique homomorphism $u_{o}\colon U^{1}\rightarrow\mathrm{Hol}(D)$ such that
$u_{o}(z)$ fixes $o$ and acts on $T_{o}D$ as multiplication by $z$ ($z\in
U^{1}$).\footnote{See, for example, \cite{milneSVI}, Theorem 1.9.} In
particular, $u_{o}(-1)$ is the symmetry at $o$.
\end{plain}

\begin{example}
\label{h20}Let $D$ be the complex upper half plane and let $o=i$. Let $h\colon
U^{1}\rightarrow\SL_{2}(\mathbb{R}{})$ be the homomorphism $a+bi\mapsto\left(
\begin{smallmatrix}
\hfill a & b\\
-b & a
\end{smallmatrix}
\right)  $. Then $h(z)$ fixes $o$, and it acts as $z^{2}$ on $T_{o}(D)$. For
$z\in U^{1}$, choose a square root $\sqrt{z}$ in $U^{1}$, and let
$u_{o}(z)=h(\sqrt{z})$ mod $\pm I$. Then $u_{o}(z)$ is independent of the
choice of $\sqrt{z}$ because $h(-1)=-I$. The homomorphism $u_{o}\colon
U^{1}\rightarrow\SL_{2}(\mathbb{Z}{})/\left\{  \pm I\right\}  =\mathrm{Hol}%
(D)$ has the correct properties.
\end{example}

Now let $D$ be a hermitian symmetric domain. Because $\mathrm{Hol}(D)$ is an
adjoint Lie group, there is a unique real algebraic group $H$ such that
$H(\mathbb{R}{})^{+}=\mathrm{Hol}(D)^{+}$. Similarly, $U^{1}$ is the group of
$\mathbb{R}{}$-points of the algebraic torus $\mathbb{S}{}^{1}$
\index{S1@$\mathbb{S}{}^{1}$}%
defined by the equation $X^{2}+Y^{2}=1$. A point $o\in D$ defines a
homomorphism $u\colon\mathbb{S}{}^{1}\rightarrow H$ of real algebraic groups.

\begin{theorem}
\label{h21}The homomorphism $u\colon\mathbb{S}{}^{1}\rightarrow H$ has the
following properties:

\begin{description}
\item[\textmd{SU1:}] only the characters $z,1,z^{-1}$ occur in the
representation of $\mathbb{S}^{1}$ on $\Lie(H)_{\mathbb{C}}$ defined
by\nolinebreak\ $u$;\footnote{The maps $\mathbb{S}^{1}\overset{u}%
{\longrightarrow}H_{\mathbb{R}{}}\overset{\Ad}{\longrightarrow}\Aut(\Lie(H))$
define an action of $\mathbb{S}{}^{1}$ on $\Lie(H)$, and hence on
$\Lie(H)_{\mathbb{C}{}}$. The condition means that $\Lie(H)_{\mathbb{C}{}}$ is
a direct sum of subspaces on which $u(z)$ acts as $z$, $1$, or $z^{-1}$.}%
\index{SU (conditions)}%

\item[\textmd{SU2:}] $\inn(u(-1))$ is a Cartan involution.
\end{description}

\noindent Conversely, if $H$ is a real adjoint algebraic group with no compact
factor and $u\colon\mathbb{S}{}^{1}\rightarrow H$ satisfies the conditions
(SU1,2), then the set $D$ of conjugates of $u$ by elements of $H(\mathbb{R}%
{})^{+}$ has a natural structure of a hermitian symmetric domain for which
$u(z)$ acts on $T_{u}D$ as multiplication by $z$; moreover, $H(\mathbb{R}%
{})^{+}=\mathrm{Hol}(D)^{+}$.
\end{theorem}

\begin{proof}
The proof is sketched in \cite{milneSVI}, 1.21; see also \cite{satake1980},
II, Proposition 3.2
\end{proof}

Thus, the pointed hermitian symmetric domains are classified by the pairs
$(H,u)$ as in the theorem. Changing the point corresponds to conjugating $u$
by an element of $H(\mathbb{R}{}$).

\subsection{Classification in terms of root systems}

We now assume that the reader is familiar with the classification of
semisimple algebraic groups over an algebraically closed field in terms of
root systems (e.g., \cite{humphreys1975}).

Let $D$ be an indecomposable hermitian symmetric domain. Then the corresonding
group $H$ is simple, and $H_{\mathbb{C}{}}$ is also simple because $H$ is an
inner form of its compact form (by SU2).\footnote{\label{inner}If
$H_{\mathbb{C}{}}$ is not simple, say, $H_{\mathbb{C}{}}=H_{1}\times H_{2}$,
then $H=(H_{1})_{\mathbb{C}/\mathbb{R}}$, and every inner form of $H$ is
isomorphic to $H$ itself (by Shapiro's lemma), which is not compact because
$H(\mathbb{R}{})=H_{1}(\mathbb{C}{})$.} Thus, from $D$ and a point $o$, we get
a simple algebraic group $H_{\mathbb{C}{}}$ over $\mathbb{C}{}$ and a
nontrivial cocharacter $\mu\overset{\textup{{\tiny def}}}{=}u_{\mathbb{C}{}%
}\colon\mathbb{G}_{m}\rightarrow H_{\mathbb{C}{}}$ satisfying the condition:

\begin{quote}
(*)\hspace{1cm}$\mathbb{G}_{m}$ acts on $\Lie(H_{\mathbb{C}{}})$ through the
characters $z$, $1$, $z^{-1}$.
\end{quote}

\noindent Changing $o$ replaces $\mu$ by a conjugate. Thus the next step is to
classify the pairs $(G,M)$ consisting of a simple algebraic group over
$\mathbb{C}{}$ and a conjugacy class of nontrivial cocharacters of $G$
satisfying (*).

Fix a maximal torus $T$ of $G$ and a base $S$ for the root system $R=R(G,T)$,
and let $R^{+}$ be the corresponding set of positive roots. As each $\mu$ in
$M$ factors through some maximal torus, and all maximal tori are conjugate, we
may choose $\mu\in M$ to factor through $T$. Among the $\mu$ in $M$ factoring
through $T$, there is exactly one such that $\langle\alpha,\mu\rangle\geq0$
for all $\alpha\in R^{+}$ (because the Weyl group acts simply transitively on
the Weyl chambers). The condition (*) says that $\langle\alpha,\mu\rangle
\in\{1,0,-1\}$ for all roots $\alpha$. Since $\mu$ is nontrivial, not all of
the $\langle\alpha,\mu\rangle$ can be zero, and so $\langle\tilde{\alpha}%
,\mu\rangle=1$ where $\tilde{\alpha}$ is the highest root. Recall that the
highest root $\tilde{\alpha}=\sum_{\alpha\in S}n_{\alpha}\alpha$ has the
property that $n_{\alpha}\geq m_{\alpha}$ for any other root $\sum_{\alpha\in
S}m_{\alpha}\alpha$; in particular, $n_{\alpha}\geq1$. It follows that
$\langle\alpha,\mu\rangle=0$ for all but one simple root $\alpha$, and that
for that simple root $\langle\alpha,\mu\rangle=1$ and $n_{\alpha}=1$. Thus,
the pairs $(G,M)$ are classified by the simple roots $\alpha$ for which
$n_{\alpha}=1$ --- these are called the \emph{special} simple roots. On
examining the tables, one finds that the special simple roots are as in the
following table:\renewcommand{\arraystretch}{1.3}

\begin{center}

\begin{tabular}
[c]{|c|c|c|c|}\hline
type & $\tilde{\alpha}$ & special roots & $\#$\\\hline
$A_{n}$ & $\alpha_{1}+\alpha_{2}+\cdots+\alpha_{n}$ & $\alpha_{1}%
,\ldots,\alpha_{n}$ & $n$\\\hline
$B_{n}$ & $\alpha_{1}+2\alpha_{2}+\cdots+2\alpha_{n}$ & $\alpha_{1}$ &
$1$\\\hline
$C_{n}$ & $2\alpha_{1}+\cdots+2\alpha_{n-1}+\alpha_{n}$ & $\alpha_{n}$ &
$1$\\\hline
$D_{n}$ & $\alpha_{1}+2\alpha_{2}+\cdots+2\alpha_{n-2}+\alpha_{n-1}+\alpha
_{n}$ & $\alpha_{1},\alpha_{n-1},\alpha_{n}$ & $3$\\\hline
$E_{6}$ & $\alpha_{1}+2\alpha_{2}+2\alpha_{3}+3\alpha_{4}+2\alpha_{5}%
+\alpha_{6}$ & $\alpha_{1},\alpha_{6}$ & $2$\\\hline
$E_{7}$ & $2\alpha_{1}+2\alpha_{2}+3\alpha_{3}+4\alpha_{4}+3\alpha_{5}%
+2\alpha_{6}+\alpha_{7}$ & $\alpha_{7}$ & 1\\\hline
$E_{8,}F_{4},G_{2}$ &  & none & 0\\\hline
\end{tabular}

\end{center}

\noindent\renewcommand{\arraystretch}{1.0}Mnemonic: the number of special
simple roots is one less than the connection index $(P(R)\colon Q(R))$ of the
root system.\footnote{It is possible to prove this directly. Let $S^{+}%
=S\cup\{\alpha_{0}\}$ where $\alpha_{0}$ is the negative of the highest root
--- the elements of $S^{+}$ correspond to the nodes of the completed Dynkin
diagram (\cite{bourbakiLie}, VI 4, 3). The group $P/Q$ acts on $S^{+}$, and it
acts simply transitively on the set $\left\{  \text{simple roots}\right\}
\cup\{\alpha_{0}\}$ (\cite{deligne1979}, 1.2.5).}

To every indecomposable hermitian symmetric domain we have attached a special
node, and we next show that every special node arises from a hermitian
symmetric domain. Let $G$ be a simple algebraic group over $\mathbb{C}{}$ with
a character $\mu$ satisfying (*). Let $U$ be the (unique) compact real form of
$G$, and let $\sigma$ be the complex conjugation on $G$ with respect to $U$.
Finally, let $H$ be the real form of $G$ such that $\inn(\mu(-1))\circ\sigma$
is the complex conjugation on $G$ with respect to $H$. The restriction of
$\mu$ to $U^{1}\subset\mathbb{C}{}^{\times}$ maps into $H(\mathbb{R}{})$ and
defines a homomorphism $u$ satisfying the conditions (SU1,2) of (\ref{h21}).
The hermitian symmetric domain corresponding to $(H,u)$ gives rise to
$(G,\mu)$. Thus there are indecomposable hermitian symmetric domains of all
possible types except $E_{8}$, $F_{4}$, and $G_{2}$.

Let $H$ be a real simple group such that there exists a homomorphism
$u\colon\mathbb{S}{}^{1}\rightarrow H$ satisfying (SV1,2). The set of such
$u$'s has two connected components, interchanged by $u\leftrightarrow u^{-1}$,
each of which is an $H(\mathbb{R}{})^{+}$-conjugacy class. The $u$'s form a
single $H(\mathbb{R}{})$-conjugacy class except when $s$ is moved by the
opposition involution (\cite{deligne1979}, 1.2.7, 1.2.8). This happens in the
following cases: type $A_{n}$ and $s\neq\frac{n}{2}$; type $D_{n}$ with $n$
odd and $s=\alpha_{n-1}$ or $\alpha_{n}$; type $E_{6}$ (see p.~\pageref{opp} below).

\subsection{Example: the Siegel upper half space}

A \emph{symplectic space}%
\index{symplectic space}
$(V,\psi)$ over a field $k$ is a finite dimensional vector space $V$ over $k$
together with a nondegenerate alternating form $\psi$ on $V$. The
\emph{symplectic group}%
\index{symplectic group}%
\index{S(psi)@$S(\psi)$}
$S(\psi)$ is the algebraic subgroup of $\GL_{V}$ of elements fixing $\psi$. It
is an almost simple simply connected group of type $C_{n-1}$ where $n=\frac
{1}{2}\dim_{k}V$.

Now let $k=\mathbb{R}{}$, and let $H=S(\psi)$. Let $D$ be the space of complex
structures $J$ on $V$ such that $(x,y)\mapsto\psi_{J}(x,y)\overset
{\textup{{\tiny def}}}{=}\psi(x,Jy)$ is symmetric and positive definite.
\noindent The symmetry is equivalent to $J$ lying in $S(\psi)$. Therefore, $D$
is the set of complex structures $J$ on $V$ for which $J\in H(\mathbb{R}{})$
and $\psi$ is a $J$-polarization for $H$.

The action,%
\[
g,J\mapsto gJg^{-1}\colon H(\mathbb{R}{})\times D\rightarrow D,
\]
of $H(\mathbb{R}{})$ on $D$ is transitive (\cite{milneSVI}, \S 6). Each $J\in
D$ defines an action of $\mathbb{C}{}$ on $V$, and%
\begin{equation}
\psi(Jx,Jy)=\psi(x,y)\text{ all }x,y\in V\implies\psi(zx,zy)=|z|^{2}%
\psi(x,y)\text{ all }x,y\in V. \label{q1}%
\end{equation}
Let $h_{J}\colon\mathbb{S}{}\rightarrow\GL_{V}$ be the homomorphism such that
$h_{J}(z)$ acts on $V$ as multiplication by $z$, and let $V_{\mathbb{C}{}%
}=V^{+}\oplus V^{-}$ be the decomposition of $V_{\mathbb{C}{}}$ into its $\pm
i$ eigenspaces for $J$. Then $h_{J}(z)$ acts on $V^{+}$ as $z$ and on $V^{-}$
as $\bar{z}$, and so it acts on%
\[
\Lie(H)_{\mathbb{C}{}}\subset\End(V)_{\mathbb{C}{}}\simeq V_{\mathbb{C}{}%
}^{\vee}\otimes V_{\mathbb{C}{}}=(V^{+}\oplus V^{-})^{\vee}\otimes(V^{+}\oplus
V^{-}),
\]
through the characters $z^{-1}\bar{z}$, $1$, $z\bar{z}^{-1}$.

For $z\in U^{1}$, (\ref{q1}) shows that $h_{J}(z)\in H$; choose a square root
$\sqrt{z}$ of $z$ in $U^{1}$, and let $u_{J}(z)=h_{J}(\sqrt{z})\bmod\pm1$.
Then $u_{J}$ is a well-defined homomorphism $U^{1}\rightarrow H^{\mathrm{ad}%
}(\mathbb{R}{})$, and it satisfies the conditions (SU1,2) of Theorem
\ref{h21}. Therefore, $D$ has a natural complex structure for which $z\in
U^{1}$ acts on $T_{J}(D)$ as multiplication by $z$ and $\mathrm{Hol}%
(D)^{+}=H^{\mathrm{ad}}(\mathbb{R}{})^{+}$. With this structure, $D$ is the
(unique) indecomposable hermitian symmetric domain of type $C_{n-1}$. It is
called the \emph{Siegel upper half space}%
\index{Siegel upper half space}
(of degree, or genus, $n$).

\section{Discrete subgroups of Lie groups}

\begin{quote}
{\small The algebraic varieties we are concerned with are quotients of
hermitian symmetric domains by the action of discrete groups. In this section,
we describe the discrete groups of interest to us.}
\end{quote}

\subsection{Lattices in Lie groups}

Let $H$ be a connected real Lie group. A \emph{lattice}%
\index{lattice}
in $H$ is a discrete subgroup ~$\Gamma$ of finite covolume,%
\index{finite covolume}
i.e., such that $H/\Gamma$ has finite volume with respect to an $H$-invariant
measure. For example, the lattices in $\mathbb{R}{}^{n}$ are exactly the
$\mathbb{Z}{}$-submodules generated by bases for $\mathbb{R}{}^{n}$, and two
such lattices are commensurable\footnote{Recall that two subgroup $S_{1}$ and
$S_{2}$ of a group are \emph{commensurable}%
\index{commensurable}
if $S_{1}\cap S_{2}$ has finite index in both $S_{1}$ and $S_{2}$.
Commensurability is an equivalence relation.} if and only if they generate the
same $\mathbb{Q}{}$-vector space. Every discrete subgroup commensurable with a
lattice is itself a lattice.

Now assume that $H$ is semisimple with finite centre. A lattice $\Gamma$ in
$H$ is \emph{irreducible}%
\index{irreducible lattic}
if $\Gamma\cdot N$ is dense in $H$ for every noncompact closed normal subgroup
$N$ of $H$. For example, if $\Gamma_{1}$ and $\Gamma_{2}$ are lattices in
$H_{1}$ and $H_{2}$, then the lattice $\Gamma_{1}\times\Gamma_{2}$ in
$H_{1}\times H_{2}$ is not irreducible because $\left(  \Gamma_{1}\times
\Gamma_{2}\right)  \cdot(1\times H_{2})=\Gamma_{1}\times H_{2}$ is not dense.
On the other hand, $\SL_{2}(\mathbb{Z}{}[\sqrt{2}])$ can be realized as an
irreducible lattice in $\SL_{2}(\mathbb{R}{})\times\SL_{2}(\mathbb{R}{})$ via
the embeddings $\mathbb{Z}{}[\sqrt{2}]\rightarrow\mathbb{R}{}$ given by
$\sqrt{2}\mapsto\sqrt{2}$ and $\sqrt{2}\mapsto-\sqrt{2}$.

\begin{theorem}
\label{h31}Let $H$ be a connected semisimple Lie group with no compact factors
and trivial centre, and let $\Gamma$ be a lattice $H$. Then $H$ can be written
(uniquely) as a direct product $H=H_{1}\times\cdots\times H_{r}$ of Lie
subgroups $H_{i}$ such that $\Gamma_{i}\overset{\textup{{\tiny def}}}{=}%
\Gamma\cap H_{i}$ is an irreducible lattice in $H_{i}$ and $\Gamma_{1}%
\times\cdots\times\Gamma_{r}$ has finite index in $\Gamma$
\end{theorem}

\begin{proof}
See \cite{morris2008}, 4.24.
\end{proof}

\begin{theorem}
\label{h32}Let $D$ be a hermitian symmetric domain, and let $H=\mathrm{Hol}%
(D)^{+}$. A discrete subgroup $\Gamma$ of $H$ is a lattice if and only if
$\Gamma\backslash D$ has finite volume. Let $\Gamma$ be a lattice in $H$; then
$D$ can be written (uniquely) as a product $D=D_{1}\times\cdots\times D_{r}$
of hermitian symmetric domains such that $\Gamma_{i}\overset
{\textup{{\tiny def}}}{=}\Gamma\cap\mathrm{Hol}(D_{i})^{+}$ is an irreducible
lattice in $\mathrm{Hol}(D_{i})^{+}$ and $\Gamma_{1}\backslash D_{1}%
\times\cdots\times\Gamma_{r}\backslash D_{r}$ is a finite covering of
$\Gamma\backslash D$.
\end{theorem}

\begin{proof}
\noindent Let $u_{o}$ be the homomorphism $\mathbb{S}{}^{1}\rightarrow H$
attached to a point $o\in D$ (see \ref{h20p}), and let $\theta$ be the Cartan
involution $\inn(u_{o}(-1))$. The centralizer of $u_{o}$ is contained in
$H(\mathbb{R}{})\cap H^{(\theta)}(\mathbb{R}{})$, which is compact. Therefore
$D$ is a quotient of $H(\mathbb{R}{})$ by a \textit{compact} subgroup, from
which the first statement follows. For the second statement, let
$H=H_{1}\times\cdots\times H_{r}$ be the decomposition of $H$ defined by
$\Gamma$ (see \ref{h31}). Then $u_{o}=(u_{1},\ldots,u_{r})$ where each $u_{i}$
is a homomorphism $\mathbb{S}{}^{1}\rightarrow H_{i}$ satisfying the
conditions SU1,2 of Theorem \ref{h21}. Now $D=D_{1}\times\cdots\times D_{r}$
with $D_{i}$ the hermitian symmetric domain corresponding to $(H_{i},u_{i})$.
This is the required decomposition.
\end{proof}

\begin{proposition}
\label{h32a}Let $\varphi\colon H\rightarrow H^{\prime}$ be a surjective
homomorphism of Lie groups with compact kernel. If $\Gamma$ is a lattice in
$H$, then $\varphi(\Gamma)$ is a lattice in $H^{\prime}$.
\end{proposition}

\begin{proof}
The proof is elementary (it requires only that $H$ and $H^{\prime}$ be locally
compact topological groups).
\end{proof}

\subsection{Arithmetic subgroups of algebraic groups}

Let $G$ be an algebraic group over $\mathbb{Q}$. When $r\colon G\rightarrow
\GL_{n}$ is an injective homomorphism, we let
\[
G(\mathbb{Z}{})_{r}=\{g\in G(\mathbb{Q}{})\mid r(g)\in\GL_{n}(\mathbb{Z}%
{})\}.
\]
Then $G(\mathbb{Z}{})_{r}$ is independent of $r$ up to commensurability
(\cite{borel1969}, 7.13), and we sometimes omit $r$ from the notation. A
subgroup $\Gamma$ of $G(\mathbb{Q}{})$ is \emph{arithmetic}%
\index{arithmetic subgroup}
if it is commensurable with $G(\mathbb{Z}{})_{r}$ for some $r$.

\begin{theorem}
\label{h33}Let $\varphi\colon G\rightarrow G^{\prime}$ be a surjective
homomorphism of algebraic groups over $\mathbb{Q}{}$. If $\Gamma$ is an
arithmetic subgroup of $G(\mathbb{Q}{})$, then $\varphi(\Gamma)$ is an
arithmetic subgroup of $G^{\prime}(\mathbb{Q}{}).$
\end{theorem}

\begin{proof}
See \cite{borel1969}, 8.11.
\end{proof}

An arithmetic subgroup $\Gamma$ of $G(\mathbb{Q}{})$ is obviously discrete in
$G(\mathbb{R}{})$, but it need not be a lattice. For example, $\mathbb{G}%
_{m}(\mathbb{Z}{})=\{\pm1\}$ is an arithmetic subgroup of $\mathbb{G}%
_{m}(\mathbb{Q}{})$ of infinite covolume in $\mathbb{G}_{m}(\mathbb{R}%
)=\mathbb{R}{}^{\times}$.

\begin{theorem}
\label{h34}Let $G$ be a reductive algebraic group over $\mathbb{Q}{}$, and let
$\Gamma$ be an arithmetic subgroup of $G(\mathbb{Q)}$.

\begin{enumerate}
\item The quotient $\Gamma\backslash G(\mathbb{R}{})$ has finite volume if and
only if $\Hom(G,\mathbb{G}_{m})=0$; in particular, $\Gamma$ is a lattice if
$G$ is semisimple.

\item (Godement compactness criterion) The quotient $\Gamma\backslash
G(\mathbb{R}{})$ is compact if and only if $\Hom(G,\mathbb{G}_{m})=0$ and
$G(\mathbb{Q}{})$ contains no unipotent element other than $1$.
\end{enumerate}
\end{theorem}

\begin{proof}
See \cite{borel1969}, 13.2, 8.4.\footnote{Statement (a) was proved in
particular cases by Siegel and others, and in general by Borel and
Harish-Chandra (1962).\nocite{borelH1962} Statement (b) was conjectured by
Godement, and proved independently by Mostow and Tamagawa (1962) and by Borel
and Harish-Chandra (1962). \nocite{borelH1962} \nocite{mostowT1962}}
\end{proof}

Let $k$ be a subfield of $\mathbb{C}{}$. An automorphism $\alpha$ of a
$k$-vector space $V$ is said to be \emph{neat}%
\index{neat}
if its eigenvalues in $\mathbb{C}{}$ generate a torsion free subgroup of
$\mathbb{C}^{\times}$. Let $G$ be an algebraic group over $\mathbb{Q}{}$. An
element $g\in G(\mathbb{Q}{})$ is \emph{neat}%
\index{neat}
if $\rho(g)$ is neat for one faithful representation $G\hookrightarrow\GL(V)$,
in which case $\rho(g)$ is neat for every representation $\rho$ of $G$ defined
over a subfield of $\mathbb{C}{}$. A subgroup of $G(\mathbb{Q}{})$ is
\emph{neat}%
\index{neat}
if all its elements are. See \cite{borel1969}, \S 17.

\begin{theorem}
\label{h35}Let $G$ be an algebraic group over $\mathbb{Q}{}$, and let $\Gamma$
be an arithmetic subgroup of $G(\mathbb{Q}{})$. Then, $\Gamma$ contains a neat
subgroup of finite index. In particular, $\Gamma$ contains a torsion free
subgroup of finite index.
\end{theorem}

\begin{proof}
\noindent In fact, the neat subgroup can be defined by congruence conditions.
See \cite{borel1969}, 17.4.
\end{proof}

\begin{definition}
\label{h35d}A semisimple algebraic group $G$ over $\mathbb{Q}{}$ is said to be
of \emph{compact type}%
\index{algebraic group!of compact type}
if $G(\mathbb{R}{})$ is compact, and it is said to be of \emph{algebraic
group!of noncompact type}%
\index{noncompact type}
if it does not contain a nontrivial connected normal algebraic subgroup of
compact type.
\end{definition}

Thus a simply connected or adjoint group over $\mathbb{Q}{}$ is of compact
type if all of its almost simple factors are of compact type, and it is of
noncompact type if \textit{none} of its almost simple factors is of compact
type. In particular, an algebraic group may be of neither type.

\begin{theorem}
[Borel density theorem]\label{h35e}Let $G$ be a semisimple algebraic group
over $\mathbb{Q}{}$. If $G$ is of noncompact type, then every arithmetic
subgroup of $G(\mathbb{Q}{})$ is dense in the Zariski topology.
\end{theorem}

\begin{proof}
See \cite{borel1969}, 15.12.
\end{proof}

\begin{proposition}
\label{h34a}Let $G$ be a simply connected algebraic group over $\mathbb{Q}{}$
of noncompact type, and let $\Gamma$ be an arithmetic subgroup of
$G(\mathbb{Q}{})$. Then $\Gamma$ is irreducible as a lattice in $G(\mathbb{R}%
{})$ if and only if $G$ is almost simple.
\end{proposition}

\begin{proof}
$\Rightarrow$: Suppose $G=G_{1}\times G_{2}$, and let $\Gamma_{1}$ and
$\Gamma_{2}$ be arithmetic subgroups in $G_{1}(\mathbb{Q}{})$ and
$G_{2}(\mathbb{Q}{})$. Then $\Gamma_{1}\times\Gamma_{2}$ is an arithmetic
subgroup of $G(\mathbb{Q}{})$, and so $\Gamma$ is commensurable with it, but
$\Gamma_{1}\times\Gamma_{2}$ is not irreducible.

$\Leftarrow$: Let $G(\mathbb{R}{})=H_{1}\times\cdots\times H_{r}$ be a
decomposition of the Lie group $G(\mathbb{R}{})$ such that $\Gamma_{i}%
\overset{\textup{{\tiny def}}}{=}\Gamma\cap H_{i}$ is an irreducible lattice
in $H_{i}$ (cf. Theorem \ref{h31}). There exists a finite Galois extension $F$
of $\mathbb{Q}{}$ in $\mathbb{R}{}$ and a decomposition $G_{F}=G_{1}%
\times\cdots\times G_{r}$ of $G_{F}$ into a product of algebraic subgroups
$G_{i}$ over $F$ such that $H_{i}=G_{i}(\mathbb{R}{})$ for all $i$. Because
$\Gamma_{i}$ is Zariski dense in $G_{i}$ (Borel density theorem), this last
decomposition is stable under the action of $\Gal(F/\mathbb{Q}{})$, and hence
arises from a decomposition over $\mathbb{Q}{}$. This contradicts the almost
simplicity of $G$ unless $r=1$.
\end{proof}

The rank, $\rank(G),$
\index{rank}%
of a semisimple algebraic group over $\mathbb{R}{}$ is the dimension of a
maximal split torus in $G$, i.e., $\rank(G)=r{}$ if $G$ contains an algebraic
subgroup isomorphic to $\mathbb{G}_{m}^{r}$ but not to $\mathbb{G}_{m}^{r+1}$.

\begin{theorem}
[Margulis superrigidity theorem]\label{h35f}Let $G$ and $H$ be algebraic
groups over $\mathbb{Q}{}$ with $G$ simply connected and almost simple. Let
$\ \Gamma$ be an arithmetic subgroup of $G(\mathbb{Q})$, and let $\delta
\colon\Gamma\rightarrow H(\mathbb{Q})$ be a homomorphism. If
$\rank(G_{\mathbb{R}{}})\geq2$, then the Zariski closure of $\delta(\Gamma)$
in $H$ is a semisimple algebraic group (possibly not connected), and there is
a unique homomorphism $\varphi\colon G\rightarrow H$ of algebraic groups such
that $\varphi(\gamma)=\delta(\gamma)$ for all $\gamma$ in a subgroup of finite
index in $\Gamma$.
\end{theorem}

\begin{proof}
This the special case of \cite{margulis1991}, Chapter VIII, Theorem B, p.~258,
in which $K=\mathbb{Q}{}=l$, $S=\{\infty\}$, $\mathbf{G}=G$, $\mathbf{H}=H$,
and $\Lambda=\Gamma$.
\end{proof}

\subsection{Arithmetic lattices in Lie groups}

For an algebraic group $G$ over $\mathbb{Q}{}$, $G(\mathbb{R}{})$ has a
natural structure of a real Lie group, which is connected if $G$ is simply
connected (Theorem of Cartan).

Let $H$ be a connected semisimple real Lie group with no compact factors and
trivial centre. A subgroup $\Gamma$ in $H$ is \emph{arithmetic}%
\index{arithmetic subgroup}
if there exists a simply connected algebraic group $G$ over $\mathbb{Q}{}$ and
a surjective homomorphism $\varphi\colon G(\mathbb{R}{})\rightarrow H$ with
compact kernel such that $\Gamma$ is commensurable with $\varphi
(G(\mathbb{Z}{}))$. Such a subgroup is a lattice by Theorem \ref{h34}(a) and
Proposition \ref{h32a}.

\begin{example}
\label{h35a}Let $H=\SL_{2}(\mathbb{R}{})$, and let $B$ be a quaternion algebra
over a totally real number field $F$ such that $H\otimes_{F,v}\mathbb{R}%
{}\approx M_{2}(\mathbb{R}{})$ for exactly one real prime $v{}$. Let $G$ be
the algebraic group over $\mathbb{Q}{}$ such that $G(\mathbb{Q}{})=\{b\in
B\mid\mathrm{Norm}_{B/\mathbb{Q}{}}(b)=1\}$. Then $H\otimes_{\mathbb{Q}{}%
}\mathbb{R}{}\approx M_{2}(\mathbb{R}{})\times\mathbb{H}{}\times\mathbb{H}%
{}\times\cdots$ where $\mathbb{H}{}$ is usual quaternion algebra, and so there
exists a surjective homomorphism $\varphi\colon G(\mathbb{R}{})\rightarrow
\SL_{2}(\mathbb{R}{})$ with compact kernel. The image under $\varphi$ of any
arithmetic subgroup of $G(\mathbb{Q}{})$ is an arithmetic subgroup $\Gamma$ of
$\SL_{2}(\mathbb{R}{})$, and every arithmetic subgroup of $\SL_{2}%
(\mathbb{R}{})$ is commensurable with one of this form. If $F=\mathbb{Q}{}$
and $B=M_{2}(\mathbb{Q}{})$, then $G=\SL_{2\mathbb{Q}{}}$ and $\Gamma
\backslash\SL_{2}(\mathbb{R}{})$ is noncompact (see \S 1); otherwise $B$ is a
division algebra, and $\Gamma\backslash\SL_{2}(\mathbb{R}{})$ is compact by
Godement's criterion (\ref{h34}b).
\end{example}

For almost a century, $\PSL_{2}(\mathbb{R}{})$ was the only simple Lie group
known to have nonarithmetic lattices, and when further examples were
discovered in the 1960s they involved only a few other Lie groups. This gave
credence to the idea that, except in a few groups of low rank, all lattices
are arithmetic (Selberg's conjecture). This was proved by Margulis in a very
precise form.

\begin{theorem}
[Margulis arithmeticity theorem]\label{h37}Every irreducible lattice in a
semisimple Lie group is arithmetic unless the group is isogenous to
$\SO(1,n)\times($compact$)$ or $\SU(1,n)\times($compact$)$.
\end{theorem}

\begin{proof}
\noindent For a discussion of the theorem, see \cite{morris2008}, \S 5B. For
proofs, see \cite{margulis1991}, Chapter IX, and \cite{zimmer1984}, Chapter 6.
\end{proof}

\begin{theorem}
\label{h37a}Let $H$ be the identity component of the group of automorphisms of
a hermitian symmetric domain $D$, and let $\Gamma$ be a discrete subgroup of
$H$ such that $\Gamma\backslash D$ has finite volume. If $\rank H_{i}\geq2$
for each factor $H_{i}$ in (\ref{h31}), then there exists a simply connected
algebraic group $G$ of noncompact type over $\mathbb{Q}{}$ and a surjective
homomorphism $\varphi\colon G(\mathbb{R}{})\rightarrow H$ with compact kernel
such that $\Gamma$ is commensurable with $\varphi(G(\mathbb{Z}{}))$. Moreover,
the pair $(G,\varphi)$ is unique up to a unique isomorphism.
\end{theorem}

\begin{proof}
The group $\Gamma$ is a lattice in $H$ by Theorem \ref{h32}. Each factor
$H_{i}$ is again the identity component of the group of automorphisms of a
hermitian symmetric domain (Theorem \ref{h32}), and so we may suppose that
$\Gamma$ is irreducible. The existence of the pair $(G,\varphi)$ just means
that $\Gamma$ is arithmetic, which follows from the Margulis arithmeticity
theorem (\ref{h37}).

Because $\Gamma$ is irreducible, $G$ is almost simple (see \ref{h34a}). As $G$
is simply connected, this implies that $G=\left(  G^{s}\right)  _{F/\mathbb{Q}%
{}}$ where $F$ is a number field and $G^{s}$ is a geometrically almost simple
algebraic group over $F$. If $F$ had a complex prime, $G_{\mathbb{R}}$ would
have a factor $(G^{\prime})_{\mathbb{C}/\mathbb{R}}$, but $(G^{\prime
})_{\mathbb{C}/\mathbb{R}}$ has no inner form except itself (by Shapiro's
lemma), and so this is impossible. Therefore $F$ is totally real.

Let $(G_{1},\varphi_{1})$ be a second pair. Because the kernel of $\varphi
_{1}$ is compact, its intersection with $G_{1}(\mathbb{Z}{})$ is finite, and
so there exists an arithmetic subgroup $\Gamma_{1}$ of $G_{1}(\mathbb{Q})$
such $\varphi_{1}|\Gamma_{1}$ is injective. Because $\varphi(G(\mathbb{Z))}$
and $\varphi_{1}(\Gamma_{1})$ are commensurable, there exists an arithmetic
subgroup $\Gamma^{\prime}$ of $G(\mathbb{Q}{})$ such that $\varphi
(\Gamma^{\prime})\subset\varphi_{1}(\Gamma_{1})$. Now the Margulis
superrigidity theorem \ref{h35f} shows that there exists a homomorphism
$\alpha\colon G\rightarrow G_{1}$ such that
\begin{equation}
\varphi_{1}(\alpha(\gamma))=\varphi(\gamma) \label{hq41}%
\end{equation}
for all $\gamma$ in a subgroup $\Gamma^{\prime\prime}$ of $\Gamma^{\prime}$ of
finite index. The subgroup $\Gamma^{\prime\prime}$ of $G(\mathbb{Q}{})$ is
Zariski-dense in $G$ (Borel density theorem \ref{h35e}), and so (\ref{hq41})
implies that
\begin{equation}
\varphi_{1}\circ\alpha(\mathbb{R}{})=\varphi. \label{hq42}%
\end{equation}
Because $G$ and $G_{1}$ are almost simple, (\ref{hq42}) implies that $\alpha$
is an isogeny, and because $G_{1}$ is simply connected, this implies that
$\alpha$ is an isomorphism. It is unique because it is uniquely determined on
an arithmetic subgroup of $G$.
\end{proof}

\subsection{Congruence subgroups of algebraic groups}

As in the case of elliptic modular curves, we shall need to consider a special
class of arithmetic subgroups, namely, the congruence subgroups.

Let $G$ be an algebraic group over $\mathbb{Q}{}$. Choose an embedding of $G$
into $\GL_{n}$, and define%
\[
\Gamma(N)=G(\mathbb{Q}{})\cap\left\{  A\in\GL_{n}(\mathbb{Z}{})\mid
A\equiv1\text{ mod }N\right\}  .
\]
A \emph{congruence subgroup}%
\index{congruence subgroup}%
\footnote{Subgroup defined by congruence conditions.} of $G(\mathbb{Q}{})$ is
any subgroup containing $\Gamma(N)$ as a subgroup of finite index. Although
$\Gamma(N)$ depends on the choice of the embedding, this definition does not
--- in fact, the congruence subgroups are exactly those of the form $K\cap
G(\mathbb{Q}{})$ for $K$ a compact open subgroup of $G(\mathbb{A}{}_{f})$.

For a surjective homomorphism $G\rightarrow G^{\prime}$ of algebraic groups
over $\mathbb{Q}{}$, the homomorphism $G(\mathbb{Q}{})\rightarrow G^{\prime
}(\mathbb{Q}{})$ need not send congruence subgroups to congruence subgroups.
For example, the image in $\PGL_{2}(\mathbb{Q}{})$ of a congruence subgroup of
$\SL_{2}(\mathbb{Q}{})$ is an arithmetic subgroup (see \ref{h33}) but not
necessarily a congruence subgroup.

Every congruence subgroup is an arithmetic subgroup, and for a simply
connected group the converse is often, but not always, true. For a survey of
what is known about the relation of congruence subgroups to arithmetic groups
(the congruence subgroup problem), see \cite{prasad2008}.

\begin{aside}
\label{h38}Let $H$ be a connected adjoint real Lie group without compact
factors. The pairs $(G,\varphi)$ consisting of a simply connected algebraic
group over $\mathbb{Q}{}$ and a surjective homomorphism $\varphi\colon
G(\mathbb{R}{})\rightarrow H$ with compact kernel have been classified (this
requires class field theory). Therefore the arithmetic subgroups of $H$ have
been classified up to commensurability. When all arithmetic subgroups are
congruence, there is even a classification of the groups themselves in terms
of congruence conditions or, equivalently, in terms of compact open subgroups
of $G(\mathbb{A}{}_{f})$.
\end{aside}

\section{Locally symmetric varieties}

\begin{quote}
{\small To obtain an algebraic variety from a hermitian symmetric domain, we
need to pass to the quotient by an arithmetic group.}
\end{quote}

\subsection{Quotients of hermitian symmetric domains}

Let $D$ be a hermitian symmetric domain, and let $\Gamma$ be a discrete
subgroup of $\mathrm{Hol}(D)^{+}$. If $\Gamma$ is torsion free, then $\Gamma$
acts freely on $D$, and there is a unique complex structure on $\Gamma
\backslash D$ for which the quotient map $\pi\colon D\rightarrow
\Gamma\backslash D$ is a local isomorphism. Relative to this structure, a map
$\varphi$ from $\Gamma\backslash D$ to a second complex manifold is
holomorphic if and only if $\varphi\circ\pi$ is holomorphic.

When $\Gamma$ is torsion free, we often write $D(\Gamma)$ for $\Gamma
\backslash D$ regarded as a complex manifold. In this case, $D$ is the
universal covering space of $D(\Gamma)$ and $\Gamma$ is the group of covering
transformations. The choice of a point $p\in D$ determines an isomorphism of
$\Gamma$ with the fundamental group $\pi_{1}(D(\Gamma),\pi p)$.\footnote{Let
$\gamma\in\Gamma$, and choose a path from $p$ to $\gamma p$; the image of this
in $\Gamma\backslash D$ is a loop whose homotopy class does not depend on the
choice of the path.}

The complex manifold $D(\Gamma)$ is locally symmetric in the sense that, for
each $p\in D(\Gamma)$, there is an involution $s_{p}$ defined on a
neighbourhood of $p$ having $p$ as an isolated fixed point.

\subsection{The algebraic structure on the quotient}

Recall that $X^{\text{an}}$ denotes the complex manifold attached to a smooth
complex algebraic variety $X$. The functor $X\rightsquigarrow X^{\text{an}}$
is faithful, but it is far from being surjective on arrows or on objects. For
example, $\left(  \mathbb{A}^{1}\right)  ^{\text{an}}=\mathbb{C}{}{}{}$ and
the exponential function is a nonpolynomial holomorphic map $\mathbb{C}%
{}\rightarrow\mathbb{C}$. A Riemann surface arises from an algebraic curve if
and only if it can be compactified by adding a finite number of points. In
particular, if a Riemann surface is an algebraic curve, then every bounded
function on it is constant, and so the complex upper half plane is not an
algebraic curve (the function $\frac{z-i}{z+i}$ is bounded).

\subsubsection{Chow's theorem}

An algebraic variety (resp. complex manifold) is \emph{projective}%
\index{projective}
if it can be realized as a closed subvariety of $\mathbb{P}{}^{n}$ for some
$n$ (resp. closed submanifold of $\left(  \mathbb{P}{}^{n}\right)
^{\text{an}}$).

\begin{theorem}
[\cite{chow1949}]\label{h41}The functor $X\rightsquigarrow X^{\text{an}}$ from
smooth projective complex algebraic varieties to projective complex manifolds
is an equivalence of categories.
\end{theorem}

\noindent In other words, a projective complex manifold has a unique structure
of a smooth projective algebraic variety, and every holomorphic map of
projective complex manifolds is regular for these structures. See
\cite{taylor2002}, 13.6, for the proof.

Chow's theorem remains true when singularities are allowed and
\textquotedblleft complex manifold\textquotedblright\ is replaced by
\textquotedblleft complex space\textquotedblright.

\subsubsection{The Baily-Borel theorem}

\begin{theorem}
[\cite{bailyB1966}]\label{h42}Every quotient $D(\Gamma)$ of a hermitian
symmetric domain $D$ by a torsion-free arithmetic subgroup $\Gamma$ of
$\mathrm{Hol}(D)^{+}$ has a canonical structure of an algebraic variety.
\end{theorem}

More precisely, let $G$ be the algebraic group over $\mathbb{Q}{}$ attached to
$(D,\Gamma)$ in Theorem \ref{h37a}, and assume, for simplicity, that $G$ has
no normal algebraic subgroup of dimension $3$. Let $A_{n}$ be the vector space
of automorphic forms on $D$ for the $n$th power of the canonical automorphy
factor. Then $A=\bigoplus\nolimits_{n\geq0}A_{n}$ is a finitely generated
graded $\mathbb{C}{}$-algebra, and the canonical map
\[
D(\Gamma)\rightarrow D(\Gamma)^{\ast}\overset{\textup{{\tiny def}}}%
{=}\Proj(A)
\]
realizes $D(\Gamma)$ as a Zariski-open subvariety of the projective algebraic
variety $D(\Gamma)^{\ast}$ (\cite{bailyB1966}, \S 10).

\subsubsection{Borel's theorem}

\begin{theorem}
[\cite{borel1972}]\label{h43}Let $D(\Gamma)$ be the quotient $\Gamma\backslash
D$ in (\ref{h42}) endowed with its canonical algebraic structure, and let $V$
be a smooth complex algebraic variety. Every holomorphic map $f\colon
V^{\text{an}}\rightarrow D(\Gamma)^{\text{an}}$ is regular.
\end{theorem}

\noindent In the proof of Proposition \ref{h11}, we saw that for curves this
theorem follows from the big Picard theorem. Recall that this says that every
holomorphic map from a punctured disk to $\mathbb{P}{}^{1}(\mathbb{C}%
{})\smallsetminus\{$three points$\}$ extends to a holomorphic map from the
whole disk to $\mathbb{P}{}^{1}(\mathbb{C}{})$. Following earlier work of
Kwack and others, Borel generalized the big Picard theorem in two respects:
the punctured disk is replaced by a product of punctured disks and disks, and
the target space is allowed to be any quotient of a hermitian symmetric domain
by a torsion-free arithmetic group.

Resolution of singularities (\cite{hironaka1964}) shows that every smooth
quasi-projective algebraic variety $V$ can be embedded in a smooth projective
variety $\bar{V}$ as the complement of a divisor with normal crossings. This
condition means that $\bar{V}^{\text{an}}\smallsetminus V^{\text{an}}$ is
locally a product of disks and punctured disks. Therefore $f|V^{\text{an}}$
extends to a holomorphic map $\bar{V}^{\text{an}}\rightarrow D(\Gamma)^{\ast}$
(by Borel) and so is a regular map (by Chow).

\subsection{Locally symmetric varieties}

A \emph{locally symmetric variety}%
\index{locally symmetric variety}
is a smooth algebraic variety $X$ over $\mathbb{C}{}$ such that $X^{\text{an}%
}$ is isomorphic to $\Gamma\backslash D$ for some hermitian symmetric domain
$D$ and torsion-free subgroup $\Gamma$ of $\mathrm{Hol}(D)$.\footnote{As
$\mathrm{Hol}(D)$ has only finitely many components, $\Gamma\cap
\mathrm{Hol}(D)^{+}$ has finite index in $\Gamma$. Sometimes we only allow
discrete subgroups of $\mathrm{Hol}(D)$ contained in $\mathrm{Hol}(D)^{+}$. In
the theory of Shimura varieties, we generally consider only \textquotedblleft
sufficiently small\textquotedblright\ discrete subgroups, and we regard the
remainder as \textquotedblleft noise\textquotedblright. Algebraic geometers do
the opposite.} In other words, $X$ is a locally symmetric variety if the
universal covering space $D$ of $X^{\text{an}}$ is a hermitian symmetric
domain and the group of covering transformations of $D$ over $X^{\text{an}}$
is a torsion-free subgroup $\Gamma$ of $\mathrm{Hol}(D)$. When $\Gamma$ is an
arithmetic subgroup of $\mathrm{Hol}(D)^{+}$, $X$ is called an
\emph{arithmetic locally symmetric variety}. The group $\Gamma$ is
automatically a lattice, and so the Margulis arithmeticity theorem (\ref{h37})
shows that nonarithmetic locally symmetric varieties can occur only when there
are factors of low dimension.

A nonsingular projective curve over $\mathbb{C}{}$ has a model over
$\mathbb{Q}{}^{\mathrm{al}}$ if and only if it contains an arithmetic locally
symmetric curve as the complement of a finite set (Belyi; see \cite{serre1990}%
, p.~71). This suggests that there are too many arithmetic locally symmetric
varieties for us to be able to say much about their arithmetic.

Let $D(\Gamma)$ be an arithmetic locally symmetric variety. Recall that
$\Gamma$ is arithmetic if there is a simply connected algebraic group $G$ over
$\mathbb{Q}{}$ and a surjective homomorphism $\varphi\colon G(\mathbb{R}%
{})\rightarrow\mathrm{Hol}(D)^{+}$ with compact kernel such that $\Gamma$ is
commensurable with $\varphi(G(\mathbb{Z}{}))$. If there exists a
\textit{congruence subgroup }$\Gamma_{0}$ of $G(\mathbb{Z})$ such that
$\Gamma$ contains $\varphi(\Gamma_{0})$ as a subgroup of finite index, then we
call $D(\Gamma)$ a \emph{connected Shimura variety}%
\index{connected Shimura variety}%
. Only for Shimura varieties do we have a rich arithmetic theory (see
\cite{deligne1971}, \cite{deligne1979}, and the many articles of Shimura,
especially, \cite{shimura1964, shimura1966, shimura1967a, shimura1967c,
shimura1970}).

\subsection{Example: Siegel modular varieties}

For an abelian variety $A$ over $\mathbb{C}{}$, the exponential map defines an
exact sequence%
\[
0\longrightarrow\Lambda\longrightarrow T_{0}(A^{\text{an}})\overset{\exp
}{\longrightarrow}A^{\text{an}}\longrightarrow0
\]
with $T_{0}(A^{\text{an}})$ a complex vector space and $\Lambda$ a lattice in
$T_{0}(A^{\text{an}})$ canonically isomorphic to $H_{1}(A^{\text{an}%
},\mathbb{Z}{})$.

\begin{theorem}
[Riemann's Theorem]\label{h44}The functor $A\rightsquigarrow(T_{0}%
(A),\Lambda)$ is an equivalence from the category of abelian varieties over
$\mathbb{C}{}$ to the category of pairs consisting of a $\mathbb{C}{}$-vector
space $V$ and a lattice $\Lambda$ in $V$ that admits a Riemann form.
\end{theorem}

\begin{proof}
See, for example, \cite{mumford1970}, Chapter I.
\end{proof}

A Riemann form for a pair $(V,\Lambda)$ is an alternating form $\psi
\colon\Lambda\times\Lambda\rightarrow\mathbb{Z}{}$ such that the pairing
$(x,y)\mapsto\psi_{\mathbb{R}{}}(x,\sqrt{-1}y)\colon V\times V\rightarrow
\mathbb{R}{}$ is symmetric and positive definite. \noindent Here
$\psi_{\mathbb{R}{}}$ denotes the linear extension of $\psi$ to $\mathbb{R}%
{}\otimes_{\mathbb{Z}}\Lambda\simeq V$. A principal polarization on an abelian
variety $A$ over $\mathbb{C}{}$ is Riemann form for $(T_{0}(A),\Lambda)$ whose
discriminant is $\pm1$. A level-$N$ structure on an abelian variety over
$\mathbb{C}{}$ is defined similarly to an elliptic curve (see \S 1; we require
it to be compatible with the Weil pairing).

Let $(V,\psi)$ be a symplectic space over $\mathbb{R}{}$, and let $\Lambda$ be
a lattice in $V$ such that $\psi(\Lambda,\Lambda)\subset\mathbb{Z}$ and
$\psi|_{\Lambda\times\Lambda}$ has discriminant $\pm1$. The points of the
corresponding Siegel upper half space $D$ are the complex structures $J$ on
$V$ such that $\psi_{J}$ is Riemann form (see \S 2). The map $J\mapsto
(V,J)/\Lambda$ is a bijection from $D$ to the set of isomorphism classes of
principally polarized abelian varieties over $\mathbb{C}{}$ equipped with an
isomorphism $\Lambda\rightarrow H_{1}(A,\mathbb{Z}{})$. On passing to the
quotient by the principal congruence subgroup $\Gamma(N)$, we get a bijection
from $D_{N}\overset{\textup{{\tiny def}}}{=}\Gamma(N)\backslash D$ to the set
of isomorphism classes of principally polarized abelian over $\mathbb{C}{}$
equipped with a level-$N$ structure.

\begin{proposition}
\label{h45}Let $f\colon A\rightarrow S$ be a family of principally polarized
abelian varieties on a smooth algebraic variety $S$ over $\mathbb{C}{}$, and
let $\eta$ be a level-$N$ structure on $A/S$. The map $\gamma\colon
S(\mathbb{C}{})\rightarrow D_{N}(\mathbb{C}{})$ sending $s\in S(\mathbb{C}{})$
to the point of $\Gamma(N)\backslash D$ corresponding to $(A_{s},\eta_{s})$ is regular.
\end{proposition}

\begin{proof}
The holomorphicity of $\gamma$ can be proved by the same argument as in the
proof of Proposition \ref{h11}. Its algebraicity then follows from Borel's
theorem \ref{h43}.
\end{proof}

Let $\mathcal{F}$ be the functor sending a scheme $S$ of finite type over
$\mathbb{C}{}$ to the set of isomorphism classes of pairs consisting of a
family of principally polarized abelian varieties $f\colon A\rightarrow S$
over $S$ and a level-$N$ structure on $A$. When $N\geq3$, $\mathcal{F}{}$ is
representable by a smooth algebraic variety $S_{N}$ over $\mathbb{C}{}$
(\cite{mumford1965}, Chapter 7). This means that there exists a (universal)
family of principally polarized abelian varieties $A/S_{N}$ and a level-$N$
structure $\eta$ on $A/S_{N}$ such that, for any similar pair $(A^{\prime
}/S,\eta^{\prime})$ over a scheme $S$, there exists a unique morphism
$\alpha\colon S\rightarrow S_{N}$ for which $\alpha^{\ast}(A/S_{N}%
,\eta)\approx(A^{\prime}/S^{\prime},\eta^{\prime})$.

\begin{theorem}
\label{h46}There is a canonical isomorphism $\gamma\colon S_{N}\rightarrow
D_{N}$.
\end{theorem}

\begin{proof}
The proof is the same as that of Theorem \ref{h14}.
\end{proof}

\begin{corollary}
\label{h48}The universal family of complex tori on $D_{N}$ is algebraic.
\end{corollary}

\section{Variations of Hodge structures}

\begin{quote}
{\small We review the definitions.}
\end{quote}

\subsection{The Deligne torus}

The \emph{Deligne torus}%
\index{Deligne torus}
is the algebraic torus%
\index{S@$\mathbb{S}{}$}
$\mathbb{S}$ over $\mathbb{R}{}$ obtained from $\mathbb{G}_{m}$ over
$\mathbb{C}{}$ by restriction of the base field${}$; thus%
\[
\mathbb{S}(\mathbb{R})=\mathbb{C}^{\times},\quad\mathbb{S}_{\mathbb{C}{}%
}\simeq\mathbb{G}_{m}\times\mathbb{G}_{m}.
\]
The map $\mathbb{S}{}(\mathbb{R}{})\rightarrow\mathbb{S}{}(\mathbb{C}{})$
induced by $\mathbb{R}{}\rightarrow\mathbb{C}{}$ is $z\mapsto(z,\bar{z})$.
There are homomorphisms%
\[%
\begin{array}
[c]{ccccccc}%
\mathbb{G}_{m} & \xrightarrow[\phantom{a\mapsto a^{-1}}]{w} & \mathbb{S}{} &
\xrightarrow[\phantom{z\mapsto{z\bar{z}}}]{t} & \mathbb{G}_{m}, &  & t\circ
w=-2,\\
\mathbb{R}{}^{\times} & \xrightarrow{a\mapsto a^{-1}} & \mathbb{C}{}^{\times}
& \xrightarrow{z\mapsto{z\bar{z}}} & \mathbb{R}{}^{\times}. &  &
\end{array}
\]
The kernel of $t$ is $\mathbb{S}{}^{1}$. A homomorphism $h\colon
\mathbb{S}\rightarrow G$ of real algebraic groups gives rise to cocharacters%
\[%
\begin{array}
[c]{lll}%
\mu_{h}\colon\mathbb{G}_{m}\rightarrow G_{\mathbb{C}}, & z\mapsto
h_{\mathbb{C}{}}(z,1), & z\in\mathbb{G}_{m}(\mathbb{C})=\mathbb{C}^{\times},\\
w_{h}\colon\mathbb{G}_{m}\rightarrow G, & w_{h}=h\circ w & \text{(\emph{weight
homomorphism}).}%
\end{array}
\]
The following formulas are useful ($\mu=\mu_{h}$):%
\begin{align}
h_{\mathbb{C}{}}(z_{1},z_{2})  &  =\mu(z_{1})\cdot\overline{\mu}(z_{2});\quad
h(z)=\mu(z)\cdot\overline{\mu(z)}\label{hq40}\\
h(i)  &  =\mu(-1)\cdot w_{h}(i). \label{hq21}%
\end{align}

\subsection{Real Hodge structures}

A \emph{real Hodge structure}%
\index{real Hodge structure}
is a representation $h\colon\mathbb{S}{}\rightarrow\GL_{V}$ of $\mathbb{S}$ on
a real vector space $V$. Equivalently, it is a real vector space $V$ together
with a \emph{Hodge decomposition} (%
\index{Hodge decomposition}%
),
\[
V_{\mathbb{C}{}}=\bigoplus\nolimits_{p,q\in\mathbb{Z}{}}V^{p,q}\text{ such
that }\overline{V^{p,q}}=V^{q,p}\text{ for all }p,q.
\]
To pass from one description to the other, use the rule (\cite{deligne1973,
deligne1979}):%
\[
v\in V^{p,q}\iff h(z)v=z^{-p}\bar{z}^{-q}v\text{, all }z\in\mathbb{C}^{\times
}\text{.}%
\]
The integers $h^{p,q}\overset{\textup{{\tiny def}}}{=}\dim_{\mathbb{C}}%
V^{p,q}$ are called the \emph{Hodge numbers}%
\index{Hodge numbers}
of the Hodge structure. A real Hodge structure defines a (weight) gradation on
$V$,
\[
V=\bigoplus\nolimits_{m\in\mathbb{Z}{}}V_{m},\quad V_{m}=V\cap\left(
\bigoplus\nolimits_{p+q=m}V^{p,q}\right)  \text{,}%
\]
and a descending \emph{Hodge filtration}%
\index{Hodge filtration}%
,
\[
V_{\mathbb{C}{}}\supset\cdots\supset F^{p}\supset F^{p+1}\supset\cdots
\supset0,\quad F^{p}=\bigoplus\nolimits_{p^{\prime}\geq p}V^{p^{\prime
},q^{\prime}}\text{.}%
\]
The weight gradation and Hodge filtration together determine the Hodge
structure because%
\[
V^{p,q}=\left(  V_{p+q}\right)  _{\mathbb{C}{}}\cap F^{p}\cap\overline{F^{q}%
}.
\]
Note that the weight gradation is defined by $w_{h}$. A filtration $F$ on
$V_{\mathbb{C}{}}$ arises from a Hodge structure of weight $m$ on $V$ if and
only if%
\[
V=F^{p}\oplus\overline{F^{q}}\text{ whenever }p+q=m+1.
\]
The $\mathbb{R}{}$-linear map $C=h(i)$ is called the \emph{Weil operator}%
\index{Weil operator}%
. It acts as $i^{q-p}$ on $V^{p,q}$, and $C^{2}$ acts as $(-1)^{m}$ on $V_{m}$.

Thus a Hodge structure on a real vector space $V$ can be regarded as a
homomorphism $h\colon\mathbb{S}{}\rightarrow\GL_{V}$, a Hodge decomposition of
$V$, or a Hodge filtration together with a weight gradation of $V$. We use the
three descriptions interchangeably.

\begin{plain}
\label{h47}Let $V$ be a real vector space. To give a Hodge structure $h$ on
$V$ of type $\{(-1,0),(0,-1)\}$ is the same as giving a complex structure on
$V$: given $h$, let $J$ act as $C=h(i)$; given a complex structure, let $h(z)$
act as multiplication by $z$. The Hodge decomposition $V_{\mathbb{C}{}%
}=V^{-1,0}\oplus V^{0,-1}$ corresponds to the decomposition $V_{\mathbb{C}{}%
}=V^{+}\oplus V^{-}$ of $V_{\mathbb{C}{}}$ into its $J$-eigenspaces.
\end{plain}

\subsection{Rational Hodge structures}

A \emph{rational Hodge structure}%
\index{rational Hodge structure}
is a $\mathbb{Q}{}$-vector space $V$ together with a real Hodge structure on
$V_{\mathbb{R}{}}$ such that the weight gradation is defined over
$\mathbb{Q}{}$. Thus to give a rational Hodge structure on $V$ is the same as giving

\begin{itemize}
\item a gradation $V=\bigoplus_{m}V_{m}$ on $V$ together with a real Hodge
structure of weight $m$ on $V_{m\mathbb{R}{}}$ for each $m$, \textit{or}

\item a homomorphism $h\colon\mathbb{S}{}\rightarrow\GL_{V_{\mathbb{R}{}}}$
such that $w_{h}\colon\mathbb{G}_{m}\rightarrow\GL_{V_{\mathbb{R}{}}}$ is
defined over $\mathbb{Q}{}$.
\end{itemize}

\noindent The \emph{Tate Hodge structure}%
\index{Tate Hodge structure}
$\mathbb{Q}(m)$ is defined to be the $\mathbb{Q}$-subspace $(2\pi
i)^{m}\mathbb{Q}$ of $\mathbb{C}{}$ with $h(z)$ acting as multiplication by
$\mathrm{Norm}_{\mathbb{C}{}/\mathbb{R}{}}(z)^{m}=(z\bar{z})^{m}$. It has
weight $-2m$ and type $(-m,-m)$.

\subsection{Polarizations}

A \emph{polarization}%
\index{polarization}
of a real Hodge structure $(V,h)$ of weight $m$ is a morphism of Hodge
structures
\begin{equation}
\psi\colon V\otimes V\rightarrow\mathbb{\mathbb{R}{}}(-m),\quad m\in
\mathbb{Z}, \label{hq15}%
\end{equation}
such that
\begin{equation}
(x,y)\mapsto(2\pi i)^{m}\psi(x,Cy)\colon V\times V\rightarrow\mathbb{R}
\label{hq16}%
\end{equation}
is symmetric and positive definite. The condition (\ref{hq16}) means that
$\psi$ is symmetric if $m$ is even and skew-symmetric if it is odd, and that
$(2\pi i)^{m}\cdot i^{p-q}\psi_{\mathbb{C}{}}(x,\bar{x})>0$ for $x\in V^{p,q}$.

A \emph{polarization}%
\index{polarization}
of a rational Hodge structure $(V,h)$ of weight $m$ is a morphism of rational
Hodge structures $\psi\colon V\otimes V\rightarrow\mathbb{Q}{}(-m)$ such that
$\psi_{\mathbb{R}{}}$ is a polarization of $(V_{\mathbb{R}{}},h)$. A rational
Hodge structure $(V,h)$ is polarizable if and only if $(V_{\mathbb{R}{}},h)$
is polarizable (cf. \ref{h20c}).

\subsection{Local systems and vector sheaves with connection}

Let $S$ be a complex manifold. A \emph{connection}%
\index{connection}
on a vector sheaf $\mathcal{V}{}$ on $S$ is a $\mathbb{C}{}$-linear
homomorphism $\nabla\colon\mathcal{V}{}\rightarrow\Omega_{S}^{1}%
\otimes\mathcal{V}{}$ satisfying the Leibniz condition%
\[
\nabla(fv)=df\otimes v+f\cdot\nabla v
\]
for all local sections $f$ of $\mathcal{O}{}_{S}$ and $v$ of $\mathcal{V}{}$.
The \emph{curvature}%
\index{curvature}
of $\nabla$ is the composite of $\nabla$ with the map%
\begin{align*}
\nabla_{1}\colon\Omega_{S}^{1}\otimes\mathcal{V}{}  &  \rightarrow\Omega
_{S}^{2}\otimes\mathcal{V}{}\\
\omega\otimes v  &  \mapsto d\omega\otimes v-\omega\wedge\nabla(v).
\end{align*}
A connection $\nabla$ is said to be \emph{flat}%
\index{flat connection}
if its curvature is zero. In this case, the kernel $\mathcal{V}{}^{\nabla}$ of
$\nabla$ is a local system of complex vector spaces on $S$ such that
$\mathcal{O}{}_{S}\otimes\mathcal{V}{}^{\nabla}\simeq\mathcal{V}{}$.

Conversely, let $\mathsf{V}$ be a local system of complex vector spaces on
$S$. The vector sheaf $\mathcal{V}{}=\mathcal{O}{}_{S}\otimes\mathsf{V}$ has a
canonical connection $\nabla$: on any open set where $\mathsf{V}$ is trivial,
say $\mathsf{V}\approx\mathbb{C}{}^{n}$, the connection is the map
$(f_{i})\mapsto(df_{i})\colon\left(  \mathcal{O}{}_{S}\right)  ^{n}%
\rightarrow\left(  \Omega_{S}^{1}\right)  ^{n}$. This connection is flat
because $d\circ d=0$. Obviously for this connection, $\mathcal{V}{}^{\nabla
}\simeq\mathsf{V}$.

In this way, we obtain an equivalence between the category of vector sheaves
on $S$ equipped with a flat connection and the category of local systems of
complex vector spaces.

\subsection{Variations of Hodge structures}

Let $S$ be a complex manifold. By a \emph{family of real Hodge structures on}%
\index{family of real Hodge structures on}
$S$ we mean a holomorphic family. For example, a family of real Hodge
structures on $S$ of weight $m$ is a local system $\mathsf{V}$ of $\mathbb{R}%
$-vector spaces on $S$ together with a filtration $F$ on $\mathcal{V}%
\overset{\textup{{\tiny def}}}{=}\mathcal{O}{}_{S}\otimes_{\mathbb{R}{}%
}\mathsf{V}{}$ by holomorphic vector subsheaves that gives a Hodge filtration
at each point, i.e., such that
\[
F^{p}\mathcal{V}{}_{s}\oplus\overline{F^{m+1-p}\mathcal{V}_{s}}\simeq
\mathcal{V}{}_{s},\quad\text{all }s\in S\text{, }p\in\mathbb{Z}{}.
\]
For the notion of a \emph{family of rational Hodge structures}%
\index{family of rational Hodge structures}%
, replace $\mathbb{R}{}$ with $\mathbb{Q}{}$.

A \emph{polarization}%
\index{polarization}
of a family of real Hodge structures of weight $m$ is a bilinear pairing of
local systems%
\[
\psi\colon\mathsf{V}\times\mathsf{V}\rightarrow\mathbb{R}(-m)
\]
that gives a polarization at each point $s$ of $S$. For rational Hodge
structures, replace $\mathbb{R}{}$ with $\mathbb{Q}{}$.

Let $\nabla$ be connection on a vector sheaf $\mathcal{V}{}$. A holomorphic
vector field $Z$ on $S$ is a map $\Omega_{S}^{1}\rightarrow\mathcal{O}{}_{S}$,
and it defines a map $\nabla_{Z}\colon\mathcal{V}{}\rightarrow\mathcal{V}{}$.
A family of rational Hodge structures $\mathsf{V}$ on $S$ is a
\emph{variation}%
\index{variation of Hodge structures}
of rational Hodge structures on $S$ if it satisfies the following axiom
(\emph{Griffiths transversality}%
\index{Griffiths transversality}%
):%
\[
\nabla_{Z}(F^{p}\mathcal{V}{})\subset F^{p-1}\mathcal{V}{}\text{ for all
}p\text{ and }Z\text{.}%
\]
Equivalently,
\[
\nabla(F^{p}\mathcal{V})\subset\Omega_{S}^{1}\otimes F^{p-1}\mathcal{V}\text{
for all }p.
\]
Here $\nabla$ is the flat connection on $\mathcal{V}{}\overset
{\textup{{\tiny def}}}{=}\mathcal{O}{}_{S}\otimes_{\mathbb{Q}{}}\mathsf{V}$
defined by $\mathsf{V}$.

These definitions are motivated by the following theorem.

\begin{theorem}
[\cite{griffiths1968}]\label{h51}\textsc{ }Let $f\colon X\rightarrow S$ be a
smooth projective map of smooth algebraic varieties over $\mathbb{C}$. For
each $m$, the local system $R^{m}f_{\ast}\mathbb{Q}{}$ of $\mathbb{Q}{}%
$-vector spaces on $S^{\text{an}}$ together with the de Rham filtration on
$\mathcal{O}_{S}\otimes_{\mathbb{Q}}Rf_{\ast}\mathbb{Q}{}\simeq Rf_{\ast
}(\Omega_{X/\mathbb{C}}^{\bullet})$ is a polarizable variation of rational
Hodge structures of weight $m$ on $S^{\text{an}}$.
\end{theorem}

This theorem suggests that the first step in realizing an algebraic variety as
a moduli variety should be to show that it carries a polarized variation of
rational Hodge structures.

\section{Mumford-Tate groups and their variation in families}

{\small We define Mumford-Tate groups, and we study their variation in
families. Throughout this section, \textquotedblleft Hodge
structure\textquotedblright\ means \textquotedblleft rational Hodge
structure\textquotedblright.}

\subsection{The conditions (SV)}

We list some conditions on a homomorphism $h\colon\mathbb{S}\rightarrow G$ of
real connected algebraic groups:\label{SV}%
\index{SV (conditions)}%

\begin{description}
\item[\textmd{SV1:}] the Hodge structure on the Lie algebra of $G$ defined by
$\Ad\circ h\colon\mathbb{S}\rightarrow\GL_{\Lie(G)}$ is of type
$\{(1,-1),(0,0),(-1,1)\}$;

\item[\textmd{SV2:}] $\inn(h(i))$ is a Cartan involution of $G^{\text{ad}}$.
\end{description}

\noindent In particular, (SV2) says that the Cartan involutions of
$G^{\mathrm{ad}}$ are inner, and so $G^{\mathrm{ad}}$ is an inner form of its
compact form. This implies that the simple factors of $G^{\mathrm{ad}}$ are
geometrically simple (see footnote \ref{inner}, p.~\pageref{inner}).

Condition (SV1) implies that the Hodge structure on $\Lie(G)$ defined by $h$
has weight $0$, and so $w_{h}(\mathbb{G}_{m})\subset Z(G)$. In the presence of
this condition, we sometimes need to consider a stronger form of (SV2):

\begin{description}
\item[\textmd{SV2*:}] $\inn(h(i))$ is a Cartan involution of $G/w_{h}%
(\mathbb{G}_{m})$.
\end{description}

\noindent Note that (SV2*) implies that $G$ is reductive.

Let $G$ be an algebraic group over $\mathbb{Q}$, and let $h$ be a homomorphism
$\mathbb{S}\rightarrow G_{\mathbb{R}}$. We say that $(G,h)$ satisfies the
condition (SV1) or (SV2) if $(G_{\mathbb{R}},h)$ does. When $w_{h}$ is defined
over $\mathbb{Q}{}$, we say that $(G,h)$ satisfies (SV2*) if $(G_{\mathbb{R}%
{}},h)$ does. We shall also need to consider the condition:

\begin{description}
\item[\textmd{SV3:}] $G^{\mathrm{ad}}$ has no $\mathbb{Q}$-factor on which the
projection of $h$ is trivial.
\end{description}

\noindent In the presence of (SV1,2), the condition (SV3) is equivalent to
$G^{\mathrm{ad}}$ being of noncompact type (apply Lemma 4.7 of \cite{milneSVI}).

Each condition holds for a homomorphism $h$ if and only if it holds for a
conjugate of $h$ by an element of $G(\mathbb{R}{})$.

Let $G$ be a reductive group over $\mathbb{Q}{}$. Let $h$ be a homomorphism
$\mathbb{S}{}\rightarrow G_{\mathbb{R}{}}$, and let $\bar{h}\colon\mathbb{S}%
{}\rightarrow G_{\mathbb{R}{}}^{\mathrm{ad}}$ be $\ad\circ h$. Then $(G,h)$
satisfies (SV1,2,3) if and only if $(G^{\mathrm{ad}},\bar{h})$ satisfies the
same conditions.\footnote{For (SV1), note that $\Ad(h(z))\colon\Lie(G){}%
\rightarrow\Lie(G)$ is the derivative of $\ad(h(z))\colon G\rightarrow G$. The
latter is trivial on $Z(G)$, and so the former is trivial on $\Lie(Z(G))$.}

\begin{remark}
\label{h90a}Let $H$ be a real algebraic group. The map $z\mapsto z/\bar{z}$
defines an isomorphism $\mathbb{S}{}/w(\mathbb{G}_{m})\simeq\mathbb{S}{}^{1}$,
and so the formula%
\begin{equation}
h(z)=u(z/\bar{z}) \label{hq45}%
\end{equation}
defines a one-to-one correspondence between the homomorphisms $h\colon
\mathbb{S}{}\rightarrow H$ trivial on $w(\mathbb{G}_{m})$ and the
homomorphisms $u\colon\mathbb{S}{}^{1}\rightarrow H$. When $H$ has trivial
centre, $h$ satisfies SV1 (resp. SV2) if and only if $u$ satisfies SU1 (resp. SU2).
\end{remark}

\begin{nt}
Conditions (SV1), (SV2), and (SV3) are respectively the conditions (2.1.1.1),
(2.1.1.2), and (2.1.1.3) of \cite{deligne1979}, and (SV2*) is the condition (2.1.1.5).
\end{nt}

\subsection{Definition of Mumford-Tate groups}

Let $(V,h)$ be a rational Hodge structure. Following \cite{deligne1972}, 7.1,
we define the \emph{Mumford-Tate group}%
\index{Mumford-Tate group}
of $(V,h)$ to be the smallest algebraic subgroup $G$ of $\GL_{V}$ such that
$G_{\mathbb{R}{}}\supset h(\mathbb{S}{})$. It is also the smallest algebraic
subgroup $G$ of $\GL_{V}$ such that $G_{\mathbb{C}}\supset\mu_{h}%
(\mathbb{G}_{m})$ (apply (\ref{hq40}), p.~\pageref{hq40}). We usually regard
the Mumford-Tate group as a pair $(G,h)$, and we sometimes denote it by
$\MT_{V}$. Note that $G$ is connected, because otherwise we could replace it
with its identity component. The weight map $w_{h}\colon\mathbb{G}%
_{m}\rightarrow G_{\mathbb{R}{}}$ is defined over $\mathbb{Q}$ and maps into
the centre of $G$.\footnote{Let $Z(w_{h})$ be the centralizer of $w_{h}$ in
$G$. For any $a\in\mathbb{R}^{\times}$, $w_{h}(a)\colon V_{\mathbb{R}{}%
}\rightarrow V_{\mathbb{R}{}}$ is a morphism of real Hodge structures, and so
it commutes with the action of $h(\mathbb{S}{})$. Hence $h(\mathbb{S}%
{})\subset Z(w_{h})_{\mathbb{R}{}}$. As $h$ generates $G$, this implies that
$Z(w_{h})=G$.}

Let $(V,h)$ be a polarizable rational Hodge structure, and let $T^{m,n}$
denote the Hodge structure $V^{\otimes m}\otimes V^{\vee\otimes n}$
($m,n\in\mathbb{N}{}$). By a \emph{Hodge class}%
\index{Hodge class}
of $V$, we mean an element of $V$ of type $(0,0)$, i.e., an element of $V\cap
V^{0,0}$, and by a \emph{Hodge tensor}%
\index{Hodge tensor}
of $V$, we mean a Hodge class of some $T^{m,n}$. The elements of $T^{m,n}$
fixed by the Mumford-Tate group of $V$ are exactly the Hodge tensors, and
$\MT_{V}$ is the largest algebraic subgroup of $\GL_{V}$ fixing all the Hodge
tensors of $V$ (cf. \cite{deligne1982}, 3.4).

The real Hodge structures form a semisimple tannakian category\footnote{For
the theory of tannakian categories, we refer the reader to \cite{deligneM1982}%
. In fact, we shall only need to use the elementary part of the theory (ibid.
\S \S 1,2).} over $\mathbb{R}{}$; the group attached to the category and the
forgetful fibre functor is $\mathbb{S}{}$. The rational Hodge structures form
a tannakian category over $\mathbb{Q}{}$, and the polarizable rational Hodge
structures form a \textit{semisimple} tannakian category, which we denote%
\index{HdgQ@$\Hdg_{\mathbb{Q}{}}$}
$\Hdg_{\mathbb{Q}{}}$. Let $(V,h)$ be a rational Hodge structure, and let
$\langle V,h\rangle^{\otimes}$ be the tannakian subcategory generated by
$(V,h)$. The Mumford-Tate group of $(V,h)$ is the algebraic group attached
$\langle V,h\rangle^{\otimes}$ and the forgetful fibre functor.

Let $G$ and $G^{e}$ respectively denote the Mumford-Tate groups of $V$ and
$V\oplus\mathbb{Q}{}(1)$. The action of $G^{e}$ on $V$ defines a homomorphism
$G^{e}\rightarrow G$, which is an isogeny unless $V$ has weight $0$, in which
case $G^{e}\simeq G\times\mathbb{G}_{m}$. The action of $G^{e}$ on
$\mathbb{Q}{}(1)$ defines a homomorphism $G^{e}\rightarrow\GL_{\mathbb{Q}%
{}(1)}$ whose kernel we denote $G^{1}$ and call the \emph{special Mumford-Tate
group} of $V$. Thus \textrm{\thinspace}$G^{1}\subset\GL_{V}$, and it is the
smallest algebraic subgroup of $\GL_{V}$ such that $G_{\mathbb{R}{}}%
^{1}\supset h(\mathbb{S}{}^{1})$. Clearly $G^{1}\subset G$ and $G=G^{1}\cdot
w_{h}(\mathbb{G}_{m})$.

\begin{proposition}
\label{h53a}Let $G$ be a connected algebraic group over $\mathbb{Q}{}$, and
let $h$ be a homomorphism $\mathbb{S}{}\rightarrow G_{\mathbb{R}{}}$. The pair
$(G,h)$ is the Mumford-Tate group of a Hodge structure if and only if the
weight homomorphism $w_{h}\colon\mathbb{G}_{m}\rightarrow G_{\mathbb{R}{}}$ is
defined over $\mathbb{Q}{}$ and $G$ is generated by $h$%
\index{generated}
(i.e., any algebraic subgroup $H$ of $G$ such that $h(\mathbb{S}{})\subset
H_{\mathbb{R}{}}$ equals $G$).
\end{proposition}

\begin{proof}
If $(G,h)$ is the Mumford-Tate group of a Hodge structure $(V,h)$, then
certainly $h$ generates $G$. The weight homomorphism $w_{h}$ is defined over
$\mathbb{Q}{}$ because $(V,h)$ is a rational Hodge structure.

Conversely, suppose that $(G,h)$ satisfy the conditions. For any faithful
representation $\rho\colon G\rightarrow\GL_{V}$ of $G$, the pair
$(V,h\circ\rho)$ is a rational Hodge structure, and $(G,h)$ is its
Mumford-Tate group.
\end{proof}

\begin{proposition}
\label{h52}Let $(G,h)$ be the Mumford-Tate group of a Hodge structure $(V,h)$.
Then $(V,h)$ is polarizable if and only if $(G,h)$ satisfies (SV2*).
\end{proposition}

\begin{proof}
Let $C=h(i)$. For notational convenience, assume that $(V,h)$ has a single
weight $m$. Let $G^{1}$ be the special Mumford-Tate group of $(V,h)$. Then
$C\in G^{1}(\mathbb{R}{})$, and a pairing $\psi\colon V\times V\rightarrow
\mathbb{Q}(-m)$ is a polarization of the Hodge structure $(V,h)$ if and only
if $(2\pi i)^{m}\psi$ is a $C$-polarization of $V$ for $G^{1}$ in the sense of
\S 2. It follows from (\ref{h20b}) and (\ref{h20c}) that a polarization $\psi$
for $(V,h)$ exists if and only if $\inn(C)$ is a Cartan involution of
$G_{\mathbb{R}{}}^{1}$. Now $G^{1}\subset G$ and the quotient map
$G^{1}\rightarrow G/w_{h}(\mathbb{G}_{m})$ is an isogeny, and so $\inn(C)$ is
a Cartan involution of $G^{1}$ if and only if it is a Cartan involution of
$G/w_{h}(\mathbb{G}_{m})$.
\end{proof}

\begin{corollary}
\label{h53}The Mumford-Tate group of a polarizable Hodge structure is reductive.
\end{corollary}

\begin{proof}
An algebraic group $G$ over $\mathbb{Q}{}$ is reductive if and only if
$G_{\mathbb{R}{}}$ is reductive, and we have already observed that (SV2*)
implies that $G_{\mathbb{R}{}}$ is reductive. Alternatively, polarizable Hodge
structures are semisimple, and an algebraic group in characteristic zero is
reductive if its representations are semisimple (e.g., \cite{deligneM1982}, 2.23).
\end{proof}

\begin{remark}
\label{h52r}Note that (\ref{h52}) implies the following statement: let $(V,h)$
be a Hodge structure; if there exists an algebraic group $G\subset\GL_{V}$
such that $h(\mathbb{S}{})\subset G_{\mathbb{R}{}}$ and $(G,h)$ satisfies
(SV2*), then $(V,h)$ is polarizable.
\end{remark}

\begin{nt}
The Mumford-Tate group%
\index{Mumford-Tate group!of an abelian variety}
of a complex abelian variety $A$ is defined to be the Mumford-Tate group of
the Hodge structure $H_{1}(A^{\text{an}},\mathbb{Q}{})$. In this context,
special Mumford-Tate groups were first introduced in the talk of Mumford
(1966)\nocite{mumford1966} (which is \textquotedblleft partly joint work with
J. Tate\textquotedblright).
\end{nt}

\subsection{Special Hodge structures}

A rational Hodge structure is \emph{special}%
\index{special Hodge structure}%
\footnote{Poor choice of name, since \textquotedblleft
special\textquotedblright\ is overused and special points on Shimura varieties
don't correspond to special Hodge structures, but I can't think of a better
one. Perhaps an \textquotedblleft SV Hodge structure\textquotedblright?} if
its Mumford-Tate group satisfies (SV1,2*) or, equivalently, if it is
polarizable and its Mumford-Tate group satisfies (SV1).

\begin{proposition}
\label{h54a}The special Hodge structures form a tannakian subcategory of
$\Hdg_{\mathbb{Q}}$.
\end{proposition}

\begin{proof}
Let $(V,h)$ be a special Hodge structure. The Mumford-Tate group of any object
in the tannakian subcategory of $\Hdg_{\mathbb{Q}{}}$ generated by $(V,h)$ is
a quotient of $\MT_{V}$, and hence satisfies (SV1,2*).
\end{proof}

Recall that the \emph{level}%
\index{level}
of a Hodge structure $(V,h)$ is the maximum value of $\left\vert
p-q\right\vert $ as $(p,q)$ runs over the pairs $(p,q)$ with $V^{p,q}\neq0$.
It has the same parity as the weight of $(V,h)$.

\begin{example}
\label{h54c}Let $V_{n}(a_{1},\ldots,a_{d})$ denote a complete intersection of
$d$ smooth hypersurfaces of degrees $a_{1},\ldots,a_{d}$ in general position
in $\mathbb{P}{}^{n+d}$ over $\mathbb{C}{}$. Then $H^{n}(V_{n},\mathbb{Q}{})$
has level $\leq1$ only for the varieties $V_{n}(2)$, $V_{n}(2,2)$, $V_{2}(3)$,
$V_{n}(2,2,2)$ ($n$ odd), $V_{3}(3)$, $V_{3}(2,3)$, $V_{5}(3)$, $V_{3}(4)$
(\cite{rapoport1972}).
\end{example}

\begin{proposition}
\label{h54b}Every polarizable Hodge structure of level $\leq1$ is special.
\end{proposition}

\begin{proof}
A Hodge structure of level $0$ is direct sum of copies of $\mathbb{Q}{}(m)$
for some $m$, and so its Mumford-Tate group is $\mathbb{G}_{m}$. A Hodge
structure $(V,h)$ of level $1$ is of type $\{(p,p+1),(p+1,p)\}$ for some $p$.
Then%
\[
\Lie(\MT_{V})\subset\End(V)=V^{\vee}\otimes V,
\]
which is of type $\{(-1,1),(0,0),(1,-1)\}$.
\end{proof}

\begin{example}
\label{h54f}Let $A$ be an abelian variety over $\mathbb{C}{}$. The Hodge
structures $H_{B}^{n}(A)$ are special for all $n$. To see this, note that
$H_{B}^{1}(A)$ is of level $1$, and hence is special by (\ref{h54b}), and that%
\[
H_{B}^{n}(A)\simeq\bigwedge\nolimits^{n}H_{B}^{1}(A)\subset H_{B}%
^{1}(A)^{\otimes n},
\]
and hence $H_{B}^{n}(A)$ is special by (\ref{h54a}).
\end{example}

It follows that a nonspecial Hodge structure does not lie in the tannakian
subcategory of $\Hdg_{\mathbb{Q}{}}$ generated by the cohomology groups of
abelian varieties.

\begin{proposition}
\label{h54e}A pair $(G,h)$ is the Mumford-Tate group of a special Hodge
structure if and only if $h$ satisfies (SV1,2*), the weight $w_{h}$ is defined
over $\mathbb{Q}{}$, and $G$ is generated by $h$.
\end{proposition}

\begin{proof}
Immediate consequence of Proposition \ref{h53a}, and of the definition of a
special Hodge structure.
\end{proof}

Note that, because $h$ generates $G$, it also satisfies (SV3).

\begin{example}
\label{h54g}Let $f\colon X\rightarrow S$ be the universal family of smooth
hypersurfaces of a fixed degree $\delta$ and of a fixed odd dimension $n$. For
$s$ outside a meagre subset of $S$, the Mumford-Tate group of $H^{n}%
(X_{s},\mathbb{Q}{})$ is the full group of symplectic similitudes (see
\ref{h54d} below). This implies that $H^{n}(X_{s},\mathbb{Q}{})$ is not
special unless it has level $\leq1$. According to (\ref{h54c}), this rarely happens.
\end{example}

\subsection{The generic Mumford-Tate group}

Throughout this subsection, $(\mathsf{V},F)$ is a family of Hodge structures
on a connected complex manifold $S$. Recall that \textquotedblleft
family\textquotedblright\ means \textquotedblleft holomorphic
family\textquotedblright.

\begin{lemma}
\label{h54}For any $t\in\Gamma(S,\mathsf{V})$, the set
\[
Z(t)=\{s\in S\mid t_{s}\text{ is of type }(0,0)\text{ in }\mathsf{V}_{s}\}
\]
is an analytic subset of $S$.
\end{lemma}

\begin{proof}
An element of $\mathsf{V}_{s}$ is of type $(0,0)$ if and only if it lies in
$F^{0}\mathsf{V}_{s}$. On $S$, we have an exact sequence%
\[
0\rightarrow F^{0}\mathcal{V}\rightarrow\mathcal{V}\rightarrow\mathcal{Q}%
{}\rightarrow0
\]
of locally free sheaves of $\mathcal{O}{}_{S}$-modules. Let $U$ be an open
subset of $S$ such that $\mathcal{Q}{}$ is free over $U$. Choose an
isomorphism $\mathcal{Q}{}\simeq\mathcal{O}{}_{U}^{r}$, and let $t|U$ map to
$(t_{1},\ldots,t_{r})$ in $\mathcal{O}{}_{U}^{r}$. Then
\[
Z(t)\cap U=\{s\in U\mid t_{1}(s)=\cdots=t_{r}(s)=0\}.
\]

\end{proof}

For $m,n\in\mathbb{N}{}$, let $\mathsf{T}^{m,n}=\mathsf{T}^{m,n}\mathsf{V}$ be
the family of Hodge structures $\mathsf{V}^{\otimes m}\otimes\mathsf{V}%
^{\vee\otimes n}$ on $S$. Let $\pi\colon\tilde{S}\rightarrow S$ be a universal
covering space of $S$, and define%
\begin{equation}
\mathring{S}=S\smallsetminus\bigcup\nolimits_{t}\pi_{\ast}(Z(t)) \label{hq22}%
\end{equation}
where $t$ runs over the global sections of the local systems $\pi^{\ast
}\mathsf{T}^{m,n}$ ($m,n\in\mathbb{N}{}$) such that $\pi_{\ast}(Z(t))\neq S$.
Thus $\mathring{S}$ is the complement in $S$ of a countable union of proper
analytic subvarieties --- we shall call such a subset \emph{meagre}%
\index{meagre}%
.

\begin{example}
\label{h55a}For a \textquotedblleft general\textquotedblright\ abelian variety
of dimension $g$ over $\mathbb{C}{}$, it is known that the $\mathbb{Q}{}%
$-algebra of Hodge classes is generated by the class of an ample divisor class
(\cite{comessatti1934}, \cite{mattuck1958}). It follows that the same is true
for all abelian varieties in the subset $\mathring{S}$ of the moduli space
$S$. The Hodge conjecture obviously holds for these abelian varieties.
\end{example}

Let $t$ be a section of $\mathsf{T}^{m,n}$ over an open subset $U$ of
$\mathring{S}$; if $t$ is a Hodge class in $\mathsf{T}_{s}^{m,n}$ for one
$s\in U$, then it is Hodge tensor for every $s\in U$. Thus, there exists a
local system of $\mathbb{Q}{}$-subspaces $H\mathsf{T}^{m,n}$ on $\mathring{S}$
such that $\left(  H\mathsf{T}^{m,n}\right)  _{s}$ is the space of Hodge
classes in $\mathsf{T}_{s}^{m,n}$ for each $s$. Since the Mumford-Tate group
of $(\mathsf{V}_{s},F_{s})$ is the largest algebraic subgroup of $\GL_{V_{s}}$
fixing the Hodge tensors in the spaces $\mathsf{T}_{s}^{m,n}$, we have the
following result.

\begin{proposition}
\label{h55}Let $G_{s}$ be the Mumford-Tate group of $(\mathsf{V}_{s},F_{s})$.
Then $G_{s}$ is locally constant on $\mathring{S}$.
\end{proposition}

More precisely:\bquote Let $U$ be an open subset of $S$ on which $\mathsf{V}$
is constant, say, $\mathsf{V}=V_{U}$; identify the stalk $\mathsf{V}_{s}$
($s\in U$) with $V$, so that $G_{s}$ is a subgroup of $\GL_{V}$; then $G_{s}$
is constant for $s\in U\cap\mathring{S}$, say $G_{s}=G$, and $G\supset G_{s}$
for all $s\in U\smallsetminus(U\cap\mathring{S})$. \equote

\begin{plain}
\label{generic}We say that $G_{s}$ is \emph{generic}%
\index{generic Mumford-Tate group}
if $s\in\mathring{S}$. Suppose that $\mathsf{V}$ is constant, say
$\mathsf{V}=V_{S}$, and let $G=G_{s_{0}}\subset\GL_{V}$ be generic. By
definition, $G$ is the smallest algebraic subgroup of $\GL_{V}$ such that
$G_{\mathbb{R}{}}$ contains $h_{s_{0}}(\mathbb{S}{})$. As $G\supset G_{s}$ for
all $s\in S$, the generic Mumford-Tate group of $(\mathsf{V},F)$ is the
smallest algebraic subgroup $G$ of $\GL_{V}$ such that $G_{\mathbb{R}{}}$
contains $h_{s}(\mathbb{S}{})$ for \textit{all} $s\in S$.

Let $\pi\colon\tilde{S}\rightarrow S$ be a universal covering of $S$, and fix
a trivialization $\pi^{\ast}\mathsf{V}\simeq V_{S}$ of $\mathsf{V}$. Then, for
each $s\in S$, there are given isomorphisms%
\begin{equation}
V\simeq(\pi^{\ast}\mathsf{V})_{s}\simeq\mathsf{V}_{\pi s}\text{.} \label{hq23}%
\end{equation}
There is an algebraic subgroup $G$ of $\GL_{V}$ such that, for each $s\in
\pi^{-1}(\mathring{S})$, $G$ maps isomorphically onto $G_{s}$ under the
isomorphism $\GL_{V}\simeq\GL_{\mathsf{V}_{\pi s}}$ defined by (\ref{hq23}).
It is the smallest algebraic subgroup of $\GL_{V}$ such that $G_{\mathbb{R}{}%
}$ contains the image of $h_{s}\colon\mathbb{S}{}\rightarrow\GL_{V_{\mathbb{R}%
{}}}$ for all $s\in\tilde{S}$.
\end{plain}

\begin{aside}
\label{h59g}For a polarizable integral variation of Hodge structures on a
smooth algebraic variety $S$, Cattani, Deligne, and Kaplan (1995, Corollary
1.3)\nocite{cattaniDK1995} show that the sets $\pi_{\ast}(Z(t))$ in
(\ref{hq22}) are algebraic subvarieties of $S$. This answered a question of
Weil 1977\nocite{weil1977}.
\end{aside}

\subsection{Variation of Mumford-Tate groups in families}

\begin{definition}
\label{h55b}Let $(\mathsf{V},F)$ be a family of Hodge structures on a
connected complex manifold $S$.

\begin{enumerate}
\item An \emph{integral structure}%
\index{integral structure}
on $(\mathsf{V},F)$ is a local system of $\mathbb{Z}{}$-modules $\Lambda
\subset\mathsf{V}$ such that $\mathbb{Q}{}\otimes_{\mathbb{Z}{}}\Lambda
\simeq\mathsf{V}$.

\item The family $(\mathsf{V},F)$ is said to \emph{satisfy the theorem of the
fixed part} if, for every finite covering $a\colon S^{\prime}\rightarrow S$ of
$S$, there is a Hodge structure on the $\mathbb{Q}{}$-vector space
$\Gamma(S^{\prime},a^{\ast}\mathsf{V})$ such that, for all $s\in S^{\prime}$,
the canonical map $\Gamma(S^{\prime},a^{\ast}\mathsf{V})\rightarrow a^{\ast
}\mathsf{V}_{s}$ is a morphism of Hodge structures, or, in other words, if the
largest constant local subsystem $\mathsf{V}^{f}$ of $a^{\ast}\mathsf{V}$ is a
constant family of Hodge substructures of $a^{\ast}\mathsf{V}$.

\item The \emph{algebraic monodromy group}%
\index{algebraic monodromy group}
at point $s\in S$ is the smallest algebraic subgroup of $\GL_{\mathsf{V}_{s}}$
containing the image of the monodromy homomorphism $\pi_{1}(S,s)\rightarrow
\GL(\mathsf{V}_{s})$. Its identity connected component is called the
\emph{connected monodromy group }$M_{s}$ at $s$. In other words, $M_{s}$ is
the smallest connected algebraic subgroup of $\GL_{V_{s}}$ such that
$M_{s}(\mathbb{Q}{})$ contains the image of a subgroup of $\pi_{1}(S,s)$ of
finite index.
\end{enumerate}
\end{definition}

\begin{plain}
\noindent Let $\pi\colon\tilde{S}\rightarrow S$ be the universal covering of
$S$, and let $\Gamma$ be the group of covering transformations of $\tilde
{S}/S$. The choice of a point $s\in\tilde{S}$ determines an isomorphism
$\Gamma\simeq\pi_{1}(S,\pi s)$. Now choose a trivialization $\pi^{\ast
}\mathsf{V}\approx V_{\tilde{S}}$. The choice of a point $s\in\tilde{S}$
determines an isomorphism $V\simeq V_{\pi(s)}$. There is an action of $\Gamma$
on $V$ such that, for each $s\in\tilde{S}$, the diagram%
\[
\begin{tikzpicture}
\node (P0) at (0,0cm) {$\Gamma$};
\node (P1) at (1.0,0cm) {$\times$} ;
\node (P2) at (1.6,0cm) {$V$};
\node (P3) at (3.6,0cm) {$V$};
\node (P4) [below=of P0] {$\pi_1(S,\pi s)$};
\node (P5) [below=1.1cm of P1] {$\times$};
\node (P6) [below=of P2] {$V_s$};
\node (P7) [below=of P3] {$V_s$};
\draw[->,font=\scriptsize,>=angle 90]
(P2) edge (P3)
(P6) edge (P7)
(P0) edge node[right] {$\simeq$} (P4)
(P2) edge node[right] {$\simeq$} (P6)
(P3) edge node[right] {$\simeq$} (P7);
\end{tikzpicture}
\]
commutes. Let $M$ be the smallest connected algebraic subgroup of $\GL_{V}$
such $M(\mathbb{Q}{})$ contains a subgroup of $\Gamma$ of finite index; in
other words,%
\[
M=\bigcap\{H\subset\GL_{V}\mid H\text{ connected, }(\Gamma\colon
H(\mathbb{Q}{})\cap\Gamma)<\infty\}.
\]
Under the isomorphism $V\simeq V_{\pi s}$ defined by $s\in S$, $M$ maps
isomorphically onto $M_{s}$.
\end{plain}

\begin{theorem}
\label{h56}Let $(\mathsf{V},F)$ be a polarizable family of Hodge structures on
a connected complex manifold $S$, and assume that $(\mathsf{V},F)$ admits an
integral structure. Let $G_{s}$ (resp. $M_{s}$) denote the Mumford-Tate (resp.
the connected monodromy group) at $s\in S$.

\begin{enumerate}
\item For all $s\in\mathring{S}$, $M_{s}\subset G_{s}^{\mathrm{der}}$.

\item If $\mathsf{T}^{m,n}$ satisfies the theorem of the fixed part for all
$m,n$, then $M_{s}$ is normal in $G_{s}^{\mathrm{der}}$ for all $s\in
\mathring{S}$; moreover, if $G_{s^{\prime}}$ is commutative for some
$s^{\prime}\in S$, then $M_{s}=G_{s}^{\mathrm{der}}$ for all $s\in\mathring
{S}$.
\end{enumerate}
\end{theorem}

The theorem was proved by Deligne (see \cite{deligne1972}, 7.5;
\cite{zarhin1984}, 7.3) except for the second statement of (b), which is
Proposition 2 of \cite{andre1992mt}. The proof of the theorem will occupy the
rest of this subsection.

\begin{example}
\label{h54d}Let $f\colon X\rightarrow\mathbb{P}^{1}$ be a Lefschetz pencil
over $\mathbb{C}$ of hypersurfaces of fixed degree and odd dimension $n$, and
let $S$ be the open subset of $\mathbb{P}{}^{1}$ where $X_{s}$ is smooth. Let
$(\mathsf{V},F)$ be the variation of Hodge structures $R^{n}f_{\ast}%
\mathbb{Q}{}$ on $S$. The action of $\pi_{1}(S,s)$ on $\mathsf{V}_{s}%
=H^{n}(X_{s}^{\text{an}},\mathbb{Q}{})$ preserves the cup-product form on
$\mathsf{V}_{s}$, and a theorem of Kazhdan and Margulis (\cite{deligne1974},
5.10) says that the image of $\pi_{1}(S,s)$ is Zariski-dense in the symplectic
group. It follows that the generic Mumford-Tate group $G_{s}$ is the full
group of symplectic similitudes. This implies that, for $s\in\mathring{S}$,
the Hodge structure $\mathsf{V}_{s}$ is not special unless it has level
$\leq1$.
\end{example}

\subsubsection{Proof of \textnf{(a)} of Theorem \ref{h56}}

We first show that $M_{s}\subset G_{s}$ for $s\in\mathring{S}$. Recall that on
$\mathring{S}$ there is a local system of $\mathbb{Q}{}$-vector spaces
$H\mathsf{T}^{m,n}\subset\mathsf{T}^{m,n}$ such that $H\mathsf{T}_{s}^{m,n}$
is the space of Hodge tensors in $\mathsf{T}_{s}^{m,n}$. The fundamental group
$\pi_{1}(S,s)$ acts on $H\mathsf{T}_{s}^{m,n}$ through a discrete subgroup of
$\GL(H\mathsf{T}_{s}^{m,n})$ (because it preserves a lattice in $\mathsf{T}%
_{s}^{m,n}$), and it preserves a positive definite quadratic form on
$H\mathsf{T}_{s}^{m,n}$. It therefore acts on $H\mathsf{T}_{s}^{m,n}$ through
a finite quotient. As $G_{s}$ is the algebraic subgroup of $\GL_{\mathsf{V}%
_{s}}$ fixing the Hodge tensors in some finite direct sum of spaces
$\mathsf{T}_{s}^{m,n}$, this shows that the image of some finite index
subgroup of $\pi_{1}(S,s)$ is contained in $G_{s}(\mathbb{Q}{})$. Hence
$M_{s}\subset G_{s}$.

We next show that $M_{s}$ is contained in the special Mumford-Tate group
$G_{s}^{1}$ at $s$. Consider the family of Hodge structures $\mathsf{V}%
\oplus\mathbb{Q}{}(1)$, and let $G_{s}^{e}$ be its Mumford-Tate group at $s$.
As $\mathsf{V}\oplus\mathbb{Q}{}(1)$ is polarizable and admits an integral
structure, its connected monodromy group $M_{s}^{e}$ at $s$ is contained in
$G_{s}^{e}$. As $\mathbb{Q}{}(1)$ is a constant family, $M_{s}^{e}%
\subset\Ker(G_{s}^{e}\rightarrow\GL_{\mathbb{Q}{}(1)})=G_{s}^{1}$. Therefore
$M_{s}=M_{s}^{e}\subset G_{s}^{1}$.

There exists an object $W$ in $\Rep_{\mathbb{Q}{}}G_{s}\simeq\langle
V_{s}\rangle^{\otimes}\subset\Hdg_{\mathbb{Q}{}}$ such that $G_{s}%
^{\mathrm{der}}\cdot w_{h_{s}}(\mathbb{G}_{m})$ is the kernel of
$G_{s}\rightarrow\GL_{W}$. The Hodge structure $W$ admits an integral
structure, and its Mumford-Tate group is $G^{\prime}\simeq G_{s}/\left(
G_{s}^{\mathrm{der}}\cdot w_{h_{s}}(\mathbb{G}_{m})\right)  $. As $W$ has
weight $0$ and $G^{\prime}$ is commutative, we find from (\ref{h52}) that
$G^{\prime}(\mathbb{R}{})$ is compact. As the action of $\pi_{1}(S,s)$ on $W$
preserves a lattice, its image in $G^{\prime}(\mathbb{R}{})$ must be discrete,
and hence finite. This shows that
\[
M_{s}\subset\left(  G_{s}^{\mathrm{der}}\cdot w_{h_{s}}(\mathbb{G}%
_{m})\right)  \cap G_{s}^{1}=G_{s}^{\mathrm{der}}.
\]

\subsubsection{Proof of the first statement of \textnf{(b)} of Theorem
\ref{h56}}

We first prove two lemmas.

\begin{lemma}
\label{h58}Let $V$ be a $\mathbb{Q}{}$-vector space, and let $H\subset G$ be
algebraic subgroups of $\GL_{V}$. Assume:

\begin{enumerate}
\item the action of $H$ on any $H$-stable line in a finite direct sum of
spaces $T^{m,n}$ is trivial;

\item $(T^{m,n})^{H}$ is $G$-stable for all $m,n\in\mathbb{N}{}$.
\end{enumerate}

\noindent Then $H$ is normal in $G$.
\end{lemma}

\begin{proof}
There exists a line $L$ in some finite direct sum $T$ of spaces $T^{m,n}$ such
that $H$ is the stabilizer of $L$ in $\GL_{V}$ (Chevalley's theorem,
\cite{deligne1982}, 3.1a,b). According to (a), $H$ acts trivially on $L$. Let
$W$ be the intersection of the $G$-stable subspaces of $T$ containing $L$.
Then $W\subset T^{H}$ because $T^{H}$ is $G$-stable by (b). Let $\varphi$ be
the homomorphism $G\rightarrow\GL_{W^{\vee}\otimes W}$ defined by the action
of $G$ on $W$. As $H$ acts trivially on $W$, it is contained in the kernel of
$\varphi$. On the other hand, the elements of the kernel of $\varphi$ act as
scalars on $W$, and so stabilize $L$. Therefore $H=\Ker(\varphi)$, which is
normal in $G$.
\end{proof}

\begin{lemma}
\label{h57}Let $(\mathsf{V},F)$ be a polarizable family of Hodge structures on
a connected complex manifold $S$. Let $\mathsf{L}$ be a local system of
$\mathbb{Q}{}$-vector spaces on $S$ contained in a finite direct sum of local
systems $\mathsf{T}^{m,n}$. If $(\mathsf{V},F)$ admits an integral structure
and $\mathsf{L}$ has dimension $1$, then $M_{s}$ acts trivially on
$\mathsf{L}_{s}$.
\end{lemma}

\begin{proof}
The hypotheses imply that $\mathsf{L}$ also admits an integral structure, and
so $\pi_{1}(S,s)$ acts through the finite subgroup $\{\pm1\}$ of
$\GL_{\mathsf{L}_{s}}$. This implies that $M_{s}$ acts trivially on
$\mathsf{L}_{s}$.
\end{proof}

We now prove the first part of (b) of the theorem. Let $s\in\mathring{S}$; we
shall apply Lemma \ref{h58} to $M_{s}\subset G_{s}\subset\GL_{\mathsf{V}_{s}}%
$. After passing to a finite covering of $S$, we may suppose that $\pi
_{1}(S,s)\subset M_{s}(\mathbb{Q}{})$. Any $M_{s}$-stable line in
$\bigoplus\nolimits_{m,n}\mathsf{T}_{s}^{m,n}$ is of the form $\mathsf{L}_{s}$
for a local subsystem $\mathsf{L}$ of $\bigoplus\nolimits_{m,n}\mathsf{T}%
_{s}^{m,n}$, and so hypothesis (a) of Lemma \ref{h58} follows from
(\ref{h57}). It remains to show $(\mathsf{T}_{s}^{m,n})^{M_{s}}$ is stable
under $G_{s}$. Let $H$ be the stablizer of $(\mathsf{T}_{s}^{m,n})^{M_{s}}$ in
$\GL_{\mathsf{T}_{s}^{m,n}}$. Because $\mathsf{T}^{m,n}$ satisfies the theorem
of the fixed part, $(\mathsf{T}_{s}^{m,n})^{M_{s}}$ is a Hodge substructure of
$\mathsf{T}_{s}^{m,n}$, and so $(\mathsf{T}_{s}^{m,n})_{\mathbb{R}{}}^{M_{s}}$
is stable under $h(\mathbb{S}{})$. Therefore $h(\mathbb{S}{})\subset
H_{\mathbb{R}{}}$, and this implies that $G_{s}\subset H$.

\subsubsection{Proof of the second statement of \textnf{(b)} of Theorem
\ref{h56}}

We first prove a lemma.

\begin{lemma}
\label{h59d}Let $(\mathsf{V},F)$ be a variation of polarizable Hodge
structures on a connected complex manifold $S$. Assume:

\begin{enumerate}
\item $M_{s}$ is normal in $G_{s}$ for all $s\in\mathring{S}$;

\item $\pi_{1}(S,s)\subset M_{s}(\mathbb{Q)}$ for one (hence every) $s\in S$;

\item $(\mathsf{V},F)$ satisfies the theorem of the fixed part.
\end{enumerate}

\noindent Then the subspace $\Gamma(S,\mathsf{V})$ of $\mathsf{V}_{s}$ is
stable under $G_{s}$, and the image of $G_{s}$ in $\GL_{\Gamma(S,\mathsf{V})}$
is independent of $s\in S$.
\end{lemma}

\noindent In fact, (c) implies that $\Gamma(S,\mathsf{V})$ has a well-defined
Hodge structure, and we shall show that the image of $G_{s}$ in $\GL_{\Gamma
(S,\mathsf{V})}$ is the Mumford-Tate group of $\Gamma(S,\mathsf{V})$.

\begin{proof}
We begin with observation: let $G$ be the affine group scheme attached to the
tannakian category $\Hdg_{\mathbb{Q}{}}$ and the forgetful fibre functor; for
any $(V,h_{V})$ in $\Hdg_{\mathbb{Q}{}}$, $G$ acts on $V$ through a surjective
homomorphism $G\rightarrow\MT_{V}$; therefore, for any $(W,h_{W})$ in $\langle
V,h_{V}\rangle^{\otimes}$, $\MT_{V}$ acts on $W$ through a surjective
homomorphism $\MT_{V}\rightarrow\MT_{W}$.

For every $s\in S$,
\[
\Gamma(S,\mathsf{V})=\Gamma(S,\mathsf{V}^{f})=(\mathsf{V}^{f})_{s}%
=\mathsf{V}_{s}^{\pi_{1}(S,s)}\overset{\text{(b)}}{=}\mathsf{V}_{s}^{M_{s}}.
\]
The subspace $\mathsf{V}_{s}^{M_{s}}$ of $\mathsf{V}_{s}$ is stable under
$G_{s}$ when $s\in\mathring{S}$ because then $M_{s}$ is normal in $G_{s}$, and
it is stable under $G_{s}$ when $s\notin\mathring{S}$ because then $G_{s}$ is
contained in some generic Mumford-Tate group. Because $(\mathsf{V},F)$
satisfies the theorem of the fixed part, $\Gamma(S,\mathsf{V})$ has a Hodge
structure (independent of $s$) for which the inclusion $\Gamma(S,\mathsf{V}%
)\rightarrow\mathsf{V}_{s}$ is a morphism of Hodge structures. From the
observation, we see that the image of $G_{s}$ in $\GL_{\Gamma(S,\mathsf{V})}$
is the Mumford-Tate group of $\Gamma(S,\mathsf{V})$, which does not depend on
$s$.
\end{proof}

We now prove that $M_{s}=G_{s}^{\mathrm{der}}$ when some Mumford-Tate group
$G_{s^{\prime}}$ is commutative. We know that $M_{s}$ is a normal subgroup of
$G_{s}^{\mathrm{der}}$ for $s\in\mathring{S}$, and so it remains to show that
$G_{s}/M_{s}$ is commutative for $s\in\mathring{S}$ under the hypothesis.

We begin with a remark. Let $N$ be a normal algebraic subgroup of an algebraic
group $G$. The category of representations of $G/N$ can be identified with the
category of representations of $G$ on which $N$ acts trivially. Therefore, to
show that $G/N$ is commutative, it suffices to show that $G$ acts through a
commutative quotient on every $V$ on which $N$ acts trivially. If $G$ is
reductive and we are in characteristic zero, then it suffices to show that,
for one faithful representation $V$ of $G$, the group $G$ acts through a
commutative quotient on $(T^{m,n})^{N}$ for all $m,n\in\mathbb{N}{}$.

Let $\mathsf{T}=\mathsf{T}^{m,n}$. According to the remark, it suffices to
show that, for $s\in\mathring{S}$, $G_{s}$ acts on $\mathsf{T}_{s}^{M_{s}}$
through a commutative quotient. This will follow from the hypothesis, once we
check that $\mathsf{T}$ satisfies the hypotheses of Lemma \ref{h59d}.
Certainly, $M_{s}$ is a normal subgroup of $G_{s}$ for $s\in\mathring{S}$, and
$\pi_{1}(S,s)$ will be contained in $M_{s}$ once we have passed to a finite
cover. Finally, we are assuming that $\mathsf{T}$ satisfies the theorem of the
fixed part.

\subsection{Variation of Mumford-Tate groups in algebraic families}

When the underlying manifold is an algebraic variety, we have the following theorem.

\begin{theorem}
[Griffiths, Schmid]\label{h59a}A variation of Hodge structures on a smooth
algebraic variety over $\mathbb{C}{}$ satisfies the theorem of the fixed part
if it is polarizable and admits an integral structure.
\end{theorem}

\begin{proof}
When the variation of Hodge structures arises from a projective smooth map
$X\rightarrow S$ of algebraic varieties and $S$ is complete, this is the
original theorem of the fixed part (\cite{griffiths1970}, \S 7). In the
general case it is proved in \cite{schmid1973}, 7.22. See also
\cite{deligne1971h}, 4.1.2 and the footnote on p.~45.
\end{proof}

\begin{theorem}
\label{h59b}Let $(\mathsf{V},F)$ be a variation of Hodge structures on a
connected smooth complex algebraic variety $S$. If $(\mathsf{V},F)$ is
polarizable and admits an integral structure, then $M_{s}$ is a normal
subgroup of $G_{s}^{\mathrm{der}}$ for all $s\in\mathring{S}$, and the two
groups are equal if $G_{s}$ is commutative for some $s\in S$.
\end{theorem}

\begin{proof}
If $(\mathsf{V},F)$ is polarizable and admits an integral structure, then
$\mathsf{T}^{m,n}$ is polarizable and admits an integral structure, and so it
satisfies the theorem of the fixed part (Theorem \ref{h59a}). Now the theorem
follows from Theorem \ref{h56}.
\end{proof}

\section{Period subdomains}

\begin{quote}
{\small We define the notion of a period subdomain, and we show that the
hermitian symmetric domains are exactly the period subdomains on which the
universal family of Hodge structures is a \textit{variation} of Hodge
structures.}
\end{quote}

\subsection{Flag manifolds}

Let $V$ be a complex vector space and let $\mathbf{d}=(d_{1},\ldots,d_{r})$ be
a sequence of integers with $\dim V>d_{1}>\cdots>d_{r}>0$. The \emph{flag
manifold} $\Gr_{\mathbf{d}}(V)$ has as points the filtrations
\[
V\supset F^{1}V\supset\cdots\supset F^{r}V\supset0,\quad\quad\dim F^{i}%
V=d_{i}\text{.}%
\]
It is a projective complex manifold, and the tangent space to $\Gr_{\mathbf{d}%
}(V)$ at the point corresponding to a filtration $F$ is%
\[
T_{F}(\Gr_{\mathbf{d}}(V))\simeq\End(V)/F^{0}\End(V)
\]
where
\[
F^{j}\End(V)=\{\alpha\in\End(V)\mid\alpha(F^{i}V)\subset F^{i+j}V\text{ for
all }i\}.
\]

\begin{theorem}
\label{h61}Let $V_{S}$ be the constant sheaf on a connected complex manifold
$S$ defined by a real vector space $V$, and let $(V_{S},F)$ be a family of
Hodge structures on $S$. Let $\mathbf{d}$ be the sequence of ranks of the
subsheaves in $F{}$.

\begin{enumerate}
\item The map $\varphi\colon S\rightarrow\Gr_{\mathbf{d}}(V_{\mathbb{C}{}})$
sending a point $s$ of $S$ to the point of $\Gr_{\mathbf{d}}(V_{\mathbb{C}{}%
})$ corresponding to the filtration $F{}_{s}$ on $V$ is holomorphic.

\item The family $(V_{S},F{})$ satisfies Griffiths transversality if and only
if the image of the map%
\[
(d\varphi)_{s}\colon T_{s}S\rightarrow T_{\varphi(s)}\Gr_{\mathbf{d}%
}(V_{\mathbb{C}{}})
\]
lies in the subspace $F_{s}^{-1}\End(V_{\mathbb{C}{}})/F_{s}^{0}%
\End(V_{\mathbb{C}{}})$ of $\End(V_{\mathbb{C}{}})/F_{s}^{0}\End(V_{\mathbb{C}%
{}})$ for all $s\in S$.
\end{enumerate}
\end{theorem}

\begin{proof}
Statement (a) simply says that the filtration is holomorphic, and (b) restates
the definition of Griffiths transversality.
\end{proof}

\subsection{Period domains}

We now fix a real vector space $V$, a Hodge filtration $F_{0}$ on $V$ of
weight $m$, and a polarization $t_{0}\colon V\times V\rightarrow\mathbb{R}%
{}(m)$ of the Hodge structure $(V,F_{0})$.

Let $D=D(V,F_{0},t_{0})$ be the set of Hodge filtrations $F$ on $V$ of weight
$m$ with the same Hodge numbers as $(V,F_{0})$ for which $t_{0}$ is a
polarization. Thus $D$ is the set of descending filtrations%
\[
V_{\mathbb{C}{}}\supset\cdots\supset F^{p}\supset F^{p+1}\supset\cdots
\supset0
\]
on $V_{\mathbb{C}{}}$ such that

\begin{enumerate}
\item $\dim_{\mathbb{C}{}}F^{p}=\dim_{\mathbb{C}{}}F_{0}^{p}$ for all $p$,

\item $V_{\mathbb{C}{}}=F^{p}\oplus\overline{F^{q}}$ whenever $p+q=m+1$,

\item $t_{0}(F^{p},F^{q})=0$ whenever $p+q=m+1$, and

\item $\left(  2\pi i\right)  ^{m}t_{0\mathbb{C}{}}(v,C\bar{v})>0$ for all
nonzero elements $v$ of $V_{\mathbb{C}}$.
\end{enumerate}

\noindent Condition (b) requires that $F$ be a Hodge filtration of weight $m$,
condition (a) requires that $(V,F)$ have the same Hodge numbers as $(V,F_{0}%
)$, and the conditions (c) and (d) require that $t_{0}$ be a polarization.

Let $D^{\vee}=D^{\vee}(V,F_{0},t_{0})$ be the set of filtrations on
$V_{\mathbb{C}{}}$ satisfying (a) and (c).

\begin{theorem}
\label{h62}The set $D^{\vee}$ is a compact complex submanifold of
$\Gr_{\mathbf{d}}(V)$, and $D$ is an open submanifold of $D^{\vee}$.
\end{theorem}

\begin{proof}
We first remark that, in the presence of (a), condition (c) requires that
$F^{m+1-p}$ be the orthogonal complement of $F^{p}$ for all $p$. In
particular, each of $F^{p}$ and $F^{m+1-p}$ determines the other.

When $m$ is odd, $t_{0}$ is alternating, and the remark shows that $D^{\vee}$
can be identified with the set of filtrations%
\[
V_{\mathbb{C}{}}\supset F^{(m+1)/2}\supset F^{(m+3)/2}\supset\cdots\supset0
\]
satisfying (a) and such that $F^{(m+1)/2}$ is totally isotropic for $t_{0}$.
Let $S$ be the symplectic group for $t_{0}$. Then $S(\mathbb{C}{})$ acts
transitively on these filtrations, and the stabilizer $P$ of the filtration
$F_{0}$ is a parabolic subgroup of $S$. Therefore $S(\mathbb{C}{}%
)/P(\mathbb{C}{})$ is a compact complex manifold, and the bijection
$S(\mathbb{C}{})/P(\mathbb{C}{})\simeq D^{\vee}$ is holomorphic. The proof
when $m$ is even is similar.

The submanifold $D$ of $D^{\vee}$ is open because the conditions (b) and (d)
are open.
\end{proof}

The complex manifold $D=D(V,F_{0},t_{0})$ is the (Griffiths) \emph{period
domain}%
\index{period domain}
defined by $(V,F_{0},t_{0})$.

\begin{theorem}
\label{h63}Let $(\mathsf{V},F,t)$ be a polarized family of Hodge structures on
a complex manifold $S$. Let $U$ be an open connected subset of $S$ on which
the local system $V$ is trivial, and choose an isomorphism $\mathsf{V}|U\simeq
V_{U}$ and a point $o\in U$. The map $\mathcal{P}{}\colon U\rightarrow
D(V,F_{o},t_{o})$ sending a point $s\in U$ to the point $(V_{s},F_{s},t_{s})$
is holomorphic.
\end{theorem}

\begin{proof}
The map $s\mapsto F_{s}\colon U\rightarrow\Gr_{\mathbf{d}}(V)$ is holomorphic
by (\ref{h61}) and it takes values in $D$. As $D$ is a complex submanifold of
$\Gr_{\mathbf{d}}(V)$ this implies that the map $U\rightarrow D$ is
holomorphic (\cite{grauertR1984}, 4.3.3).
\end{proof}

The map $\mathcal{P}{}$ is called the \emph{period map}%
\index{period map}%
.

The constant local system of real vector spaces $V_{D}$ on $D$ becomes a
polarized family of Hodge structures on $D$ in an obvious way (called the
\emph{universal family})%
\index{universal family}%

\begin{theorem}
\label{h64}If the universal family of Hodge structures on $D=D(V,F_{0},t_{0})$
satisfies Griffiths transversality, then $D$ is a hermitian symmetric domain.
\end{theorem}

\begin{proof}
Let $h_{0}\colon\mathbb{S}{}\rightarrow\GL_{V}$ be the homomorphism
corresponding to the Hodge filtration $F_{0}$, and let $G$ be the algebraic
subgroup of $\GL_{V}$ whose elements fix $t_{0}$ up to scalar. Then $h_{0}$
maps into $G$, and $h_{0}\circ w$ maps into its centre (recall that $V$ has a
single weight $m$). Therefore (see \ref{h90a}), there exists a homomorphism
$u_{0}\colon\mathbb{S}{}^{1}\rightarrow G^{\mathrm{ad}}$ such that
$h_{0}(z)=u_{0}(z/\bar{z})$ mod $Z(G)(\mathbb{\mathbb{R}{}})$.

Let $o$ be the point $F_{0}$ of $D$, and let $\mathfrak{g}{}$ denote $\Lie G$
with the Hodge structure provided by $\Ad\circ h_{0}$. Then
\[
\mathfrak{g}{}_{\mathbb{C}{}}/\mathfrak{g}{}^{00}\simeq T_{o}(D)\subset
T_{o}(\Gr_{\mathbf{d}}(V))\simeq\End(V)/F^{0}\End(V).
\]
If the universal family of Hodge structures satisfies Griffiths
transversality, then $\mathfrak{g}{}_{\mathbb{C}{}}=F^{-1}\mathfrak{g}%
{}_{\mathbb{C}{}}$ (by \ref{h61}b). As $\mathfrak{g}{}$ is of weight $0$, it
must be of type $\{(1,-1),(0,0),(-1,1)$, and so $h_{0}$ satisfies the
condition SV1. Hence $u_{0}$ satisfies condition SU1 of Theorem \ref{h21}.

Let $G^{1}$ be the subgroup of $G$ of elements fixing $t_{0}$. As $t_{0}$ is a
polarization of the Hodge structure, $(2\pi i)^{m}t_{0}$ is a $C$-polarization
of $V$ relative to $G^{1}$, and so $\inn\left(  C\right)  $ is a Cartan
involution of $G^{1}$ (Theorem \ref{h20b}). Now $C=h_{0}(i)=u_{0}(-1)$, and so
$u_{0}$ satisfies condition SU2 of Theorem \ref{h21}. The set $D$ is a
connected component of the space of homomorphisms $u\colon\mathbb{S}{}%
^{1}\rightarrow(G^{1})^{\mathrm{ad}}$, and so it is equal to the set of
conjugates of $u_{0}$ by elements of $(G^{1})^{\mathrm{ad}}(\mathbb{R}{})^{+}$
(apply \ref{h65a} below with $\mathbb{S}{}$ replaced by $\mathbb{S}{}^{1}$).
Any compact factors of $(G^{1})^{\mathrm{ad}}$ can be discarded, and so
Theorem \ref{h21} shows that $D$ is a hermitian symmetric domain.
\end{proof}

\begin{remark}
\label{h65}The universal family of Hodge structures on the period domain
$D(V,h,t_{0})$ satisfies Griffiths transversality only if (a) $(V,h)$ is of
type $\{(-1,0),(0,-1)\}$, or (b) $(V,h)$ of type $\{(-1,1),(0,0),(1,-1)\}$ and
$h^{-1,1}\leq1$, or (c) $(V,h)$ is a Tate twist of one of these Hodge structures.
\end{remark}

\subsection{Period subdomains}

\begin{plain}
\label{h65a}We shall need the following statement (\cite{deligne1979},
1.1.12.). Let $G$ be a real algebraic group, and let $X$ be a (topological)
connected component of the space of homomorphisms $\mathbb{S}{}\rightarrow G$.
Let $G_{1}$ be the smallest algebraic subgroup of $G$ through which all the
$h\in X$ factor. Then $X$ is again a connected component of the space of
homomorphisms of $\mathbb{S}{}$ into $G_{1}$. Since $\mathbb{S}{}$ is a torus,
any two elements of $X$ are conjugate, and so the space $X$ is a
$G_{1}(\mathbb{R}{})^{+}$-conjugacy class of morphisms from $\mathbb{S}{}$
into $G$. It is also a $G(\mathbb{R}{})^{+}$-conjugacy class, and $G_{1}$ is a
normal subgroup of the identity component of $G$.
\end{plain}

Let $(V,F_{0})$ be a real Hodge structure of weight $m$. A tensor $t\colon
V^{\otimes2r}\rightarrow\mathbb{R}(-mr)$ of $V$ is a \emph{Hodge tensor} of
$(V,F_{0})$ if it is a morphism of Hodge structures. Concretely, this means
that $t$ is of type $(0,0)$ for the natural Hodge structure on
\[
\Hom(V^{\otimes2r},\mathbb{R}{}(-mr))\simeq\left(  V^{\vee}\right)
^{\otimes2r}(-mr),
\]
or that it lies in $F^{0}\left(  \Hom(V^{\otimes2r},\mathbb{R}{}(-mr))\right)
$.

We now fix a real Hodge structure $(V,F_{0})$ of weight $m$ and a family
$\mathfrak{t}{}=(t_{i})_{i\in I}$ of Hodge tensors of $(V,F_{0})$. We assume
that $I$ contains an element $0$ such that $t_{0}$ is a polarization of
$(V,F_{0})$. Let $D(V,F_{0},\mathfrak{t})$ be a connected component of the set
of Hodge filtrations $F$ in $D(V,F_{0},t_{0})$ for which every $t_{i}$ is a
Hodge tensor. Thus, $D(V,F_{0},\mathfrak{t}{})$ is a connected component of
the space of Hodge structures on $V$ for which every $t_{i}$ is a Hodge tensor
and $t_{0}$ is a polarization.

Let $G$ be the algebraic subgroup of $\GL_{V}\times\GL_{\mathbb{Q}{}(1)}$
fixing the $t_{i}$. Then $G(\mathbb{R}{})$ consists of the pairs $(g,c)$ such
that
\[
t_{i}(gv_{1},\ldots,gv_{2r})=c^{rm}t_{i}(v_{1},\ldots,v_{2r})
\]
for $i\in I$. Let $h$ be a homomorphism $\mathbb{S}{}\rightarrow\GL_{V}$. The
$t_{i}$ are Hodge tensors for $(V,h)$ if and only if the homomorphism%
\[
z\mapsto(h(z),z\bar{z})\colon\mathbb{S}{}\rightarrow\GL_{V}\times
\mathbb{G}_{m}%
\]
factors through $G$. Thus, to give a Hodge structure on $V$ for which all the
$t_{i}$ are Hodge tensors is the same as giving a homomorphism $h\colon
\mathbb{S}{}\rightarrow G$, and so $D$ is a connected component of the space
of homomorphisms $\mathbb{S}{}\rightarrow G$.

Let $G_{1}$ be the smallest algebraic subgroup of $G$ through which all the
$h$ in $D$ factor. According to (\ref{h65a}), $D$ is a $G_{1}(\mathbb{R}%
{})^{+}$-conjugacy class of homomorphisms $\mathbb{S}{}\rightarrow G_{1}$. The
group $G_{1}(\mathbb{C}{})$ acts on $D^{\vee}(V,F_{0},t_{0})$, and we let
$D^{\vee}(V,F_{0},\mathfrak{t}{})$ denote the orbit of $F_{0}$.

\begin{theorem}
\label{h66}The set $D^{\vee}(V,F_{0},\mathfrak{t}{})$ is a compact complex
submanifold of $D^{\vee}(V,F_{0},t_{0})$, and $D$ is an open complex
submanifold of $D^{\vee}$.
\end{theorem}

\begin{proof}
In fact, $D^{\vee}(V,F_{0},t_{0})$ is a smooth projective algebraic variety.
The stabilizer $P$ of $F_{0}$ in the algebraic group $G_{1\mathbb{C}{}}$ is
parabolic, and so the orbit of $F_{0}$ in the algebraic variety $D^{\vee
}(V,F_{0},t_{0}\mathfrak{)}$ is smooth projective variety. Thus, its complex
points form a compact complex submanifold. As
\[
D(V,h_{0},\mathfrak{t}_{0})=D(V,h_{0},t_{0})\cap D^{\vee}(V,h_{0}%
,\mathfrak{t}_{0}),
\]
it is an open complex submanifold of $D^{\vee}(V,h_{0},\mathfrak{t}_{0}).$
\end{proof}

We call $D=D(V,F_{0},\mathfrak{t}{})$ the \emph{period subdomain}%
\index{period subdomain}
defined by $(V,F_{0},\mathfrak{t}{})$.

\begin{theorem}
\label{h67}Let ($\mathsf{V},F)$ be a family of Hodge structures on a complex
manifold $S$, and let $\ \mathfrak{t}{}=(t_{i})_{i\in I}$ be a family of Hodge
tensors of $\mathsf{V}$. Assume that $I$ contains an element $0$ such that
$t_{0}$ is a polarization. Let $U$ be a connected open subset of $S$ on which
the local system $\mathsf{V}$ is trivial, and choose an isomorphism
$\mathsf{V}|U\overset{\approx}{\longrightarrow}V_{U}$ and a point $o\in U$.
The map $\mathcal{P}{}\colon U\rightarrow D(V,F_{o},\mathfrak{t}{}_{o})$
sending a point $s\in U$ to the point $(\mathsf{V}_{s},F_{s},\mathfrak{t}%
{}_{s})$ is holomorphic.
\end{theorem}

\begin{proof}
Same as that of Theorem \ref{h63}.
\end{proof}

\begin{theorem}
\label{h68}If the universal family of Hodge structures on $D$ satisfies
Griffiths transversality, then $D$ is a hermitian symmetric domain.
\end{theorem}

\begin{proof}
Essentially the same as that of Theorem \ref{h64}.
\end{proof}

\begin{theorem}
\label{h69}Every hermitian symmetric domain arises as a period subdomain.
\end{theorem}

\begin{proof}
Let $D$ be a hermitian symmetric domain, and let $o\in D$. Let $H$ be the real
adjoint algebraic group such that $H(\mathbb{R}{})^{+}=\mathrm{Hol}(D)^{+}$,
and let $u\colon\mathbb{S}{}^{1}\rightarrow H$ be the homomorphism such that
$u(z)$ fixes $o$ and acts on $T_{0}(D)$ as multiplication by $z$ (see \S 2).
Let $h\colon\mathbb{S}{}\rightarrow H$ be the homomorphism such that
$h(z)=u_{o}(z/\bar{z})$ for $z\in\mathbb{C}^{\times}=\mathbb{S}{}(\mathbb{R}%
{})$. Choose a faithful representation $\rho\colon H\rightarrow\GL_{V}$ of
$G$. Because $u$ satisfies (\ref{h21}, SU2), the Hodge structure $(V,\rho\circ
h)$ is polarizable. Choose a polarization and include it in a family
$\mathfrak{t}$ of tensors for $V$ such that $H$ is the subgroup of
$\GL_{V}\times\GL_{\mathbb{Q}{}(1)}$ fixing the elements of $\mathfrak{t}{}$.
Then $D\simeq D(V,h,\mathfrak{t}{})$.
\end{proof}

\begin{nt}
The interpretation of hermitian symmetric domains as moduli spaces for Hodge
structures with tensors is taken from \cite{deligne1979}, 1.1.17.
\end{nt}

\subsection{Why moduli varieties are (sometimes) locally symmetric}

Fix a base field $k$. A \emph{moduli problem}%
\index{moduli problem}
over $k$ is a contravariant functor $\mathcal{F}{}$ from the category of (some
class of) schemes over $k$ to the category of sets. A variety $S$ over $k$
together with a natural isomorphism $\phi\colon\mathcal{F}{}\rightarrow
\Hom_{k}(-,S)$ is called a \emph{fine solution to the moduli problem}%
\index{fine solution to the moduli problem}%
. A variety that arises in this way is called a \emph{moduli variety}%
\index{moduli variety}%
.

Clearly, this definition is too general: every variety $S$ represents the
functor $h_{S}=\Hom_{k}(-,S)$. In practice, we only consider functors for
which $\mathcal{F}{}(T)$ is the set of isomorphism classes of some
algebro-geometric objects over $T$, for example, families of algebraic
varieties with additional structure.

If $S$ represents such a functor, then there is an object $\alpha
\in\mathcal{F}{}(S)$ that is universal in the sense that, for any
$\alpha^{\prime}\in\mathcal{F}{}(T)$, there is a unique morphism $a\colon
T\rightarrow S$ such that $\mathcal{F}{}(a)(\alpha)=\alpha^{\prime}$. Suppose
that $\alpha$ is, in fact, a smooth projective map $f\colon X\rightarrow S$ of
smooth varieties over $\mathbb{C}{}$. Then $R^{m}f_{\ast}\mathbb{Q}{}$ is a
polarizable variation of Hodge structures on $S$ admitting an integral
structure (Theorem \ref{h51}). A polarization of $X/S$ defines a polarization
of $R^{m}f_{\ast}\mathbb{Q}{}$ and a family of algebraic classes on $X/S$ of
codimension $m$ defines a family of global sections of $R^{2m}f_{\ast
}\mathbb{Q}{}(m)$. Let $D$ be the universal covering space of $S^{\text{an}}$.
The pull-back of $R^{m}f_{\ast}\mathbb{Q}{}$ to $D$ is a variation of Hodge
structures whose underlying locally constant sheaf of $\mathbb{Q}{}$-vector
spaces is constant, say, equal to $V_{S}$; thus we have a variation of Hodge
structures $(V_{S},F)$ on $D$. We suppose that the additional structure on
$X/S$ defines a family $\mathfrak{t}{}=(t_{i})_{i\in I}$ of Hodge tensors of
$V_{S}$ with $t_{0}$ a polarization. We also suppose that the family of Hodge
structures on $D$ is universal, i.e., that $D=$
$D(V,F_{0},\mathfrak{t}{})$. Because $(V_{S},F)$ is a variation of Hodge
structures, $D$ is a hermitian symmetric domain (by \ref{h68}). The Margulis
arithmeticity theorem (\ref{h37}) shows that $\Gamma$ is an arithmetic
subgroup of $G(D)$ except possibly when $G(D)$ has factors of small dimension.
Thus, when looking at moduli varieties, we are naturally led to consider
arithmetic locally symmetric varieties.

\begin{remark}
In fact it is unusual for a moduli problem to lead to a locally symmetric
variety. The above argument will usually break down where we assumed that the
variation of Hodge structures is universal. Essentially, this will happen only
when a \textquotedblleft general\textquotedblright\ member of the family has a
Hodge structure that is special in the sense of \S 6. Even for smooth
hypersurfaces of a fixed degree, this is rarely happens (see \ref{h54c} and
\ref{h54g}). Thus, in the whole universe of moduli varieties, locally
symmetric varieties form only a small, but important, class.
\end{remark}

\subsection{Application: Riemann's theorem in families}

Let $A$ be an abelian variety over $\mathbb{C}{}$. The exponential map defines
an exact sequence%
\[
0\rightarrow H_{1}(A^{\text{an}},\mathbb{Z}{})\rightarrow T_{0}(A^{\text{an}%
})\overset{\exp}{\longrightarrow}A^{\text{an}}\rightarrow0.
\]
From the first map in this sequence, we get an exact sequence%
\[
0\rightarrow\Ker(\alpha)\rightarrow H_{1}(A^{\text{an}},\mathbb{Z}%
{})_{\mathbb{C}{}}\overset{\alpha}{\longrightarrow}T_{0}(A^{\text{an}%
})\rightarrow0.
\]
The $\mathbb{Z}{}$-module $H_{1}(A^{\text{an}},\mathbb{Z}{})$ is an integral
Hodge structure with Hodge filtration%
\[
\renewcommand{\arraystretch}{1.3}%
\begin{array}
[c]{ccccc}%
H_{1}(A^{\text{an}},\mathbb{Z})_{\mathbb{C}{}} & \supset & \Ker(\alpha) &
\supset & 0.\\
F^{-1} &  & F^{0} &  &
\end{array}
\]
Let $\psi$ be a Riemann form for $A$. Then $2\pi i\psi$ is a polarization for
the Hodge structure $H_{1}(A^{\text{an}},\mathbb{Z}{})$.

\begin{theorem}
\label{h69a}The functor $A\rightsquigarrow H_{1}(A^{\text{an}},\mathbb{Z}{})$
is an equivalence from the category of abelian varieties over $\mathbb{C}{}$
to the category of polarizable integral Hodge structures of type
$\{(-1,0),(0,-1)\}$.
\end{theorem}

\begin{proof}
In view of the correspondence between complex structures and Hodge structures
of type $\{(-1,0),(0,-1)\}$ (see \ref{h47}), this is simply a restatement of
Theorem \ref{h44}.
\end{proof}

\begin{theorem}
\label{h69b}Let $S$ be a smooth algebraic variety over $\mathbb{C}{}$. The
functor
\[
(A\overset{f}{\longrightarrow}S)\rightsquigarrow R_{1}f_{\ast}\mathbb{Z}%
\]
is an equivalence from the category of families of abelian varieties over $S$
to the category of polarizable integral variations of Hodge structures of type
$\{(-1,0),(0,-1)\}$.
\end{theorem}

\begin{proof}
Let $f^{A}\colon A\rightarrow S$ be a family of abelian varieties over $S$.
The exponential defines an exact sequence of sheaves on $S^{\text{an}}$,%
\[
0\rightarrow R_{1}f_{\ast}^{A}\mathbb{Z}{}\rightarrow\mathcal{T}{}%
_{0}(A^{\text{an}})\rightarrow A^{\text{an}}\rightarrow0.
\]
From this one sees that the map $\Hom(A^{\text{an}},B^{\text{an}}%
)\rightarrow\Hom(R_{1}f_{\ast}^{A}\mathbb{Z}{},R_{1}f_{\ast}^{B}\mathbb{Z}{})$
is an isomorphism. The $S$-scheme $\mathcal{H}om_{S}(A,B)$ is unramified over
$S$, and so its algebraic sections coincide with its holomorphic sections (cf.
\cite{deligne1971h}, 4.4.3). Hence the functor is fully faithful. In
particular, a family of abelian varieties is uniquely determined by its
variation of Hodge structures up to a unique isomorphism. This allows us to
construct the family of abelian varieties attached to a variation of Hodge
structures locally. Thus, we may suppose that the underlying local system of
$\mathbb{Z}{}$-modules is trivial. Assume initially that the variation of
Hodge structures on $S$ has a principal polarization, and endow it with a
level-$N$ structure. According Proposition \ref{h45}, the variation of Hodge
structures on $S$ is the pull-back of the canonical variation of Hodge
structures on $D_{N}$ by a regular map $\alpha\colon S\rightarrow D_{N}$.
Since the latter variation arises from a family of abelian varieties (Theorem
\ref{h46}), so does the former.

In fact, the argument still applies when the variation of Hodge structures is
not \textit{principally} polarized, since \cite{mumford1965}, Chapter 7, hence
Theorem \ref{h46}, applies also to nonprincipally polarized abelian varieties.
Alternatively, Zarhin's trick (cf. \cite{milne1986}, 16.12) can be used to
show that (locally) the fourth multiple of the variation of Hodge structures
is principally polarized.
\end{proof}

\section{Variations of Hodge structures on locally symmetric varieties}

\begin{quote}
{\small In this section, we explain how to classify variations of Hodge
structures on arithmetic locally symmetric varieties in terms of certain
auxiliary reductive groups. Throughout, we write \textquotedblleft family of
integral Hodge structures\textquotedblright\ to mean \textquotedblleft family
of rational Hodge structures that admits an integral
structure\textquotedblright. }
\end{quote}

\subsection{Existence of Hodge structures of CM-type in a family}

\begin{proposition}
\label{h91d}Let $G$ be a reductive group over $\mathbb{Q}{}$, and let
$h\colon\mathbb{S}{}\rightarrow G_{\mathbb{R}{}}$ be a homomorphism. There
exists a $G(\mathbb{R}{})^{+}$-conjugate $h_{0}$ of $h$ such that
$h_{0}(\mathbb{S}{})\subset T_{0\mathbb{R}{}}$ for some maximal torus $T_{0}$
of $G$.
\end{proposition}

\begin{proof}
[\cite{mumford1969}, \textnf{p.~348}]Let $K$ be the centralizer of $h$ in
$G_{\mathbb{R}{}}$, and let $T$ be the centralizer in $G_{\mathbb{R}{}}$ of
some regular element of $\Lie K$; it is a maximal torus in $K$. Because
$h(\mathbb{\mathbb{S}{}}{})$ centralizes $T$, $h(\mathbb{S}{})\cdot T$ is a
torus in $K$, and so $h(\mathbb{S}{})\subset T$. If $T^{\prime}$ is a torus in
$G_{\mathbb{R}{}}$ containing $T$, then $T^{\prime}$ centralizes $h$, and so
$T^{\prime}\subset K$; therefore $T=T^{\prime}$, and so $T$ is maximal in
$G_{\mathbb{R}{}}$. For a regular element $\lambda$ of $\Lie(T)$, $T$ is the
centralizer of $\lambda$. Choose a $\lambda_{0}\in\Lie(G)$ that is close to
$\lambda$ in $\Lie(G)_{\mathbb{R}{}}$, and let $T_{0}$ be its centralizer in
$G$. Then $T_{0}$ is a maximal torus of $G$ (over $\mathbb{Q}{})$. Because
$T_{0\mathbb{R}{}}$ and $T_{\mathbb{R}{}}$ are close, they are conjugate:
$T_{0\mathbb{R}{}}=gTg^{-1}$ for some $g\in G(\mathbb{R}{})^{+}$. Now
$h_{0}\overset{\textup{{\tiny def}}}{=}\inn(g)\circ h$ factors through
$T_{0\mathbb{R}{}}$.
\end{proof}

A rational Hodge structure is said to be of \emph{CM-type}%
\index{CM-type!Hodge structure of}
if it is polarizable and its Mumford-Tate group is commutative (hence a torus
by \ref{h53}).

\begin{proposition}
\label{h91f}Let $(V,F_{0})$ be a rational Hodge structure of some weight $m$,
and let $\mathfrak{t}{}=(t_{i})_{i\in I}$ be a family of tensors of
$(V,F_{0})$ including a polarization. Then the period subdomain defined by
$(V,F_{0},\mathfrak{t})_{\mathbb{R}{}}$ includes a Hodge structure of CM-type.
\end{proposition}

\begin{proof}
We are given a $\mathbb{Q}{}$-vector space $V$, a homomorphism $h_{0}%
\colon\mathbb{S}{}\rightarrow\GL_{V_{\mathbb{R}{}}}$, and a family of Hodge
tensors $V^{\otimes2r}\rightarrow\mathbb{Q}{}(-mr)$ including a polarization.
Let $G$ be the algebraic subgroup of $\GL_{V}\times\GL_{\mathbb{Q}{}(1)}$
fixing the $t_{i}$. Then $G$ is a reductive group because $\inn(h_{0}(i))$ is
a Cartan involution. The period subdomain $D$ is the connected component
containing $h_{0}$ of the space of homomorphisms $h\colon\mathbb{S}%
{}\rightarrow G_{\mathbb{R}{}}$ (see \S 7). This contains the $G(\mathbb{R}%
{})^{+}$-conjugacy class of $h_{0}$, and so the statement follows from
Proposition \ref{h91d}.
\end{proof}

\subsection{Description of the variations of Hodge structures on $D(\Gamma)$}

Consider an arithmetic locally symmetric variety $D(\Gamma)$. Recall that this
means that $D(\Gamma)$ is an algebraic variety whose universal covering space
is a hermitian symmetric domain $D$ and that the group of covering
transformations $\Gamma$ is an arithmetic subgroup of the real Lie group
$\mathrm{Hol}(D)^{+}$; moreover, $D(\Gamma)^{\text{an}}=\Gamma\backslash D$.

According to Theorem \ref{h32}, $D$ decomposes into a product $D=D_{1}%
\times\cdots\times D_{r}$ of hermitian symmetric domains with the property
that each group $\Gamma_{i}\overset{\textup{{\tiny def}}}{=}\Gamma
\cap\mathrm{Hol}(D_{i})^{+}$ is an irreducible arithmetic subgroup of
$\mathrm{Hol}(D_{i})^{+}$ and the map%
\[
D_{1}(\Gamma_{1})\times\cdots\times D_{r}(\Gamma_{r})\rightarrow D(\Gamma)
\]
is finite covering. In order to be able to apply the theorems of Margulis we
assume that%
\begin{equation}
\rank(\mathrm{Hol}(D_{i}))\geq2\text{ for each }i \label{hq51}%
\end{equation}

\noindent in the remainder of this subsection. We also fix a point $o\in D$.

Recall (\ref{h20p}) that there exists a unique homomorphism $u\colon
U^{1}\rightarrow\mathrm{Hol}(D)$ such that $u(z)$ fixes $o$ and acts as
multiplication by $z$ on $T_{o}(D)$. That $\Gamma$ is arithmetic means that
there exists a simply connected algebraic group $H$ over $\mathbb{Q}{}$ and a
surjective homomorphism $\varphi\colon H(\mathbb{R}{})\rightarrow
\mathrm{Hol}(D)^{+}$ with compact kernel such that $\Gamma$ is commensurable
with $\varphi(H(\mathbb{Z}{}))$. The Margulis superrigidity theorem implies
that the pair $(H,\varphi)$ is unique up to a unique isomorphism (see
\ref{h37a}).

Let%
\[
H_{\mathbb{R}{}}^{\mathrm{ad}}=H_{\mathrm{c}}\times H_{\mathrm{nc}}%
\]
where $H_{\mathrm{c}}$ (resp. $H_{\mathrm{nc}}$) is the product of the compact
(resp. noncompact) simple factors of $H_{\mathbb{R}{}}^{\mathrm{ad}}$. The
homomorphism $\varphi(\mathbb{R})\colon{}H(\mathbb{R}{})\rightarrow
\mathrm{Hol}(D)^{+}$ factors through $H_{\mathrm{nc}}(\mathbb{R}{})^{+}$, and
defines an isomorphism of Lie groups $H_{\mathrm{nc}}(\mathbb{R}{}%
)^{+}\rightarrow\mathrm{Hol}(D)^{+}$. Let $\bar{h}$ denote the homomorphism
$\mathbb{S}{}/\mathbb{G}_{m}\rightarrow H_{\mathbb{R}{}}^{\mathrm{ad}}$ whose
projection into $H_{\mathrm{c}}$ is trivial and whose projection into
$H_{\mathrm{nc}}$ corresponds to $u$ as in (\ref{h90a}). In other words,%
\begin{equation}
\bar{h}(z)=(h_{\mathrm{c}}(z),h_{\mathrm{nc}}(z))\in H_{\mathrm{c}}%
(\mathbb{R}{})\times H_{\mathrm{nc}}(\mathbb{R}{}) \label{hq52}%
\end{equation}
where $h_{\mathrm{c}}(z)=1$ and $h_{\mathrm{nc}}(z)=u(z/\bar{z})$ in
$H_{\mathrm{nc}}(\mathbb{R}{})^{+}\simeq\mathrm{Hol}(D)^{+}$. The map
$^{g}h\mapsto go$ identifies $D$ with the set of $H^{\mathrm{ad}}(\mathbb{R}%
{})^{+}$-conjugates of $\bar{h}$ (Theorem \ref{h21}).

Let $(\mathsf{V},F)$ be a polarizable variation of integral Hodge structures
on $D(\Gamma)$, and let $V=\mathsf{V}_{\pi(o)}$. Then $\pi^{\ast}%
\mathsf{V}\simeq V_{D}$ where $\pi\colon D\rightarrow\Gamma\backslash D$ is
the quotient map. Let $G\subset\GL_{V}$ be the generic Mumford-Tate group of
$(\mathsf{V},F)$ (see p.~\pageref{generic}), and let $\mathfrak{t}{}$ be a
family of tensors of $V$ (in the sense of \S 7), including a polarization
$t_{0}$, such that $G$ is the subgroup of $\GL_{V}\times\GL_{\mathbb{Q}{}(1)}$
fixing the elements of $\mathfrak{t}{}$. As $G$ contains the Mumford-Tate
group at each point of $D$, $\mathfrak{t}{}$ is a family of Hodge tensors of
$(V_{D},F)$. The period map $\mathcal{P}{}\colon D\rightarrow D(V,h_{o}%
,\mathfrak{t}{})$ is holomorphic (Theorem \ref{h67}).

We now assume that the monodromy map $\varphi^{\prime}\colon\Gamma
\rightarrow\GL(V)$ has finite kernel, and we pass to a finite covering, so
that $\Gamma\subset G(\mathbb{Q}{})$. Now the elements of $\mathfrak{t}$ are
Hodge tensors of $(\mathsf{V},F)$.

There exists an arithmetic subgroup $\Gamma^{\prime}$ of $H(\mathbb{Q}{})$
such that $\varphi(\Gamma^{\prime})\subset\Gamma$. The Margulis superrigidity
theorem \ref{h35f}, shows that there is a (unique) homomorphism $\varphi
^{\prime\prime}\colon H\rightarrow G$ of algebraic groups that agrees with
$\varphi^{\prime}\circ\varphi$ on a subgroup of finite index in $\Gamma
^{\prime}$,%
\[
\begin{tikzpicture}
\node (P1) at (0,0) {$H(\mathbb{Q})^+$};
\node (P2) at (2,0) {$\mathrm{Hol}(D)^+$};
\node (P3) at (6,0) {$H$};
\node (P4) at (0,-0.7) {$\cup$};
\node (P5) at (2,-0.7) {$\cup$};
\node (P6) at (0,-1.4) {$\Gamma^{\prime}$};
\node (P7) at (2,-1.4) {$\Gamma$};
\node (P8) at (4,-1.4) {$G(\mathbb{Q})$};
\node (P9) at (10,-1.4) {$G$};
\draw[->,font=\scriptsize,>=angle 90]
(P1) edge node[above] {$\varphi$} (P2)
(P6) edge node[above] {$\varphi|\Gamma^{\prime}$} (P7)
(P7) edge node[above] {$\varphi^{\prime}$} (P8)
(P3) edge node[above] {$\varphi^{\prime\prime}$} (P9);
\end{tikzpicture}
\]

It follows from the Borel density theorem \ref{h35a} that $\varphi
^{\prime\prime}(H)$ is the connected monodromy group at each point of
$D(\Gamma)$. Hence $H\subset G^{\mathrm{der}}$, and the two groups are equal
if the Mumford-Tate group at some point of $D(\Gamma)$ is commutative (Theorem
\ref{h56}). When we assume that, the homomorphism $\varphi^{\prime\prime
}\colon H\rightarrow G$ induces an isogeny $H\rightarrow G^{\mathrm{der}}$,
and hence\footnote{\label{ad}Let $G$ be a reductive group. The algebraic
subgroup $Z(G)\cdot G^{\mathrm{der}}$ is normal, and the quotient $G/\left(
Z(G)^{\circ}\cdot G^{\mathrm{der}}\right)  $ is both semisimple and
commutative, and hence is trivial. Therefore $G=Z(G)^{\circ}\cdot
G^{\mathrm{der}}$, from which it follows that $Z(G^{\mathrm{der}})=Z(G)\cap
G^{\mathrm{der}}$. For any isogeny $H\rightarrow G^{\mathrm{der}}$, the map
$H^{\mathrm{ad}}\rightarrow(G^{\mathrm{der}})^{\mathrm{ad}}$ is certainly an
isomorphism, and we have just shown that $(G^{\mathrm{der}})^{\mathrm{ad}%
}\rightarrow G^{\mathrm{ad}}$ is an isomorphism. Therefore $H^{\mathrm{ad}%
}\rightarrow G^{\mathrm{ad}}$ is an isomorphism.} an isomorphism
$H^{\mathrm{ad}}\rightarrow G^{\mathrm{ad}}$. Let $(V,h_{o})=(\mathsf{V}%
,F)_{o}$. Then
\[
\ad\circ h_{o}\colon\mathbb{S}{}\rightarrow G_{\mathbb{R}{}}^{\mathrm{ad}%
}\simeq H^{\mathrm{ad}}%
\]
equals $\bar{h}$. Thus, we have a commutative diagram%
\begin{equation}
\begin{tikzpicture} \matrix(m)[matrix of math nodes, row sep=3em, column sep=2.5em, text height=1.5ex, text depth=0.25ex] {H\\ (H^{\text{ad}},\bar{h})&(G,h)&\GL_V\\}; \path[->,font=\scriptsize,>=angle 90] (m-1-1) edge (m-2-1) (m-1-1) edge node[auto] {} (m-2-2) (m-2-2) edge (m-2-1); \path[right hook->,font=\scriptsize,>=angle 90] (m-2-2) edge node[auto]{$\rho$} (m-2-3); \end{tikzpicture} \label{hq17}%
\end{equation}
in which $G$ is a reductive group, the homomorphism $H\rightarrow G$ has image
$G^{\mathrm{der}}$, $w_{h}$ is defined over $\mathbb{Q}$, ${}$ and $h$
satisfies (SV2*).

Conversely, suppose that we are given such a diagram (\ref{hq17}). Choose a
family $\mathfrak{t}{}$ of tensors for $V$, including a polarization, such
that $G$ is the subgroup of $\GL_{V}\times G_{\mathbb{Q}{}(1)}$ fixing the
tensors. Then we get a period subdomain $D(V,h,\mathfrak{t}{})$ and a
canonical variation of Hodge structures $(\mathsf{V},F)$ on it. Pull this back
to $D$ using the period isomorphism, and descend it to a variation of Hodge
structures on $D(\Gamma)$. The monodromy representation is injective, and some
fibre is of CM-type by Proposition \ref{h91f}.

\begin{summary}
\label{h92a}Let $D(\Gamma)$ be an arithmetic locally symmetric domain
satisfying the condition (\ref{hq51}) and fix a point $o\in D$. To give\bquote
a polarizable variation of integral Hodge structures on $D(\Gamma)$ such that
some fibre is of CM-type and the monodromy representation has finite
kernel\equote is the same as giving\bquote a diagram (\ref{hq17}) in which $G$
is a reductive group, the homomorphism $H\rightarrow G$ has image
$G^{\mathrm{der}}$, $w_{h}$ is defined over $\mathbb{Q}$, ${}$ and $h$
satisfies (SV2*).\equote\qquad
\end{summary}

\begin{fundamentalquestion}
\label{h92b}For which arithmetic locally symmetric varieties $D(\Gamma)$ is it
possible to find a diagram (\ref{hq17}) with the property that the
corresponding variation of Hodge structures underlies a family of algebraic
varieties? or, more generally, a family of motives?
\end{fundamentalquestion}

In \S \S 10,11, we shall answer Question \ref{h92b} completely when
\textquotedblleft algebraic variety\textquotedblright\ and \textquotedblleft
motive\textquotedblright\ are replaced with \textquotedblleft abelian
variety\textquotedblright\ and \textquotedblleft abelian
motive\textquotedblright.

\subsection{Existence of variations of Hodge structures}

In this subsection, we show that, for every arithmetic locally symmetric
variety, there exists a diagram (\ref{hq17}), and hence a variation of
polarizable integral Hodge structures on the variety.

\begin{proposition}
\label{h91c}Let $H$ be a semisimple algebraic group over $\mathbb{Q}{}$, and
let $\bar{h}\colon\mathbb{S}{}\rightarrow H^{\mathrm{ad}}$ be a homomorphism
satisfying (SV1,2,3). Then there exists a reductive algebraic group $G$ over
$\mathbb{Q}{}$ and a homomorphism $h\colon\mathbb{S}{}\rightarrow
G_{\mathbb{R}}$ such that

\begin{enumerate}
\item $G^{\mathrm{der}}=H$ and $\bar{h}=\ad\circ h$,

\item the weight $w_{h}$ is defined over $\mathbb{Q}{}$, and

\item the centre of $G$ is split by a CM field%
\index{CM field}
(i.e., a totally imaginary quadratic extension of a totally real number field).
\end{enumerate}
\end{proposition}

\begin{proof}
We shall need the following statement:\bquote Let $G$ be a reductive group
over a field $k$ (of characteristic zero), and let $L$ be a finite Galois
extension of $k$ splitting $G$. Let $G^{\prime}\rightarrow G^{\mathrm{der}}$
be a covering of the derived group of $G$. Then there exists a central
extension%
\[
1\rightarrow N\rightarrow G_{1}\rightarrow G\rightarrow1
\]
such that $G_{1}$ is a reductive group, $N$ is a product of copies of
$(\mathbb{G}_{m})_{L/k}$, and
\[
(G_{1}^{\mathrm{der}}\rightarrow G^{\mathrm{der}})=(G^{\prime}\rightarrow
G^{\mathrm{der}}).
\]
See \cite{milneS1982}, 3.1.\equote A number field $L$ is CM if and only if it
admits a nontrivial involution $\iota_{L}$ such that $\sigma\circ\iota
_{L}=\iota\circ\sigma$ for every homomorphism $\sigma\colon L\rightarrow
\mathbb{C}{}$. We may replace $\bar{h}$ with an $H^{\mathrm{ad}}(\mathbb{R}%
{})^{+}$-conjugate, and so assume (by Proposition \ref{h91d}) that there
exists a maximal torus $\bar{T}$ of $H^{\mathrm{ad}}$ such that $\bar{h}$
factors through $\bar{T}_{\mathbb{R}{}}$. Then $\bar{T}_{\mathbb{R}{}}$ is
anisotropic (by (SV2)), and so $\iota$ acts as $-1$ on $X^{\ast}(\bar{T})$. It
follows that, for any $\sigma\in\Aut(\mathbb{C}{})$, $\sigma\iota$ and
$\iota\sigma$ have the same action on $X^{\ast}(\bar{T})$, and so $\bar{T}$
splits over a CM-field $L$, which can be chosen to be Galois over
$\mathbb{Q}{}$. From the statement, there exists a reductive group $G$ and a
central extension%
\[
1\rightarrow N\rightarrow G\rightarrow H^{\mathrm{ad}}\rightarrow1
\]
such that $G^{\mathrm{der}}=H$ and $N$ is a product of copies of
$(\mathbb{G}_{m})_{L/\mathbb{Q}{}}$. The inverse image $T$ of $\bar{T}$ in $G$
is a maximal torus, and the kernel of $T\twoheadrightarrow\bar{T}$ is $N$.
Because $N$ is connected, there exists a $\mu\in X_{\ast}(T)$ lifting
$\mu_{\bar{h}}\in X_{\ast}(\bar{T})$.\footnote{The functor $X^{\ast}$ is
exact, and so $0\rightarrow X^{\ast}(\bar{T})\rightarrow X^{\ast
}(T)\rightarrow X^{\ast}(N)\rightarrow0$ is exact. In fact, it is split-exact
(as a sequence of $\mathbb{Z}{}$-modules) because $X^{\ast}(N)$ is
torsion-free. On applying $\Hom(-,\mathbb{Z}{})$ to it, we get the exact
sequence $\cdots\rightarrow X_{\ast}(T)\rightarrow X_{\ast}(\bar
{T})\rightarrow0$.} The weight $w=-\mu-\iota\mu$ of $\mu$ lies in $X_{\ast
}(Z)$, where $Z=Z(G)=N$. Clearly $\iota w=w$ and so, as the Tate cohomology
group\footnote{Let $g=\Gal(\mathbb{C}{}/\mathbb{R}{})$. The $g$-module
$X_{\ast}(Z)$ is induced, and so the Tate cohomology group $H_{T}%
^{0}(g,X_{\ast}(Z))=0$. By definition, $H_{T}^{0}(g,X_{\ast}(Z))=X_{\ast
}(Z)^{g}/(\iota+1)X_{\ast}(Z)$.} $H_{T}^{0}(\mathbb{R}{},X_{\ast}(Z))=0$,
there exists a $\mu_{0}\in X_{\ast}(Z)$ such that $(\iota+1)\mu_{0}=w$. When
we replace $\mu$ with $\mu-\mu_{0}$, we find that $w=0$; in particular, $w$ is
defined over $\mathbb{Q}{}$. Let $h\colon\mathbb{S}{}\rightarrow
G_{\mathbb{R}{}}$ correspond to $\mu$ as in (\ref{hq40}), p.~\pageref{hq40}.
Then $(G,h)$ fulfils the requirements.
\end{proof}

\begin{corollary}
\label{h92c}For any semisimple algebraic group $H$ over $\mathbb{Q}{}$ and
homomorphism $\bar{h}\colon\mathbb{S}{}/\mathbb{G}_{m}\rightarrow
H_{\mathbb{R}{}}^{\mathrm{ad}}$ satisfying (SV1,2,3), there exists a reductive
group $G$ with $G^{\mathrm{der}}=H$ and a homomorphism $h\colon\mathbb{S}%
{}\rightarrow G_{\mathbb{R}{}}$ lifting $\bar{h}$ and satisfying (SV1,2*,3).
\end{corollary}

\begin{proof}
Let $(G,h)$ be as in the proposition. Then $G/G^{\mathrm{der}}$ is a torus,
and we let $T$ be the smallest subtorus of it such that $T_{\mathbb{R}{}}$
contains the image of $h$. Then $T_{\mathbb{R}{}}$ is anisotropic, and when we
replace $G$ with the inverse image of $T$, we obtain a pair $(G,h)$ satisfying (SV1,2*,3).
\end{proof}

Let $G$ be a reductive group over $\mathbb{Q}{}$, and let $h\colon\mathbb{S}%
{}\rightarrow G_{\mathbb{R}{}}$ be a homomorphism satisfying (SV1,2,3). The
homomorphism $h$ is said to be \emph{special}%
\index{special homomorphism@special homomorphism $h$}
if $h(\mathbb{S}{})\subset T_{\mathbb{R}{}}$ for some torus $T\subset
G$.\footnote{Of course, $h(\mathbb{S}{})$ is always contained in a subtorus of
$G_{\mathbb{R}{}}$, even a maximal subtorus; the point is that there should
exist such a torus defined over $\mathbb{Q}{}$.} In this case, there is a
smallest such $T$, and when $(T,h)$ is the Mumford-Tate group of a CM Hodge
structure we say that $h$ is \emph{CM}%
\index{CM}%
.

\begin{proposition}
\label{h91e}Let $h\colon\mathbb{S}{}\rightarrow G_{\mathbb{R}{}}$ be special.
Then $h$ is CM if

\begin{enumerate}
\item $w_{h}$ is defined over $\mathbb{Q}{}$, and

\item the connected centre of $G$ is split by a CM-field.
\end{enumerate}
\end{proposition}

\begin{proof}
It is known that a special $h$ is CM if and only if it satisfies the Serre
condition:%
\[
(\tau-1)(\iota+1)\mu_{h}=0=(\iota+1)(\tau-1)\mu_{h}\text{ for all }\tau
\in\Gal(\mathbb{Q}{}^{\mathrm{al}}/\mathbb{Q}{})\text{.}%
\]
As $w_{h}=(\iota+1)\mu_{h}$, the first condition says that%
\[
(\tau-1)(\iota+1)\mu_{h}=0\text{ for all }\tau\in\Aut(\mathbb{C}{}),
\]
and the second condition implies that%
\[
\tau\iota\mu_{h}=\iota\tau\mu_{h}\text{ for all }\tau\in\Aut(\mathbb{C}{}).
\]
Let $T\subset G$ be a maximal torus such that $h(\mathbb{S}{})\subset
T_{\mathbb{R}{}}$. The argument in the proof of (\ref{h91c}) shows that
$\tau\iota\mu=\iota\tau\mu$ for $\mu\in X_{\ast}(T)$, and since%
\[
X_{\ast}(T)_{\mathbb{Q}}=X_{\ast}(Z)_{\mathbb{Q}{}}{}\oplus X_{\ast
}(T/Z)_{\mathbb{Q}{}}%
\]
we see that the same equation holds for $\mu\in X_{\ast}(T)$. Therefore
$(\iota+1)(\tau-1)\mu=(\tau-1)(\iota+1)\mu$, and we have already observed that
this is zero.
\end{proof}

\section{Absolute Hodge classes and motives}

\begin{quote}
{\small In order to be able to realize all but a handful of Shimura varieties
as moduli varieties, we shall need to replace algebraic varieties and
algebraic classes by more general objects, namely, by motives and absolute
Hodge classes.}
\end{quote}

\subsection{The standard cohomology theories}

Let $X$ be a smooth complete\footnote{Many statements hold without this
hypothesis, but we shall need to consider only this case.} algebraic variety
over an algebraically closed field $k$ (of characteristic zero as always).

For each prime number $\ell$, the \'{e}tale cohomology groups\footnote{The
\textquotedblleft$(m)$\textquotedblright\ denotes a Tate twist. Specifically,
for Betti cohomology it denotes the tensor product with the Tate Hodge
structure $\mathbb{Q}{}(m)$, for de Rham cohomology it denotes a shift in the
numbering of the filtration, and for \'{e}tale cohomology it denotes a change
in Galois action by a multiple of the cyclotomic character.} $H_{\ell}%
^{r}(X)(m)\overset{\textup{{\tiny def}}}{=}H_{\ell}^{r}(X_{\mathrm{et}%
},\mathbb{Q}{}_{\ell}(m))$ are finite dimensional $\mathbb{Q}_{\ell}$-vector
spaces. For any homomorphism $\sigma\colon k\rightarrow k^{\prime}$ of
algebraically closed fields, there is a canonical base change isomorphism%
\begin{equation}
H_{\ell}^{r}(X)(m)\overset{\sigma}{\longrightarrow}H_{\ell}^{r}(\sigma
X)(m),\quad\sigma X\overset{\textup{{\tiny def}}}{=}X\otimes_{k,\sigma
}k^{\prime}. \label{hq5}%
\end{equation}
When $k=\mathbb{C}{}$, there is a canonical comparison isomorphism%
\begin{equation}
\mathbb{Q}{}_{\ell}\otimes_{\mathbb{Q}{}}{}H_{B}^{r}(X)(m)\rightarrow H_{\ell
}^{r}(X)(m). \label{hq7}%
\end{equation}
Here $H_{B}^{r}(X)$ denotes the Betti cohomology group $H^{r}(X^{\text{an}%
},\mathbb{Q}{})$.

The de Rham cohomology groups $H_{\mathrm{dR}}^{r}(X)(m)\overset
{\textup{{\tiny def}}}{=}\mathbb{H}{}^{r}(X_{\mathrm{Zar}},\Omega
_{X/k}^{\bullet})(m)$ are finite dimensional $k$-vector spaces. For any
homomorphism $\sigma\colon k\rightarrow k^{\prime}$ of fields, there is a
canonical base change isomorphism%
\begin{equation}
k^{\prime}\otimes_{k}H_{\mathrm{dR}}^{r}(X)(m)\overset{\sigma}{\longrightarrow
}H_{\mathrm{dR}}^{r}(\sigma X)(m). \label{hq6}%
\end{equation}
When $k=\mathbb{C}{}$, there is a canonical comparison isomorphism%
\begin{equation}
\mathbb{C}{}\otimes_{\mathbb{Q}{}}H_{B}^{r}(X)(m)\rightarrow H_{\mathrm{dR}%
}^{r}(X)(m). \label{hq8}%
\end{equation}

We let $H_{k\times\mathbb{A}{}_{f}}^{r}(X)(m)$ denote the product of
$H_{\mathrm{dR}}^{r}(X)(m)$ with the restricted product of the topological
spaces $H_{\ell}^{r}(X)(m)$ relative to their subspaces $H^{r}(X_{\mathrm{et}%
},\mathbb{Z}{}_{\ell})(m)$. This is a finitely generated free module over the
ring $k\times\mathbb{A}{}_{f}$. For any homomorphism $\sigma\colon
k\rightarrow k^{\prime}$ of algebraically closed fields, the maps (\ref{hq5})
and (\ref{hq6}) give a base change homomorphism%
\begin{equation}
H_{k\times\mathbb{A}{}_{f}}^{r}(X)(m)\overset{\sigma}{\longrightarrow
}H_{k^{\prime}\times\mathbb{A}{}_{f}}^{r}(\sigma X)(m)\text{.} \label{hq1}%
\end{equation}
When $k=\mathbb{C}{}$, the maps (\ref{hq7}) and (\ref{hq8}) give a comparison
isomorphism%
\begin{equation}
(\mathbb{C}{}\times\mathbb{A}{}_{f})\otimes_{\mathbb{Q}{}}H_{B}^{r}%
(X)(m)\rightarrow H_{\mathbb{C}{}\times\mathbb{A}{}_{f}}^{r}(X)(m).
\label{hq2}%
\end{equation}

\begin{nt}
For more details and references, see \cite{deligne1982}, \S 1.
\end{nt}

\subsection{Absolute Hodge classes}

Let $X$ be a smooth complete algebraic variety over $\mathbb{C}{}$. The
cohomology group $H_{B}^{2r}(X)(r)$ has a Hodge structure of weight $0$, and
an element of type $(0,0)$ in it is called a \emph{Hodge class of codimension}%
\index{Hodge class of codimension}
$r$ on $X$.\footnote{As $H_{B}^{2r}(X)(r)\simeq H_{B}^{2r}(X)\otimes
\mathbb{Q}{}(r)$, this is essentially the same as an element of $H_{B}%
^{2r}(X)$ of type $(r,r)$.} We wish to extend this notion to all base fields
of characteristic zero. Of course, given a variety $X$ over a field $k$, we
can choose a homomorphism $\sigma\colon k\rightarrow\mathbb{C}{}$ and define a
Hodge class on $X$ to be a Hodge class on $\sigma X$, but this notion depends
on the choice of the embedding. Deligne's idea for avoiding this problem is to
use all embeddings (\cite{deligne1979val}, 0.7).

\begin{wrapfigure}{r}{2.8in}
\begin{tikzpicture}
\matrix(m)[matrix of math nodes,
row sep=3.0em, column sep=2.5em,
text height=1.5ex, text depth=0.25ex]
{H_{B}^{2r}(\sigma X)(r) \cap H^{0,0}& H_{\mathbb{\mathbb{C}}\times\mathbb{A}_{f}}^{2r}(\sigma X)(r)\\
AH^{r}(X) & H_{k\times\mathbb{A}_{f}}^{2r}(X)(r)\\};
\path[right hook->,font=\scriptsize,>=angle 90]
(m-1-1) edge node[auto] {(\ref{hq2})} (m-1-2)
(m-2-1) edge node[auto] {} (m-2-2)
(m-2-2) edge node[right]{$\sigma$} node[left] {(\ref{hq1})} (m-1-2);
\path[dashed,->,font=\scriptsize,>=angle 90]
(m-2-1) edge node[auto] {} (m-1-1);
\end{tikzpicture}
\end{wrapfigure}Let $X$ be a smooth complete algebraic variety over an
algebraically closed field $k$ of characteristic zero${}$, and let $\sigma$ be
a homomorphism $k\rightarrow\mathbb{C}{}$. An element $\gamma$ of
$H_{k\times\mathbb{A}{}_{f}}^{2r}(X)(r)$ is a $\sigma$-\emph{Hodge class of
codimension} $r$ if $\sigma\gamma$ lies in the subspace $H_{B}^{2r}(\sigma
X)(r)\cap H^{0,0}$ of $H_{\mathbb{\mathbb{C}{}}\times\mathbb{A}{}_{f}{}}%
^{2r}(\sigma X)(r)$. When $k$ has finite transcendence degree over
$\mathbb{Q}{}$, an element $\gamma$ of $H_{k\times\mathbb{A}{}}^{2r}(X)(r)$ is
an \emph{absolute Hodge class}%
\index{absolute Hodge class}
if it is $\sigma$-Hodge for all homomorphisms $\sigma\colon k\rightarrow
\mathbb{C}{}$. The absolute Hodge classes of codimension $r$ on $X$ form a
$\mathbb{Q}{}$-subspace $AH^{r}(X)$ of $H_{k\times\mathbb{A}{}_{f}}%
^{2r}(X)(r)$.

We list the basic properties of absolute Hodge classes.

\begin{plain}
\label{h71a} The inclusion $AH^{r}(X)\subset H_{k\times\mathbb{A}{}_{f}{}%
}^{2r}(X)(r)$ induces an injective map
\[
\left(  k\times\mathbb{A}{}_{f}{}\right)  \otimes_{\mathbb{Q}{}}%
AH^{r}(X)\rightarrow H_{k\times\mathbb{A}{}_{f}{}{}}^{2r}(X)(r);
\]
in particular $AH^{r}(X)$ is a finite dimensional $\mathbb{Q}{}$-vector space.
\end{plain}

\noindent This follows from (\ref{hq2}) because $AH^{r}(X)$ is isomorphic to a
$\mathbb{Q}{}$-subspace of $H_{B}^{2r}(\sigma X)(r)$ (each $\sigma$).

\begin{plain}
\label{h71b} For any homomorphism $\sigma\colon k\rightarrow k^{\prime}$ of
algebraically closed fields of finite transcendence degree over $\mathbb{Q}{}%
$, the map (\ref{hq1}) induces an isomorphism $AH^{r}(X)\rightarrow
AH^{r}(\sigma X)$ (\cite{deligne1982}, 2.9a).
\end{plain}

This allows us to define $AH^{r}(X)$ for a smooth complete variety over an
arbitrary algebraically closed field $k$ of characteristic zero: choose a
model $X_{0}$ of $X$ over an algebraically closed subfield $k_{0}$ of $k$ of
finite transcendence degree over $\mathbb{Q}{}$, and define $AH^{r}(X)$ to be
the image of $AH^{r}(X_{0})$ under the map $H_{k_{0}\times\mathbb{A}{}_{f}%
}^{2r}(X_{0})(r)\rightarrow H_{k\times\mathbb{A}{}_{f}}^{2r}(X)(r)$. With this
definition, (\ref{h71b}) holds for all homomorphisms of algebraically closed
fields $k$ of characteristic zero. Moreover, if $k$ admits an embedding in
$\mathbb{C}{}$, then a cohomology class is absolutely Hodge if and only if it
is $\sigma$-Hodge for every such embedding.

\begin{plain}
\label{h71c} The cohomology class of an algebraic cycle on $X$ is absolutely
Hodge; thus, the algebraic cohomology classes of codimension $r$ on $X$ form a
$\mathbb{Q}{}$-subspace $A^{r}(X)$ of $AH^{r}(X)$ (\cite{deligne1982}, 2.1a).
\end{plain}

\begin{plain}
\label{h71d} The K\"{u}nneth components of the diagonal are absolute Hodge
classes (ibid., 2.1b).
\end{plain}

\begin{plain}
\label{h71e} Let $X_{0}$ be a model of $X$ over a subfield $k_{0}$ of $k$ such
that $k$ is algebraic over $k_{0}$; then $\Gal(k/k_{0})$ acts on $AH^{r}(X)$
through a finite discrete quotient (ibid. 2.9b).
\end{plain}

\begin{plain}
\label{h71f} Let
\[
AH^{\ast}(X)=\bigoplus\nolimits_{r\geq0}AH^{r}(X);
\]
then $AH^{\ast}(X)$ is a $\mathbb{Q}{}$-subalgebra of $\bigoplus
H_{k\times\mathbb{A}{}_{f}{}}^{2r}(X)(r)$. For any regular map $\alpha\colon
Y\rightarrow X$ of complete smooth varieties, the maps $\alpha_{\ast}$ and
$\alpha^{\ast}$ send absolute Hodge classes to absolute Hodge classes. (This
follows easily from the definitions.)
\end{plain}

\begin{theorem}
[\cite{deligne1982}, 2.12, 2.14]\label{h72}Let $S$ be a smooth connected
algebraic variety over $\mathbb{C}{}$, and let $\pi\colon X\rightarrow S$ be a
smooth proper morphism. Let $\gamma\in\Gamma(S,R^{2r}\pi_{\ast}\mathbb{Q}%
(r))$, and let $\gamma_{s}$ be the image of $\gamma$ in $H_{B}^{2r}(X_{s})(r)$
($s\in S(\mathbb{C}{})$).

\begin{enumerate}
\item \ If $\gamma_{s}$ is a Hodge class for one $s\in S(\mathbb{C})$, then it
is a Hodge class for every $s\in S(\mathbb{C})$.

\item \ If $\gamma_{s}$ is an absolute Hodge class for one $s\in
S(\mathbb{C})$, then it is an absolute Hodge class for every $s\in
S(\mathbb{C})$.
\end{enumerate}
\end{theorem}

\begin{proof}
Let $\bar{X}$ be a smooth compactification of $X$ whose boundary $\bar
{X}\smallsetminus X$ is a union of smooth divisors with normal crossings, and
let $s\in S(\mathbb{C}{})$. According to \cite{deligne1971}, 4.1.1, 4.1.2,
there are maps
\[
\begin{CD}
H_B^{2r}(\bar{X})(r)@>\textrm{onto}>> \Gamma(S,R^{2r}\pi_{\ast}\mathbb{Q}(r))
@>\textrm{injective}>>H_B^{2r}(X_{s})(r)
\end{CD}
\]
whose composite $H_{B}^{2r}(\bar{X})(r)\rightarrow H_{B}^{2r}(X_{s})(r)$ is
defined by the inclusion $X_{s}\hookrightarrow\bar{X}$; moreover
$\Gamma(S,R^{2r}\pi_{\ast}\mathbb{Q}{}(r))$ has a Hodge structure (independent
of $s$) for which the injective maps are morphisms of Hodge structures
(theorem of the fixed part).

Let $\gamma\in\Gamma(S,R^{2r}\pi_{\ast}\mathbb{Q}(r))$. If $\gamma_{s}$ is of
type $(0,0)$ for one $s$, then so also is $\gamma$; then $\gamma_{s}$ is of
type $(0,0)$ for all $s$. This proves (a).

Let $\sigma$ be an automorphism of $\mathbb{C}{}$ (as an abstract field). It
suffices to prove (b) with \textquotedblleft absolute Hodge\textquotedblright%
\ replaced with \textquotedblleft$\sigma$-Hodge\textquotedblright. We shall
use the commutative diagram ($\mathbb{A}{}=\mathbb{C}{}\times\mathbb{A}{}_{f}%
$):%
\[
\begin{CD}
H_{B}^{2r}(\bar{X})(r)@>\mathrm{onto}>>\Gamma(S,R^{2r}\pi_{\ast
}\mathbb{Q}(r)) @>\mathrm{injective}>>H_{B}^{2r}(X_{s})(r)\\
@VV{?\mapsto ?_{\mathbb{A}}}V@VVV@VVV\\
H_{\mathbb{A}}^{2r}(\bar{X})(r)@>\mathrm{onto}>>\Gamma(S,R^{2r}\pi_{\ast
}\mathbb{A}(r)) @>\mathrm{injective}>>H_{\mathbb{A}}^{2r}(X_{s})(r)\\
@VV{\sigma}V@VV{\sigma}V@VV{\sigma}V\\
H_{\mathbb{A}}^{2r}(\sigma\bar{X})(r)@>\mathrm{onto}>>\Gamma(\sigma S,R^{2r}(\sigma\pi)_{\ast
}\mathbb{A}(r)) @>\mathrm{injective}>>H_{\mathbb{A}}^{2r}(\sigma X_{s})(r)\\
@AAA@AAA@AAA\\
H_{B}^{2r}(\sigma\bar{X})(r)@>\mathrm{onto}>>\Gamma(\sigma S,R^{2r}(\sigma\pi)_{\ast
}\mathbb{Q}(r)) @>\mathrm{injective}>>H_{B}^{2r}(\sigma X_{s})(r).\\
\end{CD}
\]
The middle map $\sigma$ uses a relative version of the base change map
(\ref{hq1}). The other maps $\sigma$ are the base change isomorphisms and the
remaining vertical maps are essential tensoring with $\mathbb{A}$ (and are
denoted $?\mapsto?_{\mathbb{A}{}}$).

Let $\gamma$ be an element of $\Gamma(S,R^{2r}\pi_{\ast}\mathbb{Q}(r))$ such
that $\gamma_{s}$ is $\sigma$-Hodge for one $s$. Recall that this means that
there is a $\gamma_{s}^{\sigma}\in H_{B}^{2r}(\sigma X_{s})(r)$ of type
$(0,0)$ such that $\left(  \gamma_{s}^{\sigma}\right)  _{\mathbb{A}{}}%
=\sigma(\gamma_{s})_{\mathbb{A}{}}$ in $H_{\mathbb{A}{}}^{2r}(\sigma
X_{s})(r)$. As $\gamma_{s}$ is in the image of%
\[
H_{B}^{2r}(\bar{X})(r)\rightarrow H_{B}^{2r}(X_{s})(r),
\]
$\sigma(\gamma_{s})_{\mathbb{A}{}}$ is in the image of
\[
H_{\mathbb{A}{}}^{2r}(\sigma\bar{X})(r)\rightarrow H_{\mathbb{A}{}}%
^{2r}(\sigma X_{s})(r)\text{.}%
\]
Therefore $\left(  \gamma_{s}^{\sigma}\right)  _{\mathbb{A}}$ is also, which
implies (by linear algebra\footnote{Apply the following elementary statement:
\bquote Let $E$, $W$, and $V$ be vector spaces, and let $\alpha\colon
W\rightarrow V$ be a linear map; let $v\in V$; if $e\otimes v$ is in the image
of $1\otimes\alpha\colon E\otimes W\rightarrow E\otimes V$ for some nonzero
$e\in E$, then $v$ is in the image of $\alpha$.\equote To prove the statement,
choose an $f\in E^{\vee}$ such that $f(e)=1$. If $\sum e_{i}\otimes
\alpha(w_{i})=e\otimes v$, then $\sum f(e_{i})w_{i}=v$.}) that $\gamma
_{s}^{\sigma}$ is in the image of%
\[
H_{B}^{2r}(\sigma\bar{X})(r)\rightarrow H_{B}^{2r}(\sigma X_{s})(r).
\]
Let $\tilde{\gamma}^{\sigma}$ be a pre-image of $\gamma_{s}^{\sigma}$ in
$H_{B}^{2r}(\sigma\bar{X})(r)$.

Let $s^{\prime}$ be a second point of $S$, and let $\tilde{\gamma}_{s^{\prime
}}^{\sigma}$ be the image of $\tilde{\gamma}^{\sigma}$ in $H_{B}^{2r}(\sigma
X_{s^{\prime}})(r)$. By construction, $(\tilde{\gamma}^{\sigma})_{\mathbb{A}%
{}}$ maps to $\sigma\gamma_{\mathbb{A}{}}$ in $\Gamma(\sigma S,R^{2r}%
(\sigma\pi)_{\ast}\mathbb{A}{}(r))$, and so $\left(  \tilde{\gamma}%
_{s^{\prime}}^{\sigma}\right)  _{\mathbb{A}{}}=\sigma(\gamma_{s^{\prime}%
})_{\mathbb{A}{}}$ in $H_{\mathbb{A}{}}^{2r}(\sigma X_{s^{\prime}})(r)$, which
demonstrates that $\gamma_{s^{\prime}}$ is $\sigma$-Hodge.
\end{proof}

\begin{conjecture}
[Deligne 1979\upshape{a}, 0.10]\nocite{deligne1979val}\label{h73}Every
$\sigma$-Hodge class on a smooth complete variety over an algebraically closed
field of characteristic zero is absolutely Hodge, i.e.,%
\[
\sigma\text{-Hodge (for one }\sigma\text{)}\implies\text{ absolutely Hodge.}%
\]

\end{conjecture}

\begin{theorem}
[\cite{deligne1982}, 2.11]\label{h74}Conjecture \ref{h73} is true for abelian varieties.
\end{theorem}

To prove the theorem, it suffices to show that every Hodge class on an abelian
variety over $\mathbb{C}{}$ is absolutely Hodge.\footnote{Let $A$ be an
abelian variety over $k$, and suppose that $\gamma$ is $\sigma_{0}$-Hodge for
some homomorphism $\sigma_{0}\colon k\rightarrow\mathbb{C}{}$. We have to show
that it is $\sigma$-Hodge for every $\sigma\colon k\rightarrow\mathbb{C}{}$.
But, using the Zorn's lemma, one can show that there exists a homomorphism
$\sigma^{\prime}\colon\mathbb{C}{}\rightarrow\mathbb{C}{}$ such that
$\sigma=\sigma^{\prime}\circ\sigma_{0}$. Now $\gamma$ is $\sigma$-Hodge if and
only if $\sigma_{0}\gamma$ is $\sigma^{\prime}$-Hodge.} We defer the proof of
the theorem to the next subsection.

\begin{aside}
\label{h75}Let $X_{\mathbb{C}{}}$ be a smooth complete algebraic variety over
$\mathbb{C}{}$. Then $X_{\mathbb{C}{}}$ has a model $X_{0}$ over a subfield
$k_{0}$ of $\mathbb{C}{}$ finitely generated over $\mathbb{Q}{}$. Let $k$ be
the algebraic closure of $k_{0}$ in $\mathbb{C}{}$, and let $X=X_{0k}$. For a
prime number $\ell$, let
\[
\mathcal{T}{}_{\ell}^{r}(X)=\bigcup\nolimits_{U}H_{\ell}^{2r}(X)(r)^{U}%
\quad\quad\text{(space of Tate classes)}%
\]
where $U$ runs over the open subgroups of $\Gal(k/k_{0})$ --- as the notation
suggests, $\mathcal{T}{}_{\ell}^{r}(X)$ depends only on $X/k$. The Tate
conjecture (\cite{tate1964}, Conjecture 1) says that the $\mathbb{Q}{}_{\ell}%
$-vector space $\mathcal{T}{}_{\ell}^{r}(X)$ is spanned by algebraic classes.
Statement \ref{h71e} implies that $AH^{r}(X)$ projects into $\mathcal{T}%
{}_{\ell}^{r}(X)$, and (\ref{h71a}) implies that the map $\mathbb{Q}_{\ell
}\otimes_{\mathbb{Q}{}}{}AH^{r}(X)\rightarrow\mathcal{T}{}_{\ell}^{r}(X)$ is
injective. Therefore the Tate conjecture implies that $A^{r}(X)=AH^{r}(X)$,
and so the Tate conjecture for $X$ and one $\ell$ implies that all absolute
Hodge classes on $X_{\mathbb{C}{}}$ are algebraic. Thus, in the presence of
Conjecture \ref{h73}, the Tate conjecture implies the Hodge conjecture. In
particular, Theorem \ref{h74} shows that, for an abelian variety, the Tate
conjecture implies the Hodge conjecture.
\end{aside}

\subsection{Proof of Deligne's theorem}

It is convenient to prove Theorem \ref{h74} in the following more abstract form.

\begin{theorem}
\label{h76}Suppose that for each abelian variety $A$ over $\mathbb{C}{}$ we
have a $\mathbb{Q}{}$-subspace $C^{r}(A)$ of the Hodge classes of codimension
$r$ on $A$. Assume:

\begin{enumerate}
\item $C^{r}(A)$ contains all algebraic classes of codimension $r$ on $A$;

\item pull-back by a homomorphism $\alpha\colon A\rightarrow B$ of abelian
varieties maps $C^{r}(B)$ into $C^{r}(A)$;

\item let $\pi\colon\mathcal{A}{}\rightarrow S$ be an abelian scheme over a
connected smooth complex algebraic variety $S$, and let $t\in\Gamma
(S,R^{2r}\pi_{\ast}\mathbb{Q}(r))$; if $t_{s}$ lies in $C^{r}(A_{s})$ for one
$s\in S(\mathbb{C}{})$, then it lies in $C^{r}(A_{s})$ for all $s$.
\end{enumerate}

\noindent Then $C^{r}(A)$ contains all the Hodge classes of codimension $r$ on
$A$.
\end{theorem}

\begin{corollary}
\label{h77}If hypothesis (c) of the theorem holds for algebraic classes on
abelian varieties, then the Hodge conjecture holds for abelian varieties. (In
other words, for abelian varieties, the variational Hodge conjecture implies
the Hodge conjecture.)
\end{corollary}

\begin{proof}
Immediate consequence of the theorem, because the algebraic classes satisfy
(a) and (b).
\end{proof}

The proof of Theorem \ref{h76} requires four steps.

\subsubsection*{Step 1: The Hodge conjecture holds for powers of an elliptic
curve}

As Tate observed (1964, p.~19),\nocite{tate1964} the $\mathbb{Q}{}$-algebra of
Hodge classes on a power of an elliptic curve is generated by those of type
$(1,1)$.\footnote{This is most conveniently proved by applying the criterion
\cite{milne1999}, 4.8.} These are algebraic by a theorem of Lefschetz.

\subsubsection*{Step 2: Split Weil classes lie in $C$}

Let $A$ be a complex abelian variety, and let $\nu$ be a homomorphism from a
CM-field $E$ into $\End(A)_{\mathbb{Q}{}}$. The pair $(A,\nu)$ is said to be
of \emph{Weil type}%
\index{Weil type}
if the tangent space $T_{0}(A)$ is a free $E\otimes_{\mathbb{Q}{}}\mathbb{C}%
{}$-module. In this case, $d\overset{\textup{{\tiny def}}}{=}\dim_{E}H_{B}%
^{1}(A)$ is even and the subspace $\bigwedge\nolimits_{E}^{d}H_{B}%
^{1}(A)(\frac{d}{2})$ of $H_{B}^{d}(A)(\frac{d}{2})$ consists of Hodge classes
(\cite{deligne1982}, 4.4). When $E$ is quadratic over $\mathbb{Q}{}$, these
Hodge classes were studied by Weil (1977)\nocite{weil1977}, and for this
reason are called \emph{Weil classes}%
\index{Weil classes}%
. A \emph{polarization}%
\index{polarization}
of $(A,\nu)$ is a polarization $\lambda$ of $A$ whose whose Rosati involution
acts on $\nu(E)$ as complex conjugation. The Riemann form of such a
polarization can be written%
\[
(x,y)\mapsto\Tr_{E/\mathbb{Q}{}}(f\phi(x,y))
\]
for some totally imaginary element $f$ of $E$ and $E$-hermitian form $\phi$ on
$H_{1}(A,\mathbb{Q}{})$. If $\lambda$ can be chosen so that $\phi$ is split%
\index{split}
(i.e., admits a totally isotropic subspace of dimension $d/2$), then the Weil
classes are said to be \emph{split}%
\index{split}%
.

\begin{lemma}
\label{h78}All split Weil classes of codimension $r$ on an abelian variety $A$
lie in $C^{r}(A)$.
\end{lemma}

\begin{proof}
Let $(A,\nu,\lambda)$ be a polarized abelian variety of split Weil type. Let
$V=H_{1}(A,\mathbb{Q}{})$, and let $\psi$ be the Riemann form of $\lambda$.
The Hodge structures on $V$ for which the elements of $E$ act as morphisms and
$\psi$ is a polarization are parametrized by a period subdomain, which is
hermitian symmetric domain (cf. \ref{h68}). On dividing by a suitable
arithmetic subgroup, we get a smooth proper map $\pi\colon\mathcal{A}%
{}\rightarrow S$ of smooth algebraic varieties whose fibres are abelian
varieties with an action of $E$ (Theorem \ref{h69b}). There is a $\mathbb{Q}%
{}$-subspace $W$ of $\Gamma(S,R^{d}\pi_{\ast}\mathbb{Q}{}(\frac{d}{2}))$ whose
fibre at every point $s$ is the space of Weil classes on $A_{s}$. One fibre of
$\pi$ is $(A,\nu)$ and another is a power of an elliptic curve. Therefore the
lemma follows from Step 1 and hypotheses (a,c). (See \cite{deligne1982}, 4.8,
for more details.)
\end{proof}

\subsubsection*{Step 3: Theorem \ref{h76} for abelian varieties of CM-type}

A simple abelian variety $A$ is of \emph{CM-type}%
\index{CM-type}
if $\End(A)_{\mathbb{Q}{}}$ is a field of degree $2\dim A$ over $\mathbb{Q}{}%
$, and a general abelian variety is of \emph{CM-type}%
\index{CM-type}
if every simple isogeny factor of it is of CM-type. Equivalently, it is of
CM-type if the Hodge structure $H_{1}(A^{\text{an}},\mathbb{Q}{})$ is of
CM-type. According to \cite{andre1992}:

\begin{quote}
For any complex abelian variety $A$ of CM-type, there exist complex abelian
varieties $B_{J}$ of CM-type and homomorphisms $A\rightarrow B_{J}$ such that
every Hodge class on $A$ is a linear combination of the pull-backs of split
Weil classes on the $B_{J}$.
\end{quote}

\noindent Thus Theorem \ref{h76} for abelian varieties of CM-type follows from
Step 2 and hypothesis (b). (See \cite{deligne1982}, \S 5, for the original
proof of this step.)

\subsubsection*{Step 4: Completion of the proof of Theorem \ref{h76}}

Let $t$ be a Hodge class on a complex abelian variety $A$. Choose a
polarization $\lambda$ for $A$. Let $V=H_{1}(A,\mathbb{Q}{})$ and let $h_{A}$
be the homomorphism defining the Hodge structure on $H_{1}(A,\mathbb{Q}{})$.
Both $t$ and the Riemann form $t_{0}$ of $\lambda$ can be regarded as Hodge
tensors for $V$. The period subdomain $D=D(V,h_{A},\{t,t_{0}\})$ is a
hermitian symmetric domain (see \ref{h68}). On dividing by a suitable
arithmetic subgroup, we get a smooth proper map $\pi\colon\mathcal{A}%
{}\rightarrow S$ of smooth algebraic varieties whose fibres are abelian
varieties (Theorem \ref{h69b}) and a section $t$ of $R^{2r}\pi_{\ast
}\mathbb{Q}{}(r)$. For one $s\in S$, the fibre $(\mathcal{A}{},t)_{s}=(A,t)$,
and another fibre is an abelian variety of CM-type (apply \ref{h91d}), and so
the theorem follows from Step 3 and hypothesis (c). (See \cite{deligne1982},
\S 6, for more details.)

\subsection{Motives for absolute Hodge classes}

We fix a base field $k$ of characteristic zero; \textquotedblleft
variety\textquotedblright\ will mean \textquotedblleft smooth projective
variety over $k$\textquotedblright.

For varieties $X$ and $Y$ with $X$ connected, we let%
\[
C^{r}(X,Y)=AH^{\dim X+r}(X\times Y)
\]
(correspondences of degree $r$ from $X$ to $Y$). When $X$ has connected
components $X_{i}$, $i\in I$, we let%
\[
C^{r}(X,Y)=\bigoplus\nolimits_{i\in I}C^{r}(X_{i},Y).
\]
For varieties $X,Y,Z$, there is a bilinear pairing%
\[
f,g\mapsto g\circ f\colon C^{r}(X,Y)\times C^{s}(Y,Z)\rightarrow C^{r+s}(X,Z)
\]
with%
\[
g\circ f\overset{\textup{{\tiny def}}}{=}(p_{XZ})_{\ast}(p_{XY}^{\ast}f\cdot
p_{YZ}^{\ast}g).
\]
Here the $p$'s are projection maps from $X\times Y\times Z$. These pairings
are associative and so we get a \textquotedblleft category of
correspondences\textquotedblright,\ which has one object $hX$ for every
variety over $k$, and whose Homs are defined by
\[
\Hom(hX,hY)=C^{0}(X,Y).
\]
Let $f\colon Y\rightarrow X$ be a regular map of varieties. The transpose of
the graph of $f$ is an element of $C^{0}(X,Y)$, and so $X\rightsquigarrow hX$
is a contravariant functor.

The category of correspondences is additive, but not abelian, and so we
enlarge it by adding the images of idempotents. More precisely, we define a
\textquotedblleft category of effective motives\textquotedblright,\ which has
one object $h(X,e)$ for each variety $X$ and idempotent $e$ in the ring
$\End(hX)=AH^{\dim X}(X\times X)$, and whose Homs are defined by%
\[
\Hom(h(X,e),h(Y,f))=f\circ C^{0}(X,Y)\circ e.
\]
This contains the old category by $hX\leftrightarrow h(X,\id)$, and $h(X,e)$
is the image of $hX\overset{e}{\longrightarrow}hX$.

The category of effective motives is abelian, but objects need not have duals.
In the enlarged category, the motive $h\mathbb{P}{}^{1}$ decomposes into
$h\mathbb{P}{}^{1}=h^{0}\mathbb{P}{}^{1}\oplus h^{2}\mathbb{P}^{1}$, and it
turns out that, to obtain duals for all objects, we only have to
\textquotedblleft invert\textquotedblright\ the motive $h^{2}\mathbb{P}{}%
^{1}{}$. This is most conveniently done by defining a \textquotedblleft
category of motives\textquotedblright\ which has one object $h(X,e,m)$ for
each pair $(X,e)$ as before and integer $m$, and whose Homs are defined by%
\[
\Hom(h(X,e,m),h(Y,f,n))=f\circ C^{n-m}(X,Y)\circ e.
\]
This contains the old category by $h(X,e)\leftrightarrow h(X,e,0)$.

We now list some properties of the category $\Mot(k)$ of motives.%
\index{Mot(k)@$\Mot(k)$}%

\begin{plain}
\label{h79}The Hom's in $\Mot(k)$ are finite dimensional $\mathbb{Q}{}$-vector
spaces, and $\Mot(k)$ is a semisimple abelian category.
\end{plain}

\begin{plain}
\label{h79a}Define a tensor product on $\Mot(k)$ by%
\[
h(X,e,m)\otimes h(X,f,n)=h(X\times Y,e\times f,m+n).
\]
With the obvious associativity constraint and a suitable\footnote{Not the
obvious one! It is necessary to change some signs.} commutativity constraint,
$\Mot(k)$ becomes a tannakian category.
\end{plain}

\begin{plain}
\label{h79b}The standard cohomology functors factor through $\Mot(k)$. For
example, define%
\[
\omega_{\ell}(h(X,e,m))=e\left(  \bigoplus\nolimits_{i}H_{\ell}^{i}%
(X)(m)\right)
\]
(image of $e$ acting on $\bigoplus\nolimits_{i}H_{\ell}^{i}(X)(m)$). Then
$\omega_{\ell}$ is an exact faithful functor $\Mot(k)\rightarrow
\Vc_{\mathbb{Q}{}_{\ell}}$ commuting with tensor products. Similarly, de Rham
cohomology defines an exact tensor functor $\omega_{\mathrm{dR}}%
\colon\Mot(k)\rightarrow\Vc_{k}$, and, when $k=\mathbb{C}{}$, Betti cohomology
defines an exact tensor functor $\Mot(k)\rightarrow\Vc_{\mathbb{Q}{}}$. The
functors $\omega_{\ell}$, $\omega_{\mathrm{dR}}$, and $\omega_{B}$ are called
the $\ell$-adic, de Rham, and Betti fibre functors, and they send a motive to
its $\ell$-adic, de Rham, or Betti \emph{realization}%
\index{realization}%
.
\end{plain}

The Betti fibre functor on $\Mot(\mathbb{C}{})$ takes values in
$\Hdg_{\mathbb{Q}{}}$, and is faithful (almost by definition). Deligne's
conjecture \ref{h73} is equivalent to saying that it is full.

\subsubsection{Abelian motives}

\begin{definition}
\label{h79c}A motive is \emph{abelian}%
\index{abelian motive}
if it lies in the tannakian subcategory $\Mot^{\mathrm{ab}}(k)$ of $\Mot(k)$
generated by the motives of abelian varieties.
\end{definition}

The Tate motive, being isomorphic to $\bigwedge\nolimits^{2}h_{1}E$ for any
elliptic curve $E$, is an abelian motive. It is known that $h(X)$ is an
abelian motive if $X$ is a curve, a unirational variety of dimension $\leq3$,
a Fermat hypersurface, or a $K3$ surface.

Deligne's theorem \ref{h74} implies that $\omega_{B}\colon\Mot^{\mathrm{ab}%
}(\mathbb{C}{})\rightarrow\Hdg_{\mathbb{Q}{}}$ is fully faithful.

\subsubsection{CM motives}

\begin{definition}
\label{h79e}A motive over $\mathbb{C}{}$ is of \emph{CM-type}%
\index{CM-type!motive}
if its Hodge realization is of CM-type.
\end{definition}

\begin{lemma}
\label{h79f}Every Hodge structure of CM-type is the Betti realization of an
abelian motive.
\end{lemma}

\begin{proof}
Elementary (see, for example, \cite{milne1994a}, 4.6).
\end{proof}

Therefore $\omega_{B}$ defines an equivalence from the category of abelian
motives of CM-type to the category of Hodge structures of CM-type.

\begin{proposition}
\label{h79g}Let $G_{\mathrm{Hdg}}$ (resp. $G_{\mathrm{Mab}}$) be the affine
group scheme attached to $\Hdg_{\mathbb{Q}{}}$ and its forgetful fibre functor
(resp. $\Mot^{\text{ab}}(\mathbb{C}{})$ and its Betti fibre functor). The
kernel of the homomorphism $G_{\mathrm{Hdg}}\rightarrow G_{\mathrm{Mab}}$
defined by the tensor functor $\omega_{B}\colon\Mot^{\text{ab}}(\mathbb{C}%
{})\rightarrow\Hdg_{\mathbb{Q}{}}$ is contained in $(G_{\mathrm{Hdg}%
})^{\mathrm{der}}$.
\end{proposition}

\begin{proof}
Let $S$ be the affine group scheme attached to the category $\Hdg_{\mathbb{Q}%
{}}^{\mathrm{cm}}$ of Hodge structures of CM-type and its forgetful fibre
functor. The lemma shows that the functor $\Hdg_{\mathbb{Q}{}}^{\text{cm}%
}\hookrightarrow\Hdg_{\mathbb{Q}{}}$ factors through $\Mot^{\text{ab}%
}(\mathbb{C}{})\hookrightarrow\Hdg_{\mathbb{Q}{}}$, and so $G_{\mathrm{Hdg}%
}\rightarrow S$ factors through $G_{\mathrm{Hdg}}\rightarrow G_{\mathrm{Mab}}%
$:%
\[
G_{\mathrm{Hdg}}\rightarrow G_{\mathrm{Mab}}\twoheadrightarrow S.
\]
Hence%
\[
\Ker(G_{\mathrm{Hdg}}\rightarrow G_{\mathrm{Mab}})\subset\Ker(G_{\mathrm{Hdg}%
}\twoheadrightarrow S)=\left(  G_{\mathrm{Hdg}}\right)  ^{\mathrm{der}}.
\]

\end{proof}

\subsubsection{Special motives}

\begin{definition}
\label{h80}A motive over $\mathbb{C}{}$ is \emph{special}%
\index{special motive}
if its Hodge realization is special (see p.~\pageref{h54a}).
\end{definition}

It follows from (\ref{h54a}) that the special motives form a tannakian
subcategory of $\Mot(k)$, which includes the abelian motives (see \ref{h54f}).

\begin{question}
\label{h79d}Is every special Hodge structure the Betti realization of a
motive? (Cf. \cite{deligne1979}, p.~248; \cite{langlands1979}, p.~216;
\cite{serre1994}, 8.7.)
\end{question}

More explicitly: for each simple special Hodge structure $(V,h)$, does there
exist an algebraic variety $X$ over $\mathbb{C}$ and an integer $m$ such that
$(V,h)$ is a direct factor of $\bigoplus\nolimits_{r\geq0}H_{B}^{r}(X)(m)$ and
the projection $\bigoplus\nolimits_{r\geq0}H_{B}^{r}(X)(m)\rightarrow
V\subset\bigoplus\nolimits_{r\geq0}H_{B}^{r}(X)(m)$ is an absolute Hodge class
on $X$.

A positive answer to (\ref{h79d}) would imply that all connected Shimura
varieties are moduli varieties for motives (see \S 11). Apparently, no special
motive is known that is not abelian.

\subsubsection{Families of abelian motives}

For an abelian variety $A$ over $k$, let%
\[
\omega_{f}(A)=\varprojlim A_{N}(k^{\mathrm{al}}),\quad A_{N}(k^{\mathrm{al}%
})=\Ker(N\colon A(k^{\mathrm{al}})\rightarrow A(k^{\mathrm{al}}))\text{.}%
\]
This is a free $\mathbb{A}{}_{f}$-module of rank $2\dim A$ with a continuous
action of $\Gal(k^{\mathrm{al}}/k)$.

Let $S$ be a smooth connected variety over $k$, and let $k(S)$ be its function
field. Fix an algebraic closure $k(S)^{\mathrm{al}}$ of $k(S)$, and let
$k(S)^{\mathrm{un}}$ be the union of the subfields $L$ of $k(S)^{\mathrm{al}}$
such that the normalization of $S$ in $L$ is unramified over $S$. We say that
an action of $\Gal(k(S)^{\mathrm{al}}/k(S))$ on a module is \emph{unramified}%
\index{unramified}
if it factors throught $\Gal(k(S)^{\mathrm{un}}/k(S))$.

\begin{theorem}
\label{h80a}Let $S$ be a smooth connected variety over $k$. The functor
$A\rightsquigarrow A_{\eta}\overset{\textup{{\tiny def}}}{=}A_{k(S)}$ is a
fully faithful functor from the category of families of abelian varieties over
$S$ to the category of abelian varieties over $k(S)$, with essential image the
abelian varieties $B$ over $k(S)$ such that $\omega_{f}(B)$ is unramified.
\end{theorem}

\begin{proof}
When $S$ has dimension $1$, this follows from the theory of N\'{e}ron models.
In general, this theory shows that an abelian variety (or a morphism of
abelian varieties) extends to an open subvariety $U$ of $S$ such that
$S\smallsetminus U$ has codimension at least $2$. Now we can
apply\footnote{Recall that we are assuming that the base field has
characteristic zero --- the theorem is false without that condition.}
\cite{chaiF1990}, I 2.7, V 6.8.
\end{proof}

The functor $\omega_{f}$ extends to a functor on abelian motives such that
$\omega_{f}(h_{1}A)=\omega_{f}(A)$ if $A$ is an abelian variety.

\begin{definition}
\label{h80e}Let $S$ be a smooth connected variety over $k$. A \emph{family}%
\index{family of abelian motives}
$M$ \emph{of abelian motives} over $S$ is an abelian motive $M_{\eta}$ over
$k(S)$ such that $\omega_{f}(M_{\eta})$ is unramified.
\end{definition}

Let $M$ be a family of motives over a smooth connected variety $S$, and let
$\bar{\eta}=\Spec(k(S)^{\mathrm{al}})$. The fundamental group $\pi_{1}%
(S,\bar{\eta})=\Gal(k(S)^{\mathrm{un}}/k(S))$, and so the representation of
$\pi_{1}(S,\bar{\eta})$ on $\omega_{f}(M_{\eta})$ defines a local system of
$\mathbb{A}{}_{f}$-modules $\omega_{f}(M)$. Less obvious is that, when the
ground field is $\mathbb{C}{}$, $M$ defines a polarizable variation of Hodge
structures on $S$, $\mathcal{H}{}_{B}(M/S)$. When $M$ can be represented in
the form $(A,p,m)$ on $S$, this is obvious. However, $M$ can always be
represented in this fashion on an open subset of $S$, and the underlying local
system of $\mathbb{Q}{}$-vector spaces extends to the whole of $S$ because the
monodromy representation is unramified. Now it is possible to show that the
variation of Hodge structures itself extends (uniquely) to the whole of $S$,
by using results from \cite{schmid1973}, \cite{cattaniKS1986}, and
\cite{griffithsS1969}. See \cite{milne1994}, 2.40, for the details.

\begin{theorem}
\label{h80b}Let $S$ be a smooth connected variety over $\mathbb{C}{}$. The
functor sending a family $M$ of abelian motives over $S$ to its associated
polarizable Hodge structure is fully faithful, with essential image the
variations of Hodge structures $(\mathsf{V},F)$ such that there exists a dense
open subset $U$ of $S$, an integer $m$, and a family of abelian varieties
$f\colon A\rightarrow S$ such that $(\mathsf{V},F)$ is a direct summand of
$Rf_{\ast}\mathbb{Q}{}$.
\end{theorem}

\begin{proof}
This follows from the similar statement (\ref{h69b}) for families of abelian
varieties (see \cite{milne1994}, 2.42).
\end{proof}

\section{Symplectic Representations}

\begin{quote}
{\small In this subsection, we classify the symplectic representations of
groups. These were studied by Satake in a series of papers (see especially
\cite{satake1965,satake1967,satake1980}). Our exposition follows that of
\cite{deligne1979}.}
\end{quote}

In \S 8 we proved that there exists a correspondence between variations of
Hodge structures on locally symmetric varieties and certain commutative
diagrams%
\begin{equation}
\begin{tikzpicture} \matrix(m)[matrix of math nodes, row sep=3em, column sep=2.5em, text height=1.5ex, text depth=0.25ex] {H\\ (H^{\text{ad}},\bar{h})&(G,h)&\GL_V\\}; \path[->,font=\scriptsize,>=angle 90] (m-1-1) edge (m-2-1) (m-1-1) edge node[auto] {} (m-2-2) (m-2-2) edge (m-2-1); \path[right hook->,font=\scriptsize,>=angle 90] (m-2-2) edge node[auto]{$\rho$} (m-2-3); \end{tikzpicture}
\end{equation}
In this section, we study whether there exists such a diagram and a
nondegenerate alternating form $\psi$ on $V$ such that $\rho(G)\subset
G(\psi)$ and $\rho_{\mathbb{R}{}}\circ h\in D(\psi)$. Here $G(\psi)$ is the
group of \emph{symplectic similitudes }%
\index{symplectic similitudes }%
(algebraic subgroup of $\GL_{V}$ whose elements fix $\psi$ up to a scalar) and
$D(\psi)$ is the Siegel upper half space (set of Hodge structures $h$ on $V$
of type $\{(-1,0),(0,-1)\}$ for which $2\pi i\psi$ is a
polarization\footnote{This description agrees with that in \S 2 because of the
correspondence in (\ref{h47}).}). Note that $G(\psi)$ is a reductive group
whose derived group is the symplectic group $S(\psi)$.

\subsection{Preliminaries}

\begin{plain}
The \emph{universal covering torus}%
\index{universal covering torus}
$\tilde{T}$ of a torus $T$ is the projective system $(T_{n},T_{nm}%
\xrightarrow{m}T_{n})$ in which $T_{n}=T$ for all $n$ and the indexing set is
$\mathbb{N}{}\smallsetminus\{0\}$ ordered by divisibility. For any algebraic
group $G$,%
\[
\Hom(\tilde{T},G)=\varinjlim_{n\geq1}\Hom(T_{n},G)\text{.}%
\]
Concretely, a homomorphism $\tilde{T}\rightarrow G$ is represented by a pair
$(f,n)$ in which $f$ is a homomorphism $T\rightarrow G$ and $n\in\mathbb{N}%
{}\smallsetminus\{0\}$; two pairs $(f,n)$ and $(g,m)$ represent the same
homomorphism $\tilde{T}\rightarrow G$ if and only if $f\circ m=g\circ n$. A
homomorphism $f\colon\tilde{T}\rightarrow G$ factors through $T$ if and only
if it is represented by a pair $(f,1)$. A homomorphism $\mathbb{\tilde{G}}%
_{m}\rightarrow\GL_{V}$ represented by $(\mu,n)$ defines a gradation
$V=\bigoplus V_{r}$, $r\in\frac{1}{n}\mathbb{Z}{}$; here $V_{\frac{a}{n}%
}=\{v\in V\mid\mu(t)v=t^{a}v\}$; the $r$ for which $V_{r}\neq0$ are called the
\emph{weights}%
\index{weights}
the representation of $\mathbb{\tilde{G}}_{m}$ on $V$. Similarly, a
homomorphism $\mathbb{\tilde{S}}{}\rightarrow\GL_{V}$ represented by $(h,n)$
defines a fractional Hodge decomposition $V_{\mathbb{C}{}}=\bigoplus V^{p,q}$
with $p,q\in\frac{1}{n}\mathbb{Z}{}$.
\end{plain}

\subsection{The real case}

Throughout this subsection, $H$ is a simply connected real algebraic group
without compact factors, and $\bar{h}$ is a homomorphism $\mathbb{S}%
/\mathbb{G}_{m}\rightarrow H^{\mathrm{ad}}$ satisfying the conditions (SV1,2),
p.~\pageref{SV}, and whose projection on each simple factor of $H^{\mathrm{ad}%
}$ is nontrivial.

\begin{definition}
\label{h94}A homomorphism $H\rightarrow\GL_{V}$ with finite kernel is a
\emph{symplectic\/ representation}%
\index{symplectic\/ representation}
of $(H,\bar{h})$ if there exists a commutative diagram%
\[
\begin{tikzpicture}
\matrix(m)[matrix of math nodes, row sep=3em, column sep=2.5em,
text height=1.5ex, text depth=0.25ex]
{H\\
(H^{\text{ad}},\bar{h})&(G,h)&(G(\psi ),D(\psi )),\\};
\path[->,font=\scriptsize,>=angle 60]
(m-1-1) edge  (m-2-1)
(m-1-1) edge  (m-2-2)
(m-2-2) edge  (m-2-1)
(m-2-2) edge  (m-2-3);
\end{tikzpicture}
\]
in which $\psi$ is a nondegenerate alternating form on $V$, $G$ is a reductive
group, and $h$ is a homomorphism $\mathbb{S}{}\rightarrow G$; the homomorphism
$H\rightarrow G$ is required to have image $G^{\mathrm{der}}$.
\end{definition}

In other words, there exists a real reductive group $G$, a nondegenerate
alternating form $\psi$ on $V$, and a factorization%
\[
H\overset{a}{\longrightarrow}G\overset{b}{\longrightarrow}\GL_{V}%
\]
of $H\rightarrow\GL_{V}$ such that $a(H)=G^{\mathrm{der}}$, $b(G)\subset
G(\psi)$, and $b\circ h\in D(\psi)$; the isogeny $H\rightarrow G^{\mathrm{der}%
}$ induces an isomorphism $H^{\mathrm{ad}}\overset{c}{\longrightarrow
}G^{\mathrm{ad}}$ (see footnote \ref{ad}, p.~\pageref{ad}), and it is required
that $\bar{h}=c^{-1}\circ\ad\circ h$.

We shall determine the complex representations of $H$ that occur in the
complexification of a symplectic representation (and we shall omit
\textquotedblleft the complexification of\textquotedblright).

\begin{proposition}
\label{h94a}A homomorphism $H\rightarrow\GL_{V}$ with finite kernel is a
symplectic\/ representation of $(H,\bar{h})$ if there exists a commutative
diagram%
\[
\begin{tikzpicture}
\matrix(m)[matrix of math nodes, row sep=3em, column sep=2.5em,
text height=1.5ex, text depth=0.25ex]
{H\\
(H^{\text{ad}},\bar{h})&(G,h)&\GL_V,\\};
\path[->,font=\scriptsize,>=angle 60]
(m-1-1) edge  (m-2-1)
(m-1-1) edge  (m-2-2)
(m-2-2) edge  (m-2-1)
(m-2-2) edge  node[auto]{$\rho$} (m-2-3);
\end{tikzpicture}
\]
in which $G$ is a reductive group, the homomorphism $H\rightarrow G$ has image
$G^{\mathrm{der}}$, and $(V,\rho\circ h)$ has type $\{(-1,0),(0,-1)\}$.
\end{proposition}

\begin{proof}
Let $G^{\prime}$ be the algebraic subgroup of $G$ generated by
$G^{\mathrm{der}}$ and $h(\mathbb{S)}$. After replacing $G$ with $G^{\prime}$,
we may suppose that $G$ itself is generated by $G^{\mathrm{der}}$ and
$h(\mathbb{S}{})$. Then $(G,h)$ satisfies (SV2*), and it follows from Theorem
\ref{h20b} that there exists a polarization $\psi$ of $(V,\rho\circ h)$ such
that $G$ maps into $G(\psi)$ (cf. the proof of \ref{h52}).
\end{proof}

Let $(H,\bar{h})$ be as before. The cocharacter $\mu_{\bar{h}}$ of
$H_{\mathbb{C}{}}^{\mathrm{ad}}$ lifts to a fractional cocharacter $\tilde
{\mu}$ of $H_{\mathbb{C}}$:%
\[
\begin{CD}
\mathbb{\tilde{G}}_{m} @>{\tilde{\mu}}>> H_{\mathbb{C}}\\
@VVV@VV{\ad}V\\
\mathbb{G}_{m} @>>{\mu_{\bar{h}}}> H_{\mathbb{C}{}}^{\mathrm{ad}}.
\end{CD}
\]

\begin{lemma}
\label{h94c}If an irreducible complex representation $W$ of $H$ occurs in a
symplectic representation, then $\tilde{\mu}$ has at most two weights $a$ and
$a+1$ on $W$.
\end{lemma}

\begin{proof}
Let $H\overset{\varphi}{\longrightarrow}(G,h)\overset{}{\longrightarrow
}\GL_{V}$ be a symplectic representation of $(H,\bar{h})$, and let $W$ be an
irreducible direct summand of $V_{\mathbb{C}{}}$. The homomorphisms
$\varphi_{\mathbb{C}{}}\circ\tilde{\mu}\colon\mathbb{\tilde{G}}_{m}\rightarrow
G_{\mathbb{C}{}}$ coincides with $\mu_{h}$ when composed with $G_{\mathbb{C}%
{}}\rightarrow G_{\mathbb{C}{}}^{\mathrm{ad}}$, and so $\varphi_{\mathbb{C}{}%
}\circ\tilde{\mu}=\mu_{h}\cdot\nu$ with $\nu$ central. On $V$, $\mu_{h}$ has
weights $0,1$. If $a$ is the unique weight of $\nu$ on $W$, then the only
weights of $\tilde{\mu}$ on $W$ are $a$ and $a+1$.
\end{proof}

\begin{lemma}
\label{h94g}Assume that $H$ is almost simple. A nontrivial irreducible complex
representation $W$ of $H$ occurs in a symplectic representation if and only if
$\tilde{\mu}$ has exactly two weights $a$ and $a+1$ on $W$.
\end{lemma}

\begin{proof}
$\Rightarrow$: Let $(\mu,n)$ represent $\tilde{\mu}$. As $H_{\mathbb{C}{}}$ is
almost simple and $W$ nontrivial, the homomorphism $\mathbb{G}_{m}%
\rightarrow\GL_{W}$ defined by $\mu$ is nontrivial, therefore noncentral, and
the two weights $a$ and $a+1$ occur.

$\Leftarrow$: Let $(W,r)$ be an irreducible complex representation of $H$ with
weights $a,a+1$, and let $V$ be the real vector space underlying $W$. Define
$G$ to be the subgroup of $\GL_{V}$ generated by the image of $H$ and the
homotheties: $G=r(H)\cdot\mathbb{G}_{m}$. Let $\tilde{h}$ be a fractional
lifting of $\bar{h}$ to $\tilde{H}$:%
\[
\begin{CD}
\mathbb{\tilde{\mathbb{S}}} @>{\tilde{h}}>> H_{\mathbb{C}}\\
@VVV@VV{\ad}V\\
\mathbb{S}@>>{\bar h}> H_{\mathbb{C}}^{\mathrm{ad}}.
\end{CD}
\]
Let $W_{a}$ and $W_{a+1}$ be the subspaces of weight $a$ and $a+1$ of $W$.
Then $\tilde{h}(z)$ acts on $W_{a}$ as $(z/\bar{z})^{a}$ and on $W_{a+1}$ as
$(z/\bar{z})^{a+1}$, and so $h(z)\overset{\textup{{\tiny def}}}{=}\tilde
{h}(z)z^{-a}\bar{z}^{1+a}$ acts on these spaces as $\bar{z}$ and $z$
respectively. Therefore $h$ is a true homomorphism $\mathbb{S}{}\rightarrow
G$, projecting to $\bar{h}$ on $H^{\text{ad}}$, and $V$ is of type
$\{(-1,0),(0,-1)\}$ relative to $h$. We may now apply Lemma \ref{h94a}.
\end{proof}

We interprete the condition in Lemma \ref{h94g} in terms of roots and weights.
Let $\bar{\mu}=\mu_{\bar{h}}$. Fix a maximal torus $T$ in $H_{\mathbb{C}}$,
and let $R=R(H,T)\subset X^{\ast}(T)_{\mathbb{Q}{}}$ be the corresponding root
system. Choose a base $S$ for $R$ such that $\langle\alpha,\bar{\mu}%
\rangle\,\geq0$ for all $\alpha\in S$ (cf. \S 2).

Recall that, for each $\alpha\in R$, there exists a unique $\alpha^{\vee}\in
X_{\ast}(T)_{\mathbb{Q}{}}$ such that $\langle\alpha,\alpha^{\vee}\rangle=2$
and the symmetry $s_{\alpha}\colon x\mapsto x-\langle x,\alpha^{\vee}%
\rangle\alpha$ preserves $R$; moreover, for all $\alpha\in R$, $\langle
R,\alpha^{\vee}\rangle\subset\mathbb{Z}{}$. The lattice of weights is
\[
P(R)=\{\varpi\in X^{\ast}(T)_{\mathbb{Q}{}}\mid\langle\varpi,\alpha^{\vee
}\rangle\,\in\mathbb{Z}\text{ all }\alpha\in R\},
\]
the fundamental weights are the elements of the dual basis $\{\varpi
_{1},\ldots,\varpi_{n}\}$ to $\{\alpha_{1}^{\vee},\ldots,\alpha_{n}^{\vee}\}$,
and that the dominant weights are the elements $\sum n_{i}\varpi_{i}$,
$n_{i}\in\mathbb{N}$. The quotient $P(R)/Q(R)$ of $P(R)$ by the lattice $Q(R)$
generated by $R$ is the character group of $Z(H)$:%
\[
P(R)/Q(R)\simeq X^{\ast}(Z(H)).
\]

The irreducible complex representations of $H$ are classified by the dominant
weights. We shall determine the dominant weights of the irreducible complex
representations such that $\tilde{\mu}$ has exactly two weights $a$ and $a+1$.

There is a unique permutation $\tau$ of the simple roots, called the
\emph{opposition involution}%
\index{opposition involution}%
, such that the $\tau^{2}=1$ and the map $\alpha\mapsto-\tau(\alpha)$ extends
to an action of the Weyl group. Its action on the Dynkin diagram is determined
by the following rules: it preserves each connected component; on a connected
component of type $A_{n}$, $D_{n}$ ($n$ odd), or $E_{6}$, it acts as the
unique nontrivial involution, and on all other connected components, it acts
trivially (\cite{tits1966}, 1.5.1). Thus:\label{opp} [See Fig1 at end.]

\begin{proposition}
\label{h94b}Let $W$ be an irreducible complex representation of $H$, and let
$\varpi$ be its highest weight. The representation $W$ occurs in a symplectic
representation if and only if
\begin{equation}
\langle\varpi+\tau\varpi,\bar{\mu}\rangle\,=1.\quad\label{hq43}%
\end{equation}

\end{proposition}

\begin{proof}
The lowest weight of $W$ is $-\tau(\varpi)$. The weights $\beta$ of $W$ are of
the form%
\[
\beta=\varpi+\sum_{\alpha\in R}m_{\alpha}\alpha\text{,}\quad m_{\alpha}%
\in\mathbb{Z}\text{,}%
\]
and%
\[
\langle\beta,\bar{\mu}\rangle\in\mathbb{Z}{}\text{.}%
\]
Thus, $\langle\beta,\bar{\mu}\rangle$ takes only two values $a,a+1$ if and
only if
\[
\langle-\tau(\varpi),\bar{\mu}\rangle=\langle\varpi,\bar{\mu}\rangle-1,
\]
i.e., if and only if (\ref{hq43}) holds.
\end{proof}

\begin{corollary}
\label{h94d}If $W$ is symplectic, then $\varpi$ is a fundamental weight.
Therefore the representation factors through an almost simple quotient of $H$.
\end{corollary}

\begin{proof}
For every dominant weight $\varpi$, $\langle\varpi+\tau\varpi,\bar{\mu}%
\rangle\,\in\mathbb{Z}$ because $\varpi+\tau\varpi\in Q(R)$. If $\varpi\neq0$,
then $\langle\varpi+\tau\varpi,\bar{\mu}\rangle\,>0$ unless $\bar{\mu}$ kills
all the weights of the representation corresponding to $\varpi$. Hence a
dominant weight satisfying (\ref{hq43}) can not be a sum of two dominant weights.
\end{proof}

The corollary allows us to assume that $H$ is almost simple. Recall from \S 2
that there is a unique special simple root $\alpha_{s}$ such that, for
$\alpha\in S$,%
\[
\langle\alpha,\bar{\mu}\rangle=\left\{
\begin{array}
[c]{ll}%
1 & \text{if }\alpha=\alpha_{s}\\
0 & \text{otherwise.}%
\end{array}
\right.
\]
When a weight $\varpi$ is expressed as a $\mathbb{Q}$-linear combination of
the simple roots, $\langle\varpi,\bar{\mu}\rangle$ is the coefficient of
$\alpha_{s}$. For the fundamental weights, these coefficients can be found in
the tables in \cite{bourbakiLie}, VI. A fundamental weight $\varpi$ satisfies
(\ref{hq43}) if and only if%
\begin{equation}
(\text{coefficient of }\alpha_{s}\text{ in }\varpi+\tau\varpi)=1. \label{hq46}%
\end{equation}

In the following, we write $\alpha_{1},\ldots,\alpha_{n}$ for the simple roots
and $\varpi_{1},\ldots,\varpi_{n}$ for the fundamental weights with the usual
numbering. 

\subsubsection{Type $A_{n}$.}

The opposition involution $\tau$ switches the nodes $i$ and $n+1-i$. According
to the tables in Bourbaki, for $1\leq i\leq(n+1)/2,$%
\[
\varpi_{i}=\tstyle\frac{n+1-i}{n+1}\alpha_{1}+\frac{2(n+1-i)}{n+1}\alpha
_{2}+\cdots+\frac{i(n+1-i)}{n+1}\alpha_{i}+\cdots+\frac{2i}{n+1}\alpha
_{n-1}+\frac{i}{n+1}\alpha_{n}.
\]
Replacing $i$ with $n+1-i$ reflects the coefficients, and so%
\[
\tau\varpi_{i}=\varpi_{n+1-i}=\tstyle\frac{i}{n+1}\alpha_{1}+\frac{2i}%
{n+1}\alpha_{2}+\cdots+\frac{2(n+1-i)}{n+1}\alpha_{n-1}+\frac{n+1-i}%
{n+1}\alpha_{n}.
\]
Therefore,%
\[
\varpi_{i}+\tau\varpi_{i}=\alpha_{1}+2\alpha_{2}+\cdots+i\alpha_{i}%
+i\alpha_{i+1}+\cdots+i\alpha_{n+1-(i+1)}+i\alpha_{n+1-i}+\cdots+2\alpha
_{n-1}+\alpha_{n},
\]
i.e., the sequence of coefficients is%
\[
(1,2,\ldots,i,i,\ldots,i,i,\ldots,2,1).
\]
Let $\alpha_{s}=\alpha_{1}$ or $\alpha_{n}$. Then every fundamental weight
satisfies (\ref{hq46}):\footnote{\cite{deligne1979}, Table 1.3.9, overlooks
this possibility.} [See Fig2 at end.]

\medskip\noindent Let $\alpha_{s}=\alpha_{j}$, with $1<j<n$. Then only the
fundamental weights $\varpi_{1}$ and $\varpi_{n}$ satisfy (\ref{hq46}): [See
Fig3 at end.]

\medskip\noindent As $P/Q$ is generated by $\varpi_{1}$, the symplectic
representations form a faithful family.

\subsubsection{Type $B_{n}$.}

In this case, $\alpha_{s}=\alpha_{1}$ and the opposition involution acts
trivially on the Dynkin diagram, and so we seek a fundamental weight
$\varpi_{i}$ such that $\varpi_{i}=\frac{1}{2}\alpha_{1}+\cdots$. According to
the tables in Bourbaki,%
\begin{align*}
\varpi_{i}  &  =\alpha_{1}+2\alpha_{2}+\cdots+(i-1)\alpha_{i-1}+i(\alpha
_{i}+\alpha_{i+1}+\cdots+\alpha_{n})\quad(1\leq i<n)\\
\varpi_{n}  &  =\tstyle\frac{1}{2}(\alpha_{1}+2\alpha_{2}+\cdots+n\alpha_{n}),
\end{align*}
and so only $\varpi_{n}$ satisfies (\ref{hq46}): [See Fig4 at end.]

\noindent As $P/Q$ is generated by $\varpi_{n}$, the symplectic
representations form a faithful family.

\subsubsection{Type $C_{n}$.}

In this case $\alpha_{s}=\alpha_{n}$ and the opposition involution acts
trivially on the Dynkin diagram, and so we seek a fundamental weight
$\varpi_{i}$ such that $\varpi_{i}=\cdots+\frac{1}{2}\alpha_{n}$. According to
the tables in Bourbaki,%
\[
\varpi_{i}=\alpha_{1}+2\alpha_{2}+\cdots+(i-1)\alpha_{i-1}+i(\alpha_{i}%
+\alpha_{i+1}+\cdots+\alpha_{n-1}+\tstyle\tstyle\frac{1}{2}\alpha_{n}),
\]
and so only $\varpi_{1}$ satisfies (\ref{hq46}): [See Fig5 at end.]

\noindent As $P/Q$ is generated by $\varpi_{1}$, the symplectic
representations form a faithful family.

\subsubsection{Type $D_{n}$.}

The opposition involution acts trivially if $n$ is even, and switches
$\alpha_{n-1}$ and $\alpha_{n}$ if $n$ is odd. According to the tables in
Bourbaki,%
\begin{align*}
\varpi_{i}  &  =\alpha_{1}+2\alpha_{2}+\cdots+(i-1)\alpha_{i-1}+i(\alpha
_{i}+\cdots+\alpha_{n-2})+\tstyle\frac{i}{2}(\alpha_{n-1}+\alpha_{n}%
),\quad1\leq i\leq n-2\\
\varpi_{n-1}  &  =\tstyle\frac{1}{2}\left(  \alpha_{1}+2\alpha_{2}%
+\cdots+(n-2)\alpha_{n-2}+\tstyle\frac{1}{2}n\alpha_{n-1}+\tstyle\frac{1}%
{2}(n-2)\alpha_{n}\right) \\
\varpi_{n}  &  =\tstyle\frac{1}{2}\left(  \alpha_{1}+2\alpha_{2}%
+\cdots+(n-2)\alpha_{n-2}+\tstyle\frac{1}{2}(n-2)\alpha_{n-1}+\tstyle\frac
{1}{2}n\alpha_{n}\right)
\end{align*}

Let $\alpha_{s}=\alpha_{1}$. As $\alpha_{1}$ is fixed by the opposition
involution, we seek a fundamental weight $\varpi_{i}$ such that $\varpi
_{i}=\tstyle\frac{1}{2}\alpha_{1}+\cdots$. Both $\varpi_{n-1}$ and $\varpi
_{n}$ give rise to symplectic representations: [See Fig6 at end.]

\noindent When $n$ is odd, $\varpi_{n-1}$ and $\varpi_{n}$ each generates
$P/Q$, and when $n$ is even $\varpi_{n-1}$ and $\varpi_{n}$ together generate
$P/Q$. Therefore, in both cases, the symplectic representations form a
faithful family.

Let $\alpha_{s}=\alpha_{n-1}$ or $\alpha_{n}$ and let $n=4$. The nodes
$\alpha_{1}$, $\alpha_{3}$, and $\alpha_{4}$ are permuted by automorphisms of
the Dynkin diagram (hence by outer automorphisms of the corresponding group),
and so this case is the same as the case $\alpha_{s}=\alpha_{1}$: [See Fig7 at end.]

\noindent The symplectic representations form a faithful family.

Let $\alpha_{s}=\alpha_{n-1}$ or $\alpha_{n}$ and let $n\geq5$. When $n$ is
odd, $\tau$ interchanges $\alpha_{n-1}$ and $\alpha_{n}$, and so we seek a
fundamental weight $\varpi_{i}$ such that $\varpi_{i}=\cdots+a\alpha
_{n-1}+b\alpha_{n}$ with $a+b=1$; when $n$ is even, $\tau$ is trivial, and we
seek a fundamental weight $\varpi_{i}$ such that $\varpi_{i}=\cdots
+\tstyle\frac{1}{2}\alpha_{n-1}+\cdots$ or $\cdots+\tstyle\frac{1}{2}%
\alpha_{n}$. In each case, only $\varpi_{1}$ gives rise to a symplectic
representation: [see Fig8 at end].

\noindent The weight $\varpi_{1}$ generates a subgroup of order $2$ (and index
2) in $P/Q$. Let $C\subset Z(H)$ be the kernel of $\varpi_{1}$ regarded as a
character of $Z(H)$. Then every symplectic representation factors through
$H/C$, and the symplectic representations form a faithful family of
representations of $H/C$.

\subsubsection{Type $E_{6}$.}

In this case, $\alpha_{s}=\alpha_{1}$ or $\alpha_{6}$, and the opposition
involution interchanges $\alpha_{1}$ and $\alpha_{6}$. Therefore, we seek a
fundamental weight $\varpi_{i}$ such that $\varpi_{i}=a\alpha_{1}%
+\cdots+b\alpha_{6}$ with $a+b=1$. In the following diagram, we list the value
$a+b$ for each fundamental weight $\varpi_{i}$: [See Fig9 at end.]

\noindent As no value equals $1$, there are no symplectic representations.

\subsubsection{Type $E_{7}$.}

In this case, $\alpha_{s}=\alpha_{7}$, and the opposition involution is
trivial. Therefore, we seek a fundamental weight $\varpi_{i}$ such that
$\varpi_{i}=\cdots+\tstyle\frac{1}{2}\alpha_{7}$. In the following diagram, we
list the coefficient of $\alpha_{7}$ for each fundamental weight $\varpi_{i}$:
[See Fig10 at end.]

\noindent As no value is $\frac{1}{2}$, there are no symplectic
representations.\bigskip

Following \cite{deligne1979}, 1.3.9, we write $D^{\mathbb{R}{}}$ for the case
$D_{n}(1)$ and $D^{\mathbb{H}{}}$ for the cases $D_{n}(n-1)$ and $D_{n}(n)$.

\begin{summary}
\label{h94e}Let $H$ be a simply connected almost simple group over
$\mathbb{R}$, and let $\bar{h}\colon\mathbb{S}/\mathbb{G}_{m}\rightarrow
H^{\text{ad}}$ be a nontrivial homomorphism satisfying (SV1,2). There exists a
symplectic representation of $(H,\bar{h})$ if and only if it is of type $A$,
$B$, $C$, or $D$. Except when $(H,\bar{h})$ is of type $D_{n}^{\mathbb{H}}$,
$n\geq5$, the symplectic representations form a faithful family of
representations of $H$; when $(H,\bar{h})$ is of type $D_{n}^{\mathbb{H}{}}$,
$n\geq5$, they form a faithful family of representations of the quotient of
the simply connected group by the kernel of $\varpi_{1}$.
\end{summary}

\subsection{The rational case}

Now let $H$ be a semisimple algebraic group over $\mathbb{Q}{}$, and let
$\bar{h}$ be a homomorphism $\mathbb{S}{}/\mathbb{G}_{m}\rightarrow
H_{\mathbb{R}{}}^{\mathrm{ad}}$ satisfying (SV1,2) and generating
$H^{\mathrm{ad}}$.

\begin{definition}
\label{h95}A homomorphism $H\rightarrow\GL_{V}$ with finite kernel is a
\emph{symplectic\/ representation}%
\index{symplectic\/ representation}
of $(H,\bar{h})$ if there exists a commutative diagram%
\begin{equation}
\begin{tikzpicture} \matrix(m)[matrix of math nodes, row sep=3em, column sep=2.5em, text height=1.5ex, text depth=0.25ex] {H\\ (H^{\text{ad}},\bar{h})&(G,h)&(G(\psi ),D(\psi )),\\}; \path[->,font=\scriptsize,>=angle 90] (m-1-1) edge (m-2-1) (m-1-1) edge (m-2-2) (m-2-2) edge (m-2-1); \path[->,font=\scriptsize,>=angle 90] (m-2-2) edge node[above]{$\rho$} (m-2-3); \end{tikzpicture} \label{hq95}%
\end{equation}
in which $\psi$ is a nondegenerate alternating form on $V$, $G$ is a reductive
group (over $\mathbb{Q}$), and $h$ is a homomorphism $\mathbb{S}{}\rightarrow
G_{\mathbb{R}{}}$; the homomorphism $H\rightarrow G$ is required to have image
$G^{\mathrm{der}}$,
\end{definition}

Given a diagram (\ref{hq95}), we may replace $G$ with its image in $\GL_{V}$
and so assume that the representation $\rho$ is faithful.

\textit{We now assume that} $H$ \textit{is simply connected and almost
simple}. Then $H=\left(  H^{s}\right)  _{F/\mathbb{Q}{}}$ for some
geometrically almost simple algebraic group $H^{s}$ over a number field $F$.
Because $H_{\mathbb{R}{}}$ is an inner form of its compact form, the field $F$
is totally real (see the proof of \ref{h37a}). Let $I=\Hom(F,\mathbb{R}{})$.
Then,%
\[
H_{\mathbb{R}}=\prod_{v\in I}H_{v},\quad H_{v}=H^{s}\otimes_{F,v}\mathbb{R}.
\]
The Dynkin diagram $D$ of $H_{\mathbb{C}{}}$ is a disjoint union of the Dynkin
diagrams $D_{v}$ of the group $H_{v\mathbb{C}{}}$. The Galois group
$\Gal(\mathbb{Q}^{\text{al}}/\mathbb{Q})$ acts on it in a manner consistent
with its projection to $I$. In particular, it acts transitively on $D$ and so
all the factors $H_{v}$ of $H_{\mathbb{R}{}}$ are of the same type. We let
$I_{\text{c}}$ (resp. $I_{\mathrm{nc}}$) denote the subset of $I$ of $v$ for
which $H_{v}$ is compact (resp. not compact), and we let $H_{\mathrm{c}}%
=\prod_{v\in I_{\mathrm{c}}}H_{v}$ and $H_{\mathrm{nc}}=\prod_{v\in
I_{\mathrm{nc}}}H_{v}$. Because $\bar{h}$ generates $H^{\text{ad}}$,
$I_{\text{nc}}$ is nonempty.

\begin{proposition}
\label{h95a}Let $F$ be a totally real number field. Suppose that for each real
prime $v$ of $F$, we are given a pair $(H_{v},\bar{h}_{v})$ in which $H_{v}$
is a simply connected algebraic group over $\mathbb{R}{}$ of a fixed type, and
$\bar{h}_{v}$ is a homomorphism $\mathbb{S}{}/\mathbb{G}_{m}\rightarrow
H_{v}^{\mathrm{ad}}$ satisfying (SV1,2) (possibly trivial). Then there exists
an algebraic group $H$ over $\mathbb{Q}{}$ such that $H\otimes_{F,v}%
\mathbb{R}\approx H_{v}$ for all $v$.
\end{proposition}

\begin{proof}
There exists an algebraic group $H$ over $F$ such that $H\otimes
_{F,v}\mathbb{R}{}$ is an inner form of its compact form for all real primes
$v$ of $F$. For each such $v$, $H_{v}$ is an inner form of $H\otimes
_{F,v}\mathbb{R}$, and so defines a cohomology class in $H^{1}(F_{v}%
,H^{\mathrm{ad}})$. The proposition now follows from the surjectivity of the
map%
\[
H^{1}(F,H^{\mathrm{ad}})\rightarrow\prod\nolimits_{v\text{ real}}H^{1}%
(F_{v},H^{\mathrm{ad}})
\]
(\cite{prasadR2006}, Proposition 1).
\end{proof}

\subsubsection{Pairs $(H,\bar{h})$ for which there do
not exist symplectic representations}

\paragraph{$H$ is of exceptional type}

Assume that $H$ is of exceptional type. If there exists an $\bar{h}$
satisfying (SV1,2), then $H$ is of type $E_{6}$ or $E_{7}$ (see \S 2). A
symplectic representation of $(H,\bar{h})$ over $\mathbb{Q}{}$ gives rise to a
symplectic representation of $(H_{\mathbb{R}{}},\bar{h})$ over $\mathbb{R}{}$,
but we have seen (\ref{h94e}) that no such representations exist.\label{not}

\paragraph{$(H,\bar{h})$ is of mixed type $D$.}

By this we mean that $H$ is of type $D_{n}$ with $n\geq5$ and that at least
one factor $(H_{v},\bar{h}_{v})$ is of type $D_{n}^{\mathbb{R}{}}$ and one of
type $D_{n}^{\mathbb{H}{}}$. Such pairs $(H,\bar{h})$ exist by Proposition
\ref{h95a}. The Dynkin diagram of $H_{\mathbb{R}{}}$ contains connected
components [See Fig11 at end.]

\noindent or $D_{n}(n-1)$. To give a symplectic representation for
$H_{\mathbb{R}{}}$, we have to choose a symplectic node for each real prime
$v$ such that $H_{v}$ is noncompact. In order for the representation to be
rational, the collection of symplectic nodes must be stable under
$\Gal(\mathbb{Q}{}^{\mathrm{al}}/\mathbb{Q}{})$, but this is impossible,
because there is no automorphism of the Dynkin diagram of type $D_{n}$,
$n\geq5$, carrying the node $1$ into either the node $n-1$ or the node $n$.

\subsubsection{Pairs $(H,\bar{h})$ for which there
exist symplectic representations}

\begin{lemma}
\label{h95c}Let $G$ be a reductive group over $\mathbb{Q}{}$ and let $h$ be a
homomorphism $\mathbb{S}{}\rightarrow G_{\mathbb{R}{}}$ satisfying (SV1,2*)
and generating $G$. For any representation $(V,\rho)$ of $G$ such that
$(V,\rho\circ h)$ is of type $\{(-1,0),(0,-1)\}$, there exists an alternating
form $\psi$ on $V$ such that $\rho$ induces a homomorphism $(G,h)\rightarrow
(G(\psi),D(\psi))$.
\end{lemma}

\begin{proof}
The pair $(\rho G,\rho\circ h)$ is the Mumford-Tate group of $(V,\rho\circ h)$
and satisfies (SV2*). The proof of Proposition \ref{h52} constructs a
polarization $\psi$ for $(V,\rho\circ h)$ such that $\rho G\subset G(\psi)$.
\end{proof}

\begin{proposition}
\label{h95d}A homomorphism $H\rightarrow\GL_{V}$ is a symplectic\/
representation of $(H,\bar{h})$ if there exists a commutative diagram%
\[
\begin{tikzpicture}
\matrix(m)[matrix of math nodes, row sep=3em, column sep=2.5em,
text height=1.5ex, text depth=0.25ex]
{H\\
(H^{\text{ad}},\bar{h})&(G,h)&\GL_V,\\};
\path[->,font=\scriptsize,>=angle 90]
(m-1-1) edge  (m-2-1)
(m-1-1) edge  (m-2-2)
(m-2-2) edge  (m-2-1);
\path[right hook->,font=\scriptsize,>=angle 90]
(m-2-2) edge  node[above]{$\rho$} (m-2-3);
\end{tikzpicture}
\]
in which $G$ is a reductive group whose connected centre splits over a
CM-field, the homomorphism $H\rightarrow G$ has image $G^{\mathrm{der}}$, the
weight $w_{h}$ is defined over $\mathbb{Q}{}$, and the Hodge structure
$(V,\rho\circ h)$ is of type $\{(-1,0),(0,-1)\}$.
\end{proposition}

\begin{proof}
The hypothesis on the connected centre $Z^{\circ}$ says that the largest
compact subtorus of $Z_{\mathbb{R}{}}^{\circ}$ is defined over $\mathbb{Q}{}$.
Take $G^{\prime}$ to be the subgroup of $G$ generated by this torus,
$G^{\text{der}}$, and the image of $w_{h}$. Now $(G^{\prime},h)$ satisfies
(SV2*), and we can apply \ref{h95c}.
\end{proof}

We classify the symplectic representations of $(H,\bar{h})$ with $\rho$
faithful. Note that the quotient of $H$ acting faithfully on $V$ is isomorphic
to $G^{\text{der}}$.

Let $(V,r)$ be a symplectic representation of $(H,\bar{h})$. The restriction
of the representation to $H_{\text{nc}}$ is a real symplectic representation
of $H_{\text{nc}}$, and so, according to Corollary \ref{h94d}, every
nontrivial irreducible direct summand of $r_{\mathbb{C}}|H_{\text{nc}}$
factors through $H_{v}$ for some $v\in I_{\text{nc}}$ and corresponds to a
symplectic node of the Dynkin diagram $D_{v}$ of $H_{v}$.

\begin{wrapfigure}[6]{r}{1.5in}
\vspace{-0.5cm}
\begin{tikzpicture}
\matrix(m)[matrix of math nodes,
row sep=3.0em, column sep=3.0em,
text height=1.5ex, text depth=0.25ex]
{D\\I&T\\};
\path[->,font=\scriptsize,>=angle 90]
(m-1-1) edge (m-2-1)
(m-2-2) edge node[right]{s} (m-1-1);
\path[left hook->]
(m-2-2) edge (m-2-1);
\end{tikzpicture}
\end{wrapfigure}Let $W$ be an irreducible direct summand of $V_{\mathbb{C}{}}%
$. Then
\[
W\approx\bigotimes\nolimits_{v\in T}W_{v}%
\]
for some irreducible symplectic representations $W_{v}$ of $H_{v\mathbb{C}}$
indexed by a subset $T$ of $I$. The irreducible representation $W_{v}$
corresponds to a symplectic node $s(v)$ of $D_{v}$. Because $r$ is defined
over $\mathbb{Q}{}$, the set $s(T)$ is stable under the action of
$\Gal(\mathbb{Q}^{\text{al}}/\mathbb{Q})$. For $v\in I_{\mathrm{nc}}$, the set
$s(T)\cap D_{v}$ consists of a single symplectic node.

Given a diagram (\ref{hq95}), we let $\mathcal{S}{}(V)$ denote the set of
subsets $s(T)$ of the nodes of $D$ as $W$ runs over the irreducible direct
summands of $V$. The set $\mathcal{S}{}(V)$ satisfies the following
conditions:\refstepcounter{X}\label{h95bb}

\begin{description}
\item[\textnf{(\theX a)}] for $S\in\mathcal{S}{}(V)$, $S\cap D_{\mathrm{nc}}$
is either empty or consists of a single symplectic node of $D_{v}$ for some
$v\in I_{\mathrm{nc}}$;

\item[\textnf{(\theX b)}] $\mathcal{S}{}$ is stable under $\Gal(\mathbb{Q}%
{}^{\text{al}}/\mathbb{Q}{})$ and contains a nonempty subset.
\end{description}

Given such a set $\mathcal{S}{}$, let $H(\mathcal{S}{})_{\mathbb{C}{}}$ be the
quotient of $H_{\mathbb{C}{}}$ that acts faithfully on the representation
defined by $\mathcal{S}{}$. The condition (\ref{h95bb}b) ensures that
$H(\mathcal{S}{})$ is defined over $\mathbb{Q}{}$. According to Galois theory
(in the sense of Grothendieck), there exists an \'{e}tale $\mathbb{Q}{}%
$-algebra $K_{\mathcal{S}{}}$ such that
\[
\Hom(K_{\mathcal{S}{}},\mathbb{Q}{}^{\mathrm{al}}{})\simeq\mathcal{S}%
{}\text{\hspace{1cm}(as sets with an action of }\Gal(\mathbb{Q}{}%
^{\mathrm{al}}/\mathbb{Q}{})\text{)}.
\]

\begin{theorem}
\label{h95e}For any set $\mathcal{S}{}$ satisfying the conditions
(\ref{h95bb}), there exists a diagram (\ref{hq95}) such that the quotient of
$H$ acting faithfully on $V$ is $H(\mathcal{S}{})$.
\end{theorem}

\begin{proof}
We prove this only in the case that $\mathcal{S}{}$ consists of one-point
sets. For an $\mathcal{S}{}$ as in the theorem, the set $\mathcal{S}{}%
^{\prime}$ of $\{s\}$ for $s\in S\in\mathcal{S}{}$ satisfies (\ref{h95bb}) and
$H(\mathcal{S}{})$ is a quotient of $H(\mathcal{S}{}^{\prime})$.

Recall that $H=\left(  H^{s}\right)  _{F/\mathbb{Q}{}}$ for some totally real
field $F$. We choose a totally imaginary quadratic extension $E$ of $F$ and,
for each real embedding $v$ of $F$ in $I_{\mathrm{c}}$, we choose an extension
$\sigma$ of $v$ to a complex embedding of $E$. Let $T$ denote the set of
$\sigma$'s. Thus%
\[%
\begin{array}
[c]{ccc}%
E & \overset{\sigma}{\longrightarrow} & \mathbb{C}{}\\
\cup &  & \cup\\
F & \overset{v}{\longrightarrow} & \mathbb{R}%
\end{array}
\quad\quad T=\{\sigma\mid v\in I_{c}\}.
\]

We regard $E$ as a $\mathbb{Q}{}$-vector space, and define a Hodge structure
$h_{T}$ on it as follows: $E\otimes_{\mathbb{Q}{}}\mathbb{C}{}\simeq
\mathbb{C}{}^{\Hom(E,\mathbb{C}{})}$ and the factor with index $\sigma$ is of
type $(-1,0)$ if $\sigma\in T$, type $(0,-1)$ if $\bar{\sigma}\in T$, and of
type $(0,0)$ if $\sigma$ lies above $I_{\mathrm{nc}}$. Thus ($\mathbb{C}%
{}_{\sigma}=\mathbb{C}{}$):
\[
\renewcommand{\arraystretch}{1.4}%
\begin{array}
[c]{ccccccc}%
E\otimes_{\mathbb{Q}{}}\mathbb{C} & = & \bigoplus\nolimits_{\sigma\in
T}\mathbb{C}{}_{\sigma} & \oplus & \bigoplus\nolimits_{\bar{\sigma}\in
T}\mathbb{C}{}_{\sigma} & \oplus & \bigoplus\nolimits_{\sigma\notin T\cup
\bar{T}}\mathbb{C}{}_{\sigma}.\\
h_{T}(z) &  & z &  & \bar{z} &  & 1
\end{array}
\]

Because the elements of $\mathcal{S}{}$ are one-point subsets of $D$, we can
identify them with elements of $D$, and so regard $\mathcal{S}{}$ as a subset
of $D$. It has the properties:

\begin{enumerate}
\item if $s\in\mathcal{S}{}\cap D_{\mathrm{nc}}$, then $s$ is a symplectic node;

\item $\mathcal{S}{}$ is stable under $\Gal(\mathbb{Q}{}^{\mathrm{al}%
}/\mathbb{Q}{})$ and is nonempty.
\end{enumerate}

Let $K_{D}$ be the smallest subfield of $\mathbb{Q}{}^{\mathrm{al}}$ such that
$\Gal(\mathbb{Q}{}^{\mathrm{al}}/K_{D})$ acts trivially on $D$. Then $K_{D}$
is a Galois extension of $\mathbb{Q}{}$ in $\mathbb{Q}{}^{\mathrm{al}}$ such
that $\Gal(K_{D}/K)$ acts faithfully on $D$. Complex conjugation acts as the
opposition involution on $D$, which lies in the centre of $\Aut(D)$; therefore
$K_{D}$ is either totally real or CM.

The $\mathbb{Q}{}$-algebra $K_{\mathcal{S}{}}$ can be taken to be a product of
subfields of $K_{D}$. In particular, $K_{\mathcal{S}{}}$ is a product of
totally real fields and CM fields. The projection $\mathcal{S}{}\rightarrow I$
corresponds to a homomorphism $F\rightarrow K_{\mathcal{S}{}}$.

For $s\in\mathcal{S}{}$, let $V(s)$ be a complex representation of
$H_{\mathbb{C}{}}$ with dominant weight the fundamental weight corresponding
to $s$. The isomorphism class of the representation $\bigoplus_{s\in
\mathcal{S}{}}V(s)$ is defined over $\mathbb{Q}{}$. The obstruction to the
representation itself being defined over $\mathbb{Q}{}$ lies in the Brauer
group of $\mathbb{Q}{}$, which is torsion, and so some multiple of the
representation is defined over $\mathbb{Q}{}$. Let $V$ be a representation of
$H$ over $\mathbb{Q}{}$ such that $V_{\mathbb{C}{}}\approx\bigoplus
_{s\in\mathcal{S}{}}nV(s)$ for some integer $n$, and let $V_{s}$ denote the
direct summand of $V_{\mathbb{C}{}}$ isomorphic to $nV(s)$. These summands are
permuted by $\Gal(\mathbb{Q}{}^{\mathrm{al}}/\mathbb{Q}{})$ in a fashion
compatible with the action of $\Gal(\mathbb{Q}^{\mathrm{al}}/\mathbb{Q}{})$ on
$\mathcal{S}{}$, and the decomposition $V_{\mathbb{C}{}}=\bigoplus
_{s\in\mathcal{S}}V_{s}$ corresponds therefore to a structure of a
$K_{\mathcal{S}{}}$-module on $V$: let $s^{\prime}\colon K_{\mathcal{S}{}%
}\rightarrow\mathbb{Q}^{\mathrm{al}}$ be the homomorphism corresponding to
$s\in\mathcal{S}{}$; then $a\in K_{\mathcal{S}{}}$ acts on $V_{s}$ as
multiplication by $s^{\prime}(a)$.

Let $H^{\prime}$ denote the quotient of $H$ that acts faithfully on $V$. Then
$H_{\mathbb{R}{}}^{\prime}$ is the quotient of $H_{\mathbb{R}{}}$ described in
(\ref{h94e}).

A lifting of $\bar{h}$ to a fractional morphism of $\mathbb{S}{}$ into
$H_{\mathbb{R}{}}^{\prime}$ defines a fractional Hodge structure on $V$ of
weight $0$, which can be described as follows. Let $s\in\mathcal{S}{}$, and
let $v$ be its image in $I$; if $v\in I_{\mathrm{c}}$, then $V_{s}$ is of type
$(0,0)$; if $v\in I_{\mathrm{nc}}$, then $V_{s}$ is of type
$\{(r,-r),(r-1,1-r)\}$ where $r=\langle\varpi_{s},\bar{\mu}\rangle$ (notations
as in \ref{h94b}). We renumber this Hodge structure to obtain a new Hodge
structure on $V$:%
\[
\renewcommand{\arraystretch}{1.2}%
\begin{array}
[c]{|c|c|c|}\hline
& \text{old} & \text{new}\\\hline
V_{s}\text{, }v\in I_{\mathrm{c}} & (0,0) & (0,0)\\\hline
V_{s}\text{, }v\in I_{\mathrm{nc}} & (r,-r) & (0,-1)\\\hline
V_{s}\text{, }v\in I_{\mathrm{nc}} & (r-1,1-r) & (-1,0).\\\hline
\end{array}
\]

We endow the $\mathbb{Q}{}$-vector space $E\otimes_{F}V$ with the tensor
product Hodge structure. The decomposition%
\[
(E\otimes_{F}V)\otimes_{\mathbb{Q}{}}\mathbb{R}{}=\bigoplus\nolimits_{v\in
I}(E\otimes_{F,v}\mathbb{R}{})\otimes_{\mathbb{R}{}}(V\otimes_{F,v}%
\mathbb{R}{}),
\]
is compatible with the Hodge structures. The type of the Hodge structure on
each direct summand is given by the following table:%
\[
\renewcommand{\arraystretch}{1.5}%
\begin{array}
[c]{|c|c|c|}\hline
& E\otimes_{F,v}\mathbb{R}{} & V\otimes_{F,v}\mathbb{R}\\\hline
v\in I_{\mathrm{c}} & \{(-1,0),(0,-1)\} & \{(0,0)\}\\\hline
v\in I_{\mathrm{nc}} & \{(0,0)\} & \{(-1,0),(0,-1)\}.\\\hline
\end{array}
\]
Therefore, $E\otimes_{F}V$ has type $\{(-1,0),(0,-1)\}$. Let $G$ be the
algebraic subgroup of $\GL_{E\otimes_{F}V}$ generated by $E^{\times}$ and
$H^{\prime}$. The homomorphism $h\colon\mathbb{S}{}\rightarrow(\GL_{E\otimes
_{F}V})_{\mathbb{R}{}}$ corresponding to the Hodge structure factors through
$G_{\mathbb{R}{}}$, and the derived group of $G$ is $H^{\prime}$. Now apply
(\ref{h95d}).
\end{proof}

\begin{aside}
\label{h96}The trick of using a quadratic imaginary extension $E$ of $F$ in
order to obtain a Hodge structure of type $\{(-1,0),(0,-1)\}$ from one of type
$\{(-1,0),(0,0),(0,-1)\}$ in essence goes back to Shimura (cf.
\cite{deligne1971}, \S 6).
\end{aside}

\subsubsection{Conclusion}

Now let $H$ be a semisimple algebraic group over $\mathbb{Q}{}$, and let
$\bar{h}$ be a homomorphism $\mathbb{S}{}\rightarrow H_{\mathbb{R}{}%
}^{\mathrm{ad}}$ satisfying (SV1,2) and generating $H$.

\begin{definition}
\label{h97}The pair $(H,\bar{h})$ is of \emph{Hodge type}%
\index{Hodge type}
if it admits a faithful family of symplectic representations.
\end{definition}

\begin{theorem}
\label{h98}A pair $(H,\bar{h})$ is of Hodge type if it is a product of pairs
$(H_{i},\bar{h}_{i})$ such that either

\begin{enumerate}
\item $(H_{i},\bar{h}_{i})$ is of type $A$, $B$, $C$, or $D^{\mathbb{R}{}}$,
and $H$ is simply connected, or \noindent

\item $(H_{i},\bar{h}_{i})$ is of type $D_{n}^{\mathbb{H}{}}$ $(n\geq5)$ and
equals $\left(  H^{s}\right)  _{F/\mathbb{Q}{}}$ for the quotient $H^{s}$ of
the simply connected group of type $D_{n}^{\mathbb{H}{}}$ by the kernel of
$\varpi_{1}$ (cf. \ref{h94e}).
\end{enumerate}

\noindent Conversely, if $(H,\bar{h})$ is a Hodge type, then it is a quotient
of a product of pairs satisfying (a) or (b).
\end{theorem}

\begin{proof}
Suppose that $(H,\bar{h})$ is a product of pairs satisfying (a) and (b), and
let $(H^{\prime},\bar{h}^{\prime})$ be one of these factors with $H^{\prime}$
almost simple. Let $\tilde{H}^{\prime}$ be the simply connected covering group
of $H$. Then (\ref{h94e}) allows us to choose a set $\mathcal{S}{}$ satisfying
(\ref{h95bb}) and such that $H^{\prime}=H(\mathcal{S}{})$. Now Theorem
\ref{h95e} shows that $(H^{\prime},\bar{h}^{\prime})$ admits a faithful
symplectic representation. A product of pairs of Hodge type is clearly of
Hodge type.

Conversely, suppose that $(H,\bar{h})$ is of Hodge type, let $\tilde{H}$ be
the simply connected covering group of $H$, and let $(H^{\prime},\bar
{h}^{\prime})$ be an almost simple factor of $(\tilde{H},\bar{h})$. Then
$(H^{\prime},\bar{h}^{\prime})$ admits a symplectic representation with finite
kernel, and so $(H^{\prime},\bar{h}^{\prime})$ is not of type $E_{6}$, $E_{7}%
$, or mixed type $D$ (see p.~\pageref{not}). Moreover, if $(H^{\prime},\bar
{h}^{\prime})$ is of type $D_{n}^{\mathbb{H}{}}$, $n\geq5$, then (\ref{h94e})
shows that it factors through the quotient described in (b).
\end{proof}

Notice that we haven't completely classified the pairs $(H,\bar{h})$ of Hodge
type because we haven't determined exactly which quotients of products of
pairs satisfying (a) or (b) occur as $H(\mathcal{S}{})$ for some set
$\mathcal{S}{}$ satisfying (\ref{h95bb}).

\section{Moduli}

\begin{quote}
{\small In this section, we determine (a) the pairs $(G,h)$ that arise as the
Mumford-Tate group of an abelian variety (or an abelian motive); (b) the
arithmetic locally symmetric varieties that carry a faithful family of abelian
varieties (or abelian motives); (c) the Shimura varieties that arise as moduli
varieties for polarized abelian varieties (or motives) with Hodge class and
level structure. }
\end{quote}

\subsection{Mumford-Tate groups}

\begin{theorem}
\label{h11a}Let $G$ be an algebraic group over $\mathbb{Q}{}$, and let
$h\colon\mathbb{S}{}\rightarrow G_{\mathbb{R}{}}$ be a homomorphism that
generates $G$ and whose weight is rational. The pair $(G,h)$ is the
Mumford-Tate group of an abelian variety if and only if $h$ satisfies (SV2*)
and there exists a faithful representation $\rho\colon G\rightarrow\GL_{V}$
such that $(V,\rho\circ h)$ is of type $\{(-1,0),(0,-1)\}$
\end{theorem}

\begin{proof}
The necessity is obvious (apply (\ref{h52}) to see that $(G,h)$ satisfies
(SV2*)). For the sufficiency, note that $(G,h)$ is the Mumford-Tate group of
$(V,\rho\circ h)$ because $h$ generates $G$. The Hodge structure is
polarizable because $(G,h)$ satisfies (SV2*) (apply \ref{h52}), and so it is
the Hodge structure $H_{1}(A^{\text{an}},\mathbb{Q}{})$ of an abelian variety
$A$ by Riemann's theorem \ref{h44}.
\end{proof}

The Mumford-Tate group%
\index{Mumford-Tate group!of an abelian motive}
of a motive is defined to be the Mumford-Tate group of its Betti realization.

\begin{theorem}
\label{h11b}Let $(G,h)$ be an algebraic group over $\mathbb{Q}{}$, and let
$h\colon\mathbb{S}{}\rightarrow G_{\mathbb{R}{}}$ be a homomorphism satisfying
(SV1,2*) and generating $G{}$. Assume that $w_{h}$ is defined over
$\mathbb{Q}{}$. The pair $(G,h)$ is the Mumford-Tate group of an abelian
motive if and only if $(G^{\mathrm{der}},\bar{h})$ is a quotient of a product
of pairs satisfying (a) and (b) of (\ref{h98}).
\end{theorem}

The proof will occupy the rest of this subsection. Recall that
$G_{\mathrm{Hdg}}$ is the affine group scheme attached to the tannakian
category $\Hdg_{\mathbb{Q}{}}$ of polarizable rational Hodge structures and
the forgetful fibre functor (see \ref{h79g}). It is equipped with a
homomorphism $h_{\mathrm{Hdg}}\colon\mathbb{S}{}\rightarrow(G_{\mathrm{Hdg}%
})_{\mathbb{R}{}}$. If $(G,h)$ is the Mumford-Tate group of a polarizable
Hodge structure, then there is a unique homomorphism $\rho(h)\colon
G_{\mathrm{Hdg}}\rightarrow G$ such that $h=\rho(h)_{\mathbb{R}{}}\circ
h_{\mathrm{Hdg}}$. Moreover, $(G_{\mathrm{Hdg}},h_{\mathrm{Hdg}}%
)=\varprojlim(G,h)$.

\begin{lemma}
\label{h11c}Let $H$ be a semisimple algebraic group over $\mathbb{Q}{}$, and
let $\bar{h}\colon\mathbb{S}{}/\mathbb{G}_{m}\rightarrow H_{\mathbb{R}{}%
}^{\mathrm{ad}}$ be a homomorphism satisfying (SV1,2,3). There exists a unique
homomorphism
\[
\rho(H,\bar{h})\colon\left(  G_{\mathrm{Hdg}}\right)  ^{\mathrm{der}%
}\rightarrow H
\]
such that the following diagram commutes:%
\[
\begin{CD}
(G_{\mathrm{Hdg}})^{\mathrm{der}} @>{\rho(H,\bar{h})}>>H\\
@VVV@VVV\\
G_{\mathrm{Hdg}}@>>{\rho(\bar{h})}>H^{\mathrm{ad}}.
\end{CD}
\]

\end{lemma}

\begin{proof}
Two such homomorphisms $\rho(H,\bar{h})$ would differ by a map into $Z(H)$.
Because $(G_{\mathrm{Hdg}})^{\mathrm{der}}$ is connected, any such map is
constant, and so the homomorphisms are equal.

For the existence, choose a pair $(G,h)$ as in (\ref{h92c}). Then $(G,h)$ is
the Mumford-Tate group of a polarizable Hodge structure, and we can take
$\rho(H,\bar{h})=\rho(h)|(G_{\mathrm{Hdg}})^{\mathrm{der}}$.
\end{proof}

\begin{lemma}
\label{h11d}The assignment $(H,\bar{h})\mapsto\rho(H,\bar{h})$ is functorial:
if $\alpha\colon H\rightarrow H^{\prime}$ is a homomorphism mapping $Z(H)$
into $Z(H^{\prime})$ and carrying $\bar{h}$ to $\bar{h}^{\prime}$, then
$\rho(H^{\prime},\bar{h}^{\prime})=\alpha\circ\rho(H,\bar{h})$.
\end{lemma}

\begin{proof}
The homomorphism $\bar{h}^{\prime}$ generates $H^{\prime\mathrm{ad}}$ (by
SV3), and so the homomorphism $\alpha$ is surjective. Choose a pair $(G,h)$
for $(H,\bar{h})$ as in (\ref{h92c}), and let $G^{\prime}=G/\Ker(\alpha)$.
Write $\alpha$ again for the projection $G\rightarrow G^{\prime}$ and let
$h^{\prime}=\alpha_{\mathbb{R}}\circ h$. This equality implies that%
\[
\rho(h^{\prime})=\alpha\circ\rho(h).
\]
On restricting this to $(G_{\text{Hdg}})^{\text{der}}$, we obtain the
equality
\[
\rho(H^{\prime},\bar{h}^{\prime})=\alpha\circ\rho(H,\bar{h}).\text{{}}%
\]

\end{proof}

Recall that $G_{\mathrm{Mab}}$ is the affine group scheme attached to the
category of abelian motives over $\mathbb{C}{}$ and the Betti fibre functor.
The functor $\Mot^{\mathrm{ab}}(\mathbb{C}{})\rightarrow\Hdg_{\mathbb{Q}{}}$
is fully faithful by Deligne's theorem (\ref{h74}), and so it induces a
surjective map $G_{\mathrm{Hdg}}\rightarrow G_{\mathrm{Mab}}$.

\begin{lemma}
\label{h11e}If $(H,h)$ is of Hodge type, then $\rho(H,\bar{h})$ factors
through $(G_{\mathrm{Mab}})^{\mathrm{der}}$.
\end{lemma}

\begin{proof}
Let $(G,h)$ be as in the definition (\ref{h95}), and replace $G$ with the
algebraic subgroup generated by $h$. Then $(G,h)$ is the Mumford-Tate group of
an abelian variety (Riemann's theorem \ref{h44}), and so $\rho(h)\colon
G_{\mathrm{Hdg}}\rightarrow G$ factors through $G_{\mathrm{Hdg}}\rightarrow
G_{\mathrm{Mab}}$. Therefore $\rho(H,\bar{h})$ maps the kernel of $\left(
G_{\mathrm{Hdg}}\right)  ^{\mathrm{der}}\rightarrow\left(  G_{\mathrm{Mab}%
}\right)  ^{\mathrm{der}}$ into the kernel of $H\rightarrow G$. By assumption,
the intersection of these kernels is trivial.
\end{proof}

\begin{lemma}
\label{h11f}The homomorphism $\rho(H,\bar{h})$ factors through
$(G_{\mathrm{Mab}})^{\mathrm{der}}$ if and only if $(H,\bar{h})$ has a finite
covering by a pair of Hodge type.
\end{lemma}

\begin{proof}
Suppose that there is a finite covering $\alpha\colon H^{\prime}\rightarrow H$
such that $(H^{\prime},\bar{h})$ is of Hodge type. By Lemma \ref{h11e},
$\rho(H^{\prime},\bar{h})$ factors through $(G_{\mathrm{Mab}})^{\mathrm{der}}%
$, and therefore so also does $\rho(H,\bar{h})=\alpha\circ\rho(H^{\prime}%
,\bar{h})$.

Conversely, suppose that $\rho(H,\bar{h})$ factors through $(G_{\text{Mab}%
})^{\text{der}}$. There will be an algebraic quotient $(G,h)$ of
$(G_{\text{Mab}},h_{\text{Mab}})$ such that $(H,\bar{h})$ is a quotient of
$(G^{\text{der}},\ad\circ h)$. Consider the category of abelian motives $M$
such that the action of $G_{\text{Mab}}$ on $\omega_{B}(M)$ factors through
$G$. By definition, this category is contained in the tensor category
generated by $h_{1}(A)$ for some abelian variety $A$. We can replace $G$ with
the Mumford-Tate group of $A$. Then $(G^{\text{der}},\ad\circ h)$ has a
faithful symplectic embedding, and so it is of Hodge type.
\end{proof}

We can now complete the proof of the Theorem \ref{h11b}. From (\ref{h79g}), we
know that $\rho(h)$ factors through $G_{\text{Mab}}$ if and only if
$\rho(G^{\text{der}},\ad\circ h)$ factors through $(G_{\text{Mab}%
})^{\text{der}}$, and from (\ref{h11f}) we know that this is true if and only
if $(G^{\text{der}},\ad\circ h)$ has a finite covering by a pair of Hodge type.

\begin{aside}
\label{h11g}Let $G$ be an algebraic group over $\mathbb{Q}{}$ and let $h$ be a
homomorphism $\mathbb{S}{}\rightarrow G_{\mathbb{R}{}}$. If $(G,h)$ is the
Mumford-Tate group of a motive, then $h$ generates $G$, $w_{h}$ is defined
over $\mathbb{Q}{}$, and $h$ satisfies (SV2*). Assume that $(G,h)$ satisfies
these conditions. A positive answer to Question \ref{h79d} would imply that
$(G,h)$ is the Mumford-Tate group of a motive if $h$ satisfies (SV1). If
$G^{\mathrm{der}}$ is of type $E_{8}$, $F_{4}$, or $G_{2}$, then there does
not exist an $h$ satisfying (SV1) (apply \S 2 to $h|\mathbb{S}{}^{1}$).
Nevertheless, it has recently been shown that there exist motives whose
Mumford-Tate group is of type $G_{2}$ (\cite{dettweiler2010}).
\end{aside}

\begin{nt}
This subsection follows \S 1 of \cite{milne1994}.
\end{nt}

\subsection{Families of abelian varieties and motives}

Let $S$ be a connected smooth algebraic variety over $\mathbb{C}{}$, and let
$o\in S(\mathbb{C}{})$. A family $f\colon A\rightarrow S$ of abelian varieties
over $S$ defines a local system $\mathsf{V}=R_{1}f_{\ast}\mathbb{Z}{}$ of
$\mathbb{Z}{}$-modules on $S^{\text{an}}$. We say that the family is
\emph{faithful}%
\index{faithful family of representations}
if the monodromy representation $\pi_{1}(S^{\text{an}},o)\rightarrow
\GL(\mathsf{V}_{o})$ is injective.

Let $D(\Gamma)=\Gamma\backslash D$ be an arithmetic locally symmetric variety,
and let $o\in D$. By definition, there exists a simply connected algebraic
group $H$ over $\mathbb{Q}{}$ and a surjective homomorphism $\varphi\colon
H(\mathbb{R})\rightarrow\mathrm{Hol}(D)^{+}$ with compact kernel such that
$\varphi(H(\mathbb{Z}))$ is commensurable with $\Gamma$. Moreover, with a mild
condition on the ranks, the pair $(H,\varphi)$ is uniquely determined up to a
unique isomorphism (see \ref{h37a}). Let $\bar{h}\colon\mathbb{S}{}\rightarrow
H^{\mathrm{ad}}$ be the homomorphism whose projection into a compact factor of
$H^{\mathrm{ad}}$ is trivial and is such that $\varphi(\bar{h}(z))$ fixes $o$
and acts on $T_{o}(D)$ as multiplication by $z/\bar{z}$ (cf. (\ref{hq52}),
p.~\pageref{hq52}).

\begin{theorem}
\label{h12}There exists a faithful family of abelian varieties on $D(\Gamma)$
having a fibre of CM-type if and only if $(H,\bar{h})$ admits a symplectic
representation (\ref{h95}).
\end{theorem}

\begin{proof}
Let $f\colon A\rightarrow D(\Gamma)$ be a faithful family of abelian varieties
on $D(\Gamma)$, and let $(\mathsf{V},F)$ be the variation of Hodge structures
$R_{1}f_{\ast}\mathbb{\mathbb{Q}{}}$. Choose a trivialization $\pi^{\ast
}\mathsf{V}\approx V_{D}$, and let $G\subset\GL_{V}$ be the generic
Mumford-Tate group (see \ref{generic}). As in (\S 8), we get a commutative
diagram%
\begin{equation}
\begin{tikzpicture} \matrix(m)[matrix of math nodes, row sep=3em, column sep=2.5em, text height=1.5ex, text depth=0.25ex] {H\\ (H^{\text{ad}},\bar{h})&(G,h)&\GL_V\\}; \path[->,font=\scriptsize,>=angle 90] (m-1-1) edge (m-2-1) (m-1-1) edge node[auto] {} (m-2-2) (m-2-2) edge (m-2-1); \path[right hook->,font=\scriptsize,>=angle 90] (m-2-2) edge node[auto]{$\rho$} (m-2-3); \end{tikzpicture} \label{hq18}%
\end{equation}
in which the image of $H\rightarrow G$ is $G^{\mathrm{der}}$. Because the
family is faithful, the map $H\rightarrow G^{\mathrm{der}}$ is an isogeny, and
so $(H,\bar{h})$ admits a symplectic representation.

Conversely, a symplectic representation of $(H,\bar{h})$ defines a variation
of Hodge structures (\ref{h92a}), which arises from a family of abelian
varieties by Theorem \ref{h69b} (Riemann's theorem in families).
\end{proof}

\begin{theorem}
\label{h12a}There exists a faithful family of abelian motives on $D(\Gamma)$
having a fibre of CM-type if and only if $(H,\bar{h})$ has finite covering by
a pair of Hodge type.
\end{theorem}

\begin{proof}
The proof is essentially the same as that of Theorem \ref{h12}. The points are
the determination of the Mumford-Tate groups of abelian motives in
(\ref{h11b}) and Theorem \ref{h80b}, which replaces Riemann's theorem in families.
\end{proof}

\subsection{Shimura varieties}

In the above, we have always considered connected varieties. As Deligne
(1971)\nocite{deligne1971} observed, it is often more convenient to consider
nonconnected varieties.

\begin{definition}
\label{h13}A \emph{Shimura datum}%
\index{Shimura datum}
is a pair $(G,X)$ consisting of a reductive group $G$ over $\mathbb{Q}{}$ and
a $G(\mathbb{R}{})^{+}$-conjugacy class of homomorphisms $\mathbb{S}%
{}\rightarrow G_{\mathbb{R}{}}$ satisfying (SV1,2,3).\footnote{In the usual
definition, $X$ is taken to be a $G(\mathbb{R}{})$-conjugacy class. For our
purposes, it is convenient to choose a connected component of $X$.}
\end{definition}

\begin{example}
\label{h13a}Let $(V,\psi)$ be a symplectic space over $\mathbb{Q}{}$. The
group $G(\psi)$ of symplectic similitudes together with the space $X(\psi)$ of
all complex structures $J$ on $V_{\mathbb{R}{}}$ such that $(x,y)\mapsto
\psi(x,Jy)$ is positive definite is a Shimura datum.
\end{example}

Let $(G,X)$ be a Shimura datum. The map $h\mapsto\bar{h}\overset
{\textup{{\tiny def}}}{=}\ad\circ h$ identifies $X$ with a $G^{\mathrm{ad}%
}(\mathbb{R}{})^{+}$-conjugacy class of homomorphisms $\bar{h}\colon
\mathbb{S}{}/\mathbb{G}_{m}\rightarrow G_{\mathbb{R}{}}^{\mathrm{ad}}$
satisfying (SV1,2,3). Thus $X$ is a hermitian symmetric domain (\ref{h21},
\ref{h90a}). More canonically, the set $X$ has a unique structure of a complex
manifold such that, for every representation $\rho_{\mathbb{R}{}}\colon
G_{\mathbb{R}{}}\rightarrow\GL_{V}$, $(V_{X},\rho\circ h)_{h\in X}$ is a
holomorphic family of Hodge structures. For this complex structure,
$(V_{X},\rho\circ h)_{h\in X}$ is a variation of Hodge structures, and so $X$
is a hermitian symmetric domain.

The Shimura variety attached to $(G,X)$ and the choice of a compact open
subgroup $K$ of $G(\mathbb{A}{}_{f})$ is\footnote{This agrees with the usual
definition because of \cite{milneSVI}, 5.11.}%
\[
\Sh_{K}(G,X)=G(\mathbb{Q})_{+}\backslash X\times G(\mathbb{A}{}_{f})/K
\]
where $G(\mathbb{Q}{})_{+}=G(\mathbb{Q}{})\cap G(\mathbb{R}{})^{+}$. In this
quotient, $G(\mathbb{Q}{})_{+}$ acts on both $X$ (by conjugation) and
$G(\mathbb{A}{}_{f})$, and $K$ acts on $G(\mathbb{A}{}_{f})$. Let
$\mathcal{C}{}$ be a set of representatives for the (finite) double coset
space $G(\mathbb{Q}{})_{+}\backslash G(\mathbb{A}{}_{f})/K$; then%
\[
G(\mathbb{Q})_{+}\backslash X\times G(\mathbb{A}{}_{f})/K\simeq\bigsqcup
\nolimits_{g\in\mathcal{C}{}}\Gamma_{g}\backslash X,\quad\Gamma_{g}%
=gKg^{-1}\cap G(\mathbb{Q}{})_{+}.
\]
Because $\Gamma_{g}$ is a congruence subgroup of $G(\mathbb{Q}{})$, its image
in $G^{\mathrm{ad}}(\mathbb{Q}{})$ is arithmetic (\ref{h33}), and so
$\Sh_{K}(G,X)$ is a finite disjoint union of connected Shimura varieties. It
therefore has a unique structure of an algebraic variety. As $K$ varies, these
varieties form a projective system.

We make this more explicit in the case that $G^{\mathrm{der}}$ is simply
connected. Let $\nu\colon G\rightarrow T$ be the quotient of $G$ by
$G^{\mathrm{der}}$, and let $Z$ be the centre of $G$. Then $\nu$ defines an
isogeny $Z\rightarrow T$, and we let%
\begin{align*}
T(\mathbb{R}{})^{\dagger}  &  =\im(Z(\mathbb{R}{})\rightarrow T(\mathbb{R}%
{})),\\
T(\mathbb{Q}{})^{\dagger}  &  =T(\mathbb{Q}{})\cap T(\mathbb{R}{})^{\dagger}.
\end{align*}
The set $T(\mathbb{Q}{})^{\dagger}\backslash T(\mathbb{A}{}_{f})/\nu(K)$ is
finite and discrete. For $K$ sufficiently small, the map%
\begin{equation}
\lbrack x,a]\mapsto\lbrack\nu(a)]\colon G(\mathbb{Q}{})\backslash X\times
G(\mathbb{A}{}_{f})/K\rightarrow T(\mathbb{Q}{})^{\dagger}\backslash
T(\mathbb{A}{}_{f})/\nu(K) \label{hq19}%
\end{equation}
is surjective, and each fibre is isomorphic to $\Gamma\backslash X$ for some
congruence subgroup $\Gamma$ of $G^{\mathrm{der}}(\mathbb{Q}{})$. For the
fibre over $[1]$, the congruence subgroup $\Gamma$ is contained in $K\cap
G^{\mathrm{der}}(\mathbb{Q}{})$, and equals it if $Z(G^{\mathrm{der}})$
satisfies the Hasse principal for $H^{1}$, for example, if $G^{\mathrm{der}}$
has no factors of type $A$.

\begin{example}
\label{h13c}Let $G=\GL_{2}$. Then
\begin{align*}
(G\overset{\nu}{\longrightarrow}T)  &  =(\GL_{2}\overset{\det}{\longrightarrow
}\mathbb{G}_{m})\\
(Z\overset{\nu}{\longrightarrow}T)  &  =(\mathbb{G}_{m}\overset{2}%
{\longrightarrow}\mathbb{G}_{m}),
\end{align*}
and therefore
\[
T(\mathbb{Q}{})^{\dagger}\backslash T(\mathbb{A}{}_{f})/\nu(K)=\mathbb{Q}%
{}^{>0}\backslash\mathbb{A}{}_{f}^{\times}/\det(K).
\]
Note that $\mathbb{A}{}_{f}^{\times}=\mathbb{Q}{}^{>0}\cdot\mathbb{\hat{Z}%
}^{\times}{}$ (direct product) where $\mathbb{\hat{Z}}=\varprojlim
_{n}\mathbb{Z}{}/n\mathbb{Z}{}\simeq\prod\nolimits_{\ell}\mathbb{Z}{}_{\ell}%
{}$. For
\[
K=K(N)\overset{\textup{{\tiny def}}}{=}\{a\in\mathbb{\hat{Z}}^{\times}\mid
a\equiv1,\bmod N\},
\]
we find that
\[
T(\mathbb{Q}{})^{\dagger}\backslash T(\mathbb{A}{}_{f})/\nu(K)\simeq
(\mathbb{Z}{}/N\mathbb{Z}{})^{\times}.
\]

\end{example}

\begin{definition}
\label{h13d}A Shimura datum $(G,X)$ is of \emph{Hodge type}%
\index{Hodge type}
if there exists an injective homomorphism $G\rightarrow G(\psi)$ sending $X$
into $X(\psi)$ for some symplectic pair $(V,\psi)$ over $\mathbb{Q}{}$.
\end{definition}

\begin{definition}
\label{h13e}A Shimura datum $(G,X)$ is of \emph{abelian type}%
\index{abelian type}
if, for one (hence all) $h\in X$, the pair $(G^{\mathrm{der}},\ad\circ h)$ is
a quotient of a product of pairs satisfying (a) or (b) of (\ref{h98}).
\end{definition}

A Shimura variety $\Sh(G,X)$ is said to be of Hodge or abelian type if $(G,X)$ is.

\begin{nt}
See \cite{milneSVI}, \S 5, for proofs of the statements in this subsection.
For the structure of the Shimura variety when $G^{\mathrm{der}}$ is not simply
connected, see \cite{deligne1979}, 2.1.16.
\end{nt}

\subsection{Shimura varieties as moduli varieties}

Throughout this subsection, $(G,X)$ is a Shimura datum such that

\begin{enumerate}
\item $w_{X}$ is defined over $\mathbb{Q}{}$ and the connected centre of $G$
is split by a CM-field, and

\item there exists a homomorphism $\nu\colon G\rightarrow\mathbb{G}_{m}%
\simeq\GL_{\mathbb{Q}{}(1)}$ such that $\nu\circ w_{X}=-2$.
\end{enumerate}

Fix a faithful representation $\rho\colon G\rightarrow\GL_{V}$. Assume that
there exists a pairing $t_{0}\colon V\times V\rightarrow\mathbb{Q}{}(m)$ such
that (i) $gt_{0}=\nu(g)^{m}t_{0}$ for all $g\in G$ and (ii) $t_{0}$ is a
polarization of $(V,\rho_{\mathbb{R}{}}\circ h)$ for all $h\in X$. Then there
exist homomorphisms $t_{i}\colon V^{\otimes r_{i}}\rightarrow\mathbb{Q}%
{}(\frac{mr_{i}}{2})$, $1\leq i\leq n$, such that $G$ is the subgroup of
$\GL_{V}$ whose elements fix $t_{0},t_{1},\ldots,t_{n}$. When $(G,X)$ is of
Hodge type, we choose $\rho$ to be a symplectic representation.

Let $K$ be a compact open subgroup of $G(\mathbb{A}_{f})$. Define
${}\mathcal{\mathcal{H}}_{K}(\mathbb{C}{})$ to be the set of triples
\[
(W,(s_{i})_{0\leq i\leq n},\eta K)
\]
in which

\begin{itemize}
\item $W=(W,h_{W})$ is a rational Hodge structure,

\item each $s_{i}$ is a morphism of Hodge structures $W^{\otimes r_{i}%
}\rightarrow\mathbb{Q}{}(\frac{mr_{i}}{2})$ and $s_{0}$ is a polarization of
$W$,

\item $\eta K$ is a $K$-orbit of $\mathbb{A}{}_{f}$-linear isomorphisms
$V_{\mathbb{A}{}_{f}}\rightarrow W_{\mathbb{A}{}_{f}}$ sending each $t_{i}$ to
$s_{i}$,
\end{itemize}

\noindent satisfying the following condition:

\begin{quote}
(*) there exists an isomorphism $\gamma\colon W\rightarrow V$ sending each
$s_{i}$ to $t_{i}$ and $h_{W}$ onto an element of $X$.
\end{quote}

\begin{lemma}
\label{h15}For $(W,\ldots)$ in $\mathcal{H}{}_{K}(\mathbb{C}{})$, choose an
isomorphism $\gamma$ as in (*), let $h$ be the image of $h_{W}$ in $X$, and
let $a\in G(\mathbb{A}{}_{f})$ be the composite $V_{\mathbb{A}{}_{f}}%
\overset{\eta}{\longrightarrow}W_{\mathbb{A}{}_{f}}\overset{\gamma
}{\longrightarrow}V_{\mathbb{A}{}_{f}}$. The class $[h,a]$ of the pair $(h,a)$
in $G(\mathbb{Q}{})_{+}\backslash X\times G(\mathbb{A}{}_{f})/K$ is
independent of all choices, and the map%
\[
(W,\ldots)\mapsto\lbrack h,a]\colon\mathcal{H}{}_{K}(\mathbb{C}{}%
)\rightarrow\Sh_{K}(G,X)(\mathbb{C}{})
\]
is surjective with fibres equal to the isomorphism classes$.$
\end{lemma}

\begin{proof}
The proof involves only routine checking.
\end{proof}

For a smooth algebraic variety $S$ over $\mathbb{C}{}$, let $\mathcal{F}{}%
_{K}(S)$ be the set of isomorphism classes of triples $(A,(s_{i})_{0\leq i\leq
n},\eta K)$ in which

\begin{itemize}
\item $A$ is a family of abelian motives over $S$,

\item each $s_{i}$ is a morphism of abelian motives $A^{\otimes r_{i}%
}\rightarrow\mathbb{Q}{}(\frac{mr_{i}}{2})$, and

\item $\eta K$ is a $K$-orbit of $\mathbb{A}{}_{f}$-linear isomorphisms
$V_{S}\rightarrow\omega_{f}(A/S)$ sending each $t_{i}$ to $s_{i}%
$,\footnote{The isomorphism $\eta$ is defined only on the universal covering
space of $S^{\text{an}}$, but the family $\eta K$ is stable under $\pi
_{1}(S,o)$, and so is \textquotedblleft defined\textquotedblright\ on $S$.}
\end{itemize}

\noindent satisfying the following condition:

\begin{quote}
(**) for each $s\in S(\mathbb{C}{})$, the Betti realization of $(A,(s_{i}%
),\eta K)_{s}$ lies in $\mathcal{H}{}_{K}(\mathbb{C}{})$.
\end{quote}

\noindent With the obvious notion of pullback, $\mathcal{F}{}_{K}$ becomes a
functor from smooth complex algebraic varieties to sets. There is a
well-defined injective map $\mathcal{F}{}_{K}(\mathbb{C})\rightarrow
\mathcal{H}{}_{K}(\mathbb{C}{})/\!\!\approx$, which is surjective when $(G,X)$
is of abelian type. Hence, in this case, we get an isomorphism $\alpha
\colon\mathcal{F}{}_{K}(\mathbb{C}{})\rightarrow\Sh_{K}(\mathbb{C}{})$.

\begin{theorem}
\label{h16}Assume that $(G,X)$ is of abelian type. The map $\alpha$ realizes
$\Sh_{K}$ as a coarse moduli variety for $\mathcal{F}{}_{K}$, and even a fine
moduli variety when $Z(\mathbb{Q)}$ is discrete in $Z(\mathbb{R}{})$ (here
$Z=Z(G)$).
\end{theorem}

\begin{proof}
To say that $(\Sh_{K},\alpha)$ is coarse moduli variety means the following:

\begin{enumerate}
\item for any smooth algebraic variety $S$ over $\mathbb{C}{}$, and $\xi
\in\mathcal{F}{}(S)$, the map $s\mapsto\alpha(\xi_{s})\colon S(\mathbb{C}%
{})\rightarrow\Sh_{K}(\mathbb{C}{})$ is regular;

\item $(\Sh_{K},\alpha)$ is universal among pairs satisfying (a) .
\end{enumerate}

To prove (a), we use that $\xi$ defines a variation of Hodge structures on $S$
(see p.~\pageref{h80b}). Now the universal property of hermitian symmetric
domains (\ref{h67}) shows that the map $s\mapsto\alpha(\xi_{s})$ is
holomorphic (on the universal covering space, and hence on the variety), and
Borel's theorem \ref{h43} shows that it is regular.

Next assume that $Z(\mathbb{Q}{})$ is discrete in $Z(\mathbb{R}{})$. Then the
representation $\rho$ defines a variation of Hodge structures on $\Sh_{K}$
itself (not just its universal covering space), which arises from a family of
abelian motives. This family is universal, and so $\Sh_{K}$ is a fine moduli variety.

We now prove (b). Let $S^{\prime}$ be a smooth algebraic variety over
$\mathbb{C}{}$ and let $\alpha^{\prime}\colon\mathcal{F}_{K}(\mathbb{C}%
{})\rightarrow S^{\prime}(\mathbb{C}{})$ be a map with the following property:
for any smooth algebraic variety $S$ over $\mathbb{C}{}$ and $\xi
\in\mathcal{F}{}(S)$, the map $s\mapsto\alpha^{\prime}(\xi_{s})\colon
S(\mathbb{C}{})\rightarrow S^{\prime}(\mathbb{C}{})$ is regular. We have to
show that the map $s\mapsto\alpha^{\prime}\alpha^{-1}(s)\colon\Sh_{K}%
(\mathbb{C}{})\rightarrow S^{\prime}(\mathbb{C}{})$ is regular. When
$Z(\mathbb{Q}{})$ is discrete in $Z(\mathbb{R}{})$, the map is that defined by
$\alpha^{\prime}$ and the universal family of abelian motives on $\Sh_{K}$,
and so it is regular by definition. In the general case, we let $G^{\prime}$
be the smallest algebraic subgroup of $G$ such that $h(\mathbb{S}{})\subset
G_{\mathbb{R}{}}^{\prime}$ for all $h\in X$. Then $(G^{\prime},X)$ is a
Shimura datum (cf. \ref{h65a}), which now is such that $Z(\mathbb{Q}{})$ is
discrete in $Z(\mathbb{R}{})$; moreover, $\Sh_{K\cap G^{\prime}(\mathbb{A}%
{}_{f})}(G^{\prime},X)$ consists of a certain number of connected components
of $\Sh_{K}(G,X)$. As the map is regular on $\Sh_{K\cap G^{\prime}%
(\mathbb{A}{}_{f})}(G^{\prime},X)$, and $\Sh_{K}(G,X)$ is a union of
translates of $\Sh_{K\cap G^{\prime}(\mathbb{A}{}_{f})}(G^{\prime},X)$, this
shows that the map is regular on $\Sh_{K}(G,X)$.
\end{proof}

\subsubsection{Remarks}

\begin{plain}
\label{h16a}When $(G,X)$ is of Hodge type in Theorem \ref{h16}, the Shimura
variety is a moduli variety for abelian \textit{varieties} with additional
structure. In this case, the moduli problem can be defined for all schemes
algebraic over $\mathbb{C}{}$ (not necessarily smooth), and Mumford's theorem
can be used to prove that the Shimura variety is moduli variety for the
expanded functor.
\end{plain}

\begin{plain}
\label{h16b}It is possible to describe the structure $\eta K$ by passing only
to a finite covering, rather than the full universal covering. This means that
it can be described purely algebraically.
\end{plain}

\begin{plain}
\label{h16c}For certain compact open groups $K$, the structure $\eta K$ can be
interpreted as a level-$N$ structure in the usual sense.
\end{plain}

\begin{plain}
\label{h16d}Consider a pair $(H,\bar{h})$ having a finite covering of Hodge
type. Then there exists a Shimura datum $(G,X)$ of abelian type such that
$(G^{\mathrm{der}},\ad\circ h)=(H,\bar{h})$ for some $h\in X$. The choice of a
faithful representation $\rho$ for $G$ gives a realization of the connected
Shimura variety defined by any (sufficiently small) congruence subgroup of
$H(\mathbb{Q}{})$ as a fine moduli variety for abelian motives with additional
structure. For example, when $H$ is simply connected, there is a map
$\mathcal{H}{}_{K}(\mathbb{C}{})\rightarrow T(\mathbb{Q}{})^{\dagger
}\backslash T(\mathbb{A}{}_{f})/\nu(K)$ (see (\ref{hq19}), p.~\pageref{hq19}),
and the moduli problem is obtained from $\mathcal{F}{}_{K}$ by replacing
$\mathcal{H}{}_{K}(\mathbb{C}{})$ with its fibre over $[1]$. Note that the
realization involves many choices.
\end{plain}

\begin{plain}
\label{h16e}For each Shimura variety, there is a well-defined number field
$E(G,X)$, called the reflex field. When the Shimura variety is a moduli
variety, it is possible choose the moduli problem so that it is defined over
$E(G,X)$. Then an elementary descent argument shows that the Shimura variety
itself has a model over $E(G,X)$. A priori, it may appear that this model
depends on the choice of the moduli problem. However, the theory of complex
multiplication shows that the model satisfies a certain reciprocity law at the
special points, which characterize it.
\end{plain}

\begin{plain}
\label{h16f}The (unique) model of a Shimura variety over the reflex field
$E(G,X)$ satisfying (Shimura's) reciprocity law at the special points is
called the \emph{canonical model}%
\index{canonical model}%
. As we have just noted, when a Shimura variety can be realized as a moduli
variety, it has a canonical model. More generally, when the associated
connected Shimura variety is a moduli variety, then $\Sh(G,X)$ has a canonical
model (\cite{shimura1970}, \cite{deligne1979}). Otherwise, the Shimura variety
can be embedded in a larger Shimura variety that contains many Shimura
subvarieties of type $A_{1}$, and this can be used to prove that the Shimura
variety has a canonical model (\cite{milne1983}).
\end{plain}

\begin{nt}
For more details on this subsection, see \cite{milne1994}.
\end{nt}

\bibliographystyle{cbe}
\bibliography{D:/Current/svh}

\pagestyle{empty} \clearpage

\includegraphics{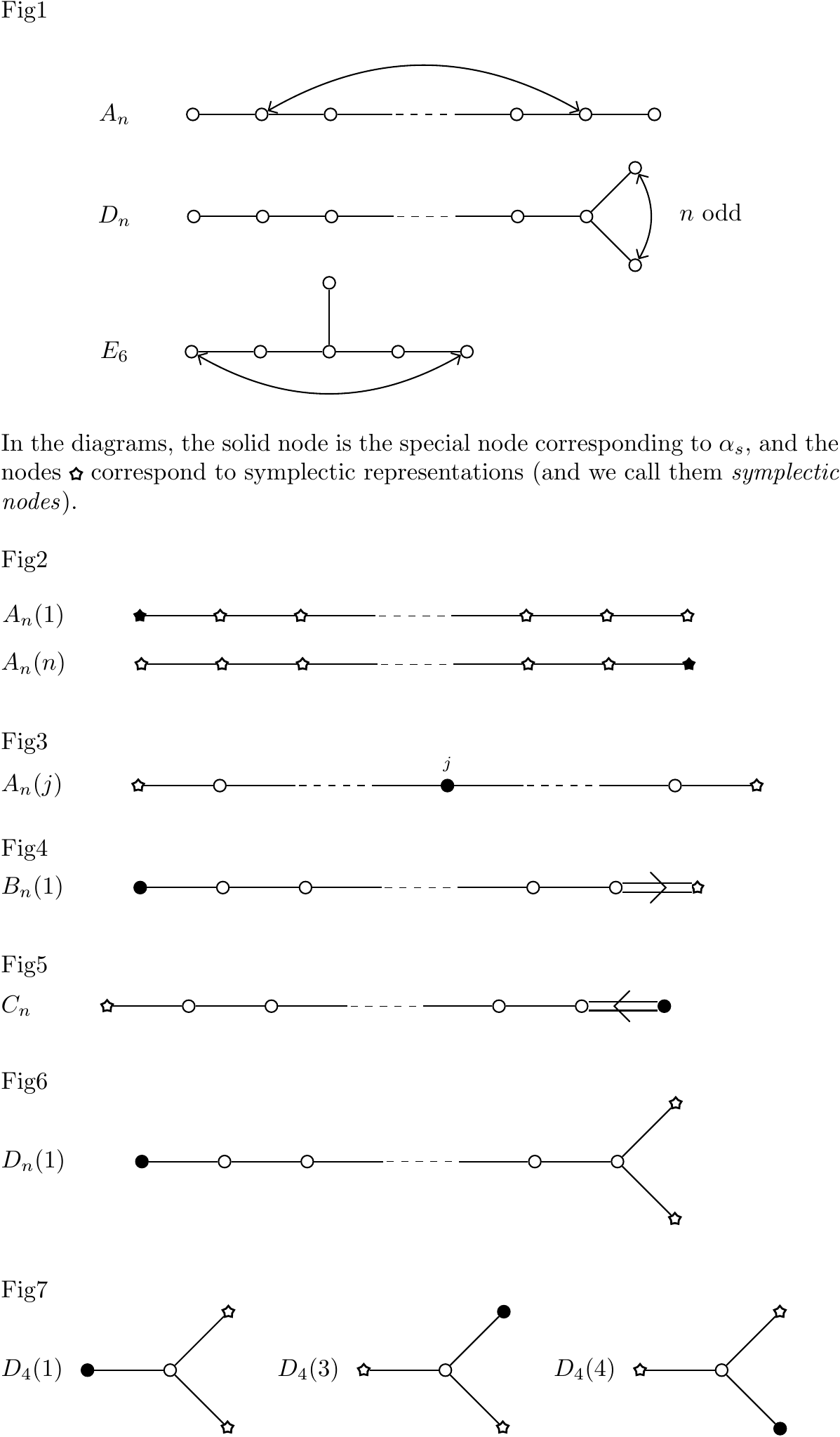}

\clearpage \includegraphics{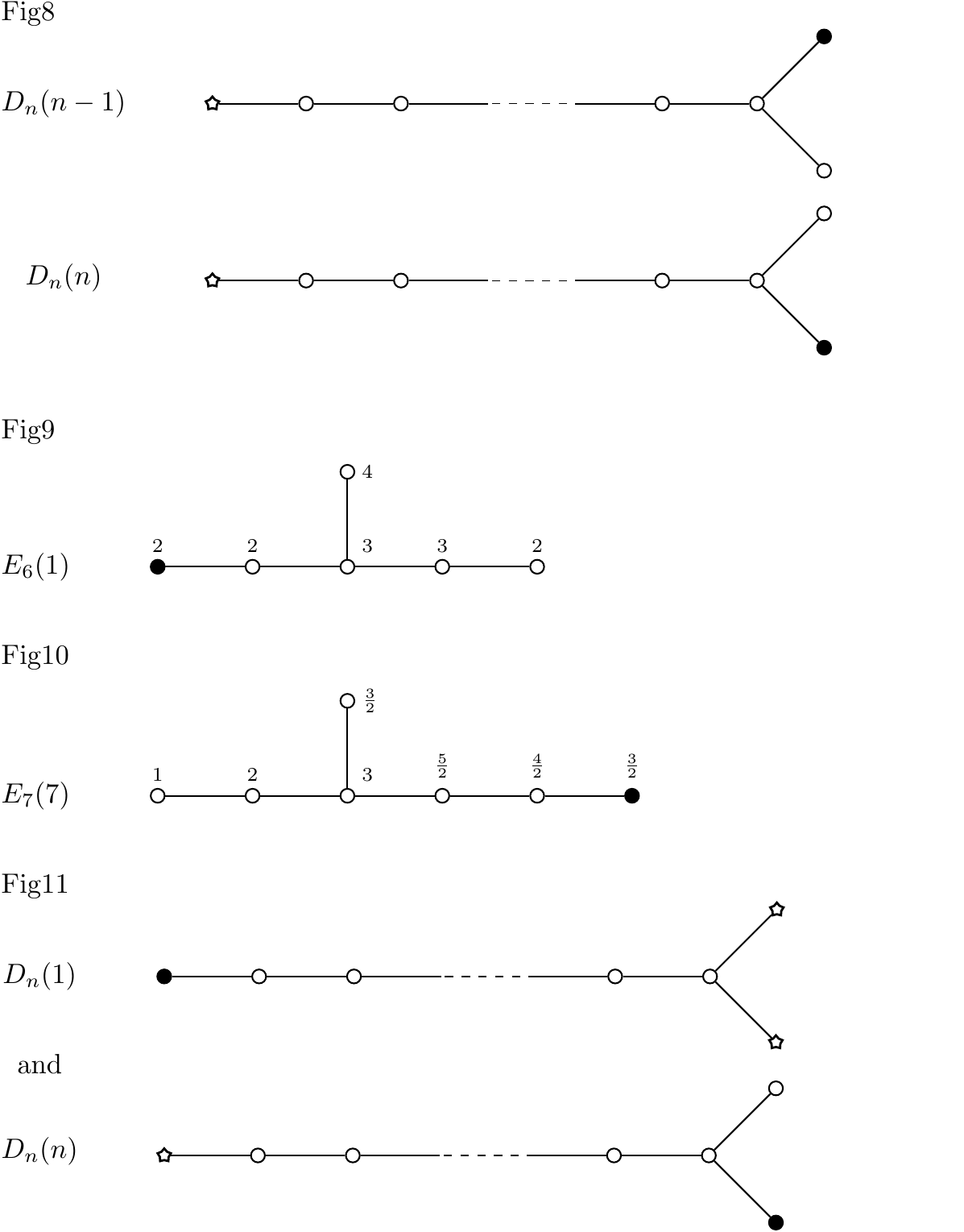}
\end{document}

%% file: defna.tex
\usepackage{mathtools}

\usetikzlibrary{matrix,arrows}
\usetikzlibrary{arrows,positioning,decorations.pathmorphing}

\titleformat{\section}
{\titlerule
\vspace{.8ex}%
\normalfont\itshape}
{\thesection.}{.5em}{}

\titleformat*{\section}{\LARGE\bfseries}
\titleformat*{\subsection}{\Large\itshape}
\titleformat*{\subsubsection}{\scshape}
\titleformat*{\paragraph}{\itshape}

\setitemize[1]{label=$\diamond$\hspace{0.07in}}
\setlist{itemsep=0em,topsep=0em,parsep=0em}

\let\cite\citealt

\newcommand{\btable}{\begin{table}}
\newcommand{\etable}{\end{table}}
\newcommand\bquote{\begin{quote}}
\newcommand\equote{\end{quote}}
\newcommand\bcomment{}
\newcommand\edocument{\end{document}}
\newcommand\bsmall{\begin{small}}
\newcommand\esmall{\end{small}}
\newcommand\bfootnotesize{\begin{footnotesize}}
\newcommand\efootnotesize{\end{footnotesize}}
\newcommand{\tstyle}{\textstyle}

\newcommand{\textnf}{\textnormal}

\def\1{{1\mkern-7mu1}}

\DeclareMathOperator{\ad}{ad}
\DeclareMathOperator{\Ad}{Ad}

\DeclareMathOperator{\Aut}{Aut}

\DeclareMathOperator{\End}{End}

\DeclareMathOperator{\Gal}{Gal}
\DeclareMathOperator{\GL}{GL}
\DeclareMathOperator{\Gr}{Gr} 

\DeclareMathOperator{\Hom}{Hom}
\DeclareMathOperator{\id}{id}

\DeclareMathOperator{\im}{Im}  

\DeclareMathOperator{\inn}{inn}

\DeclareMathOperator{\Ker}{Ker}
\DeclareMathOperator{\Lie}{Lie}

\DeclareMathOperator{\MT}{MT}

\DeclareMathOperator{\PGL}{PGL}

\DeclareMathOperator{\Proj}{\mathrm{Proj}}

\DeclareMathOperator{\PSL}{PSL}

\DeclareMathOperator{\rank}{rank}

\DeclareMathOperator{\Sh}{Sh}

\DeclareMathOperator{\SL}{SL}
\DeclareMathOperator{\SO}{SO}

\DeclareMathOperator{\Spec}{Spec}

\DeclareMathOperator{\SU}{SU}

\DeclareMathOperator{\Tr}{Tr}

\newcommand{\Hdg}{\mathsf{Hdg}}

\newcommand{\Mot}{\mathsf{Mot}}

\newcommand{\Rep}{\mathsf{Rep}}

\newcommand{\Vc}{\mathsf{Vec}}

